%% file: gs_soc_spin1.tex
\documentclass[preprint,3p]{elsarticle}
\usepackage{stmaryrd}
\SetSymbolFont{stmry}{bold}{U}{stmry}{m}{n}
\usepackage{amsmath,bbm,mathbbol}
\usepackage{amssymb,amsfonts,amsthm}
\usepackage{mathrsfs,mathtools}
\usepackage{tabularx,lmodern}
\usepackage{color,soul,enumitem}
\usepackage{epsfig,epstopdf}
\usepackage{graphicx}
\usepackage{algorithm}
\usepackage{algorithmicx}
\usepackage{algpseudocode}
\usepackage{flushend}
\usepackage{verbatim}
\usepackage{caption}
\usepackage{subfigure}
\usepackage{arydshln}
\usepackage[level]{datetime}
\usepackage{float}
\usepackage{booktabs}
\usepackage{natbib}
\usepackage{relsize}
\usepackage{ulem}
\usepackage{setspace}
\usepackage{appendix}
\usepackage{bm}

\usepackage{hyperref} 
\hypersetup{
  	colorlinks=true,
  	citecolor=blue,
  	linkcolor=red,
  	urlcolor=green}

\renewcommand{\Re}{\text{Re}}
\renewcommand{\Im}{\text{Im}}

\allowdisplaybreaks[4]

\makeatletter

\newcommand{\Rmnum}[1]{\expandafter\@slowromancap\romannumeral #1@}
\makeatother

\makeatletter
\@addtoreset{equation}{section}
\makeatother

\begin{document}

\include{command}

\title{Mathematical and numerical study on the ground states of rotating spin-orbit coupled spin-1 Bose-Einstein condensates}

\author[muc]{Jing Wang}
\ead{wjing@muc.edu.cn}

\author[tju]{Wei Yang}
\ead{weiyanghn@tju.edu.cn}

\author[hunnu]{Yongjun Yuan}
\ead{yyj1983@hunnu.edu.cn}

\author[tju]{Yong Zhang\corref{cor1}}
\ead{Zhang\_Yong@tju.edu.cn}

\address[muc]{School of Science, Minzu University of China, Beijing, 100081, China}
\address[tju]{Center for Applied Mathematics and KL-AAGDM, Tianjin University, Tianjin, 300072, China}
\address[hunnu]{LCSM (MOE), School of Mathematics and Statistics, Hunan Normal University, Changsha, Hunan 410081, China}
\cortext[cor1]{Corresponding author.}

\begin{abstract}
	In this article, we study mathematically and numerically the ground states
	of three-component rotating spin-orbit coupled (SOC) spin-1 Bose-Einstein condensates modeled by the coupled Gross–Pitaevskii equations (CGPEs). 
	Firstly, we rigorously prove existence result of the ground state and derive some analytical properties, including the virial identity and negativity of SOC energy. 
	Secondly, we propose an efficient and accurate preconditioned nonlinear conjugate gradient (PCG) algorithm to compute the ground states. 
	We truncate the whole space into a bounded rectangular domain and readily apply the Fourier spectral method to approximate the wave function. 
	The PCG method is successfully adapted with appropriate modifications to the adaptive step size control strategy for the one-parameter energy minimization problem and to the choice of preconditioners, achieving great performance in terms of accuracy and efficiency.
	Lastly, we carry out extensive numerical experiments to verify the existence and property results of the ground states, confirm the spatial spectral accuracy by traversing the most commonly-used initial guesses for each component thanks to its great efficiency, which is also attributed to a utilization of cascadic multigrid and discrete Fast Fourier Transform (FFT). 
	Moreover, we investigate the effects of local interaction, rotation and spin-orbit coupling and external trapping potential on the ground state, and unveil some interesting physical phenomena, such as giant vortex and U-shape vortex line.
\end{abstract}

\begin{keyword}
	spin-1 Bose-Einstein condensate, spin-orbit coupling, ground state, Fourier spectral method, preconditioned nonlinear conjugate gradient method
\end{keyword}
  	
\maketitle
\tableofcontents

 \section{Introduction}
 Since its first experimental realizations in 1995 with a dilute alkali-metal gas \cite{bec-davis}, the Bose-Einstein condensate (BEC) has stimulated great excitement in the physical community and attracted significant interest in atomic and molecular \cite{molecular} as well as condensed matter physics. 
 Particularly, under a rotational frame, the rotating BECs \cite{vortexlattice,vortexnucleat} were created, which possess quantized vortices and are related to superfluid properties. 
 In fact, bulk superfluids are distinguished from normal fluids by their ability to support dissipationless flow. 
 In earlier BEC experiments, the atoms were confined in a magnetic trap \cite{bec-davis}, and thus the spin degrees of freedom were frozen.
 Nevertheless, recently developed optical trapping techniques have enabled to release the spin internal degrees of freedom, opening up a new research area of quantum many-body systems named spinor BEC \cite{spinor}, noticing that it is called a spin-$\mathcal{F}$ $(\mathcal{F} \in \mathbb{N})$ BEC if its spin quantum number is $\mathcal{F}$. 
 Meanwhile, the spin-orbit coupling (SOC) is actually the interaction between the spin and motion of a particle. 
 It has been found as the key to understand many fundamental physical phenomena in quantum systems, such as the topological insulator \cite{Hall}, and quantum spin Hall effects \cite{topology}. 
 Recently, the SOC was successfully induced in neutral atomic Bose-Einstein condensates by dressing two atomic spin states with a pair of lasers \cite{SOBEC}. 
 These experiments triggered a strong activity in the area of spin-orbit-coupled cold atoms and a number of exciting phenomena have been discovered. 
 Actually, the rotating SOC spin-1 BECs under different trapping potentials have been investigated \cite{rotSOCvortex2,rotSOCvortex}.
 
 Within the mean-field regime, at temperature lower than the critical temperature $T_c$, a spin-$\mathcal{F}$ condensate can be mathematically described by coupled Gross-Pitaevskii equations (CGPEs) which consist of $2\mathcal{F} + 1$ equations \cite{spinor,spinorRep}. 
 The macroscopic wave function $\Psi(\mathbf{x},t) = (\psi_1(\mathbf{x},t),\psi_0(\mathbf{x},t),\psi_{-1}(\mathbf{x},t))^\top$ of the rotating SOC spin-1 BEC system obeys the following d-dimensional ($d = 2$ or $3$) dimensionless CGPEs \cite{dy-soc-spin1,rotSOC}:
 \begin{align}
 	\label{GPE1}
 	i\partial_t \psi_1(\mathbf{x},t) &= \left[\mathcal{H}_0 + c_1 \left(\rho_0 + \rho_1 - \rho_{-1}\right)\right] \psi_1 + c_1 \bar{\psi}_{-1}\psi_0^2 - \gamma L_0\psi_0, \\ 
 	\label{GPE2}
 	i\partial_t \psi_0(\mathbf{x},t) &= \left[\mathcal{H}_0 + c_1 \left(\rho_1 + \rho_{-1}\right)\right] \psi_0 + 2 c_1 \psi_{-1} \bar{\psi}_0 \psi_1 - \gamma \left(L_0\psi_{-1}+L_1\psi_1\right), \\
 	\label{GPE3}
 	i\partial_t \psi_{-1}(\mathbf{x},t) &= \left[\mathcal{H}_0 + c_1 \left(\rho_0 + \rho_{-1} - \rho_1\right)\right] \psi_{-1} + c_1 \psi_0^2 \bar{\psi}_{1} - \gamma L_1\psi_0,
 \end{align}
 where $t$ is the time variable, $\mathbf{x} = (x, y)^\top$ if $d = 2$ and $\mathbf{x} = (x, y, z)^\top$ if $d = 3$. 
 The single-particle Hamiltonian is given as $\mathcal{H}_0 = -\frac{1}{2}\nabla^2 + V(\mathbf{x}) + c_0 \rho - \Omega L_z$ where $\nabla$ is the gradient operator with respect to spatial variables. 
 The mean-field interaction strength $c_0$ is positive for repulsive interaction and negative for attractive interaction.
 The spin-exchange interaction strength $c_1$ is positive for antiferromagnetic interaction and negative for ferromagnetic interaction. 
 $\Omega$ and $\gamma$ denote the rotation speed and spin-orbit coupling strength, respectively. 
 The density of the $\ell$-th $(\ell=1,0,-1)$ component is $\rho_{\ell}=|\psi_{\ell}|^2$, and $\rho=\rho_{-1}+\rho_0+\rho_1$ is the total particle density. 
 $L_z=-i\left(x \partial_y - y \partial_x \right)$ is the rotation operator, i.e., the $z$-component of the angular momentum. 
 $V(\mathbf{x})$ is the real-valued external trapping potential whose shape is determined by the type of system under investigation. 
 For instance, if a harmonic potential is considered, it takes the following form:
 \begin{equation}\label{harmonic}
 	V\left(\mathbf{x}\right) = \frac{1}{2}\left\{\begin{aligned}
 		&\gamma_x^2 x^2 + \gamma_y^2 y^2, \quad &d = 2,\\
 		&\gamma_x^2 x^2 + \gamma_y^2 y^2 + \gamma_z^2 z^2, &d=3,
 	\end{aligned}
 	\right.
 \end{equation}
 where $\gamma_x$, $\gamma_y$ and $\gamma_z$ are the dimensionless trapping frequencies in $x$-, $y$- and $z$-directions, respectively. 
 $L_0 = i \partial_x + \partial_y$ and $L_1 = i \partial_x - \partial_y$ are the spin-orbit coupling operators. 
 $\overline{\psi}_{\ell}$ is the complex conjugate of $\psi_{\ell}$.
 
 Introducing the spin-1 matrices $\mathbf{f} = (f_x,f_y,f_z)^\top$ as
 \[
 f_x=\frac{1}{\sqrt{2}}\begin{pmatrix}
 	0 & 1 & 0\\
 	1 & 0 & 1\\
 	0 & 1 & 0 
 \end{pmatrix}, \qquad
 f_y=\frac{i}{\sqrt{2}}\begin{pmatrix}
 	0 & -1 & 0\\ 
 	1 & 0  & -1\\
 	0 & 1  & 0 
 \end{pmatrix}, \qquad 
 f_z=\begin{pmatrix}
 	1 & 0 & 0\\
 	0 & 0 & 0\\
 	0 & 0 & -1 
 \end{pmatrix},
 \]
 the spin vector, defined as $\mathbf{F} = (F_x,F_y,F_z)^\top := (\Psi^\dagger f_x \Psi,\Psi^\dagger f_y \Psi,\Psi^\dagger f_z \Psi)^\top \in \mathbb{R}^3$ with $\Psi^\dagger$ being the conjugate transpose of complex-valued vector $\Psi:=(\psi_1,\psi_0,\psi_{-1})^\top$, is given explicitly as 
 \begin{align*}
 	&F_x=\frac{1}{\sqrt{2}}\left[\bar{\psi}_1\psi_0 + \bar{\psi}_0\left(\psi_1 + \psi_{-1}\right) + \bar{\psi}_{-1}\psi_0\right], \\
 	&F_y=\frac{i}{\sqrt{2}}\left[-\bar{\psi}_1\psi_0+\bar{\psi}_0\left(\psi_1-\psi_{-1}\right)+\bar{\psi}_{-1} \psi_0\right], \\
 	&F_z=\left|\psi_1\right|^2-\left|\psi_{-1}\right|^2.
 \end{align*}
 The spin vector $\mathbf{F}$ is real-valued thanks to the fact that all spin-1 matrices are Hermitian.
 The CGPEs \eqref{GPE1}-\eqref{GPE3} can be written in a compact form as
 \begin{equation}
 	i \partial_t \Psi = \left[\mathcal{H}_0 + c_1 \mathbf{F} \cdot \mathbf{f} - \gamma ~\mathbf{S} \right] \Psi,
 \end{equation}
 where
 \begin{eqnarray*}
 	\textbf{F}\cdot \textbf{f} =
 	\begin{psmallmatrix}
 		F_z 					& \frac{1}{\sqrt{2}}F_{-} & 0    \\
 		\frac{1}{\sqrt{2}}F_{+}	& 0						  & \frac{1}{\sqrt{2}}F_{-} \\
 		0						& \frac{1}{\sqrt{2}}F_{+} & -F_z
 	\end{psmallmatrix}, 
 	\qquad
 	\mathbf{S} =
 	\begin{pmatrix}
 		0   & L_0 & 0   \\
 		L_1 & 0   & L_0 \\
 		0   & L_1 & 0
 	\end{pmatrix},\quad \mbox{with} \quad F_{\pm} = F_x \pm i F_y.
 \end{eqnarray*}
 
 Two important invariants of \eqref{GPE1}-\eqref{GPE3} are the normalization (or mass) of the wave function
 \begin{equation}\label{mass}
 	\mathcal{N}\left(\Psi\right) = \Vert \Psi\left(\cdot, t\right) \Vert^2 := \int_{\mathbb{R}^d} \sum\limits_{\ell=-1}^1 \vert \psi_{\ell}(\mathbf{x}, t) \vert^2 \ \mathrm{d} \mathbf{x} \equiv 1, \quad t \geq 0,
 \end{equation}
 and the energy per particle
 \begin{align}\label{energy}
 	\mathcal{E}\left(\Psi(\cdot,t)\right) &:= \int_{\mathbb{R}^d}\left[\sum_{\ell=-1}^1 \bar{\psi}_{\ell}
 	\left(-\frac{1}{2}\nabla^2 + V(\mathbf{x}) - \Omega L_z \right)\psi_{\ell} + \frac{c_0}{2} \rho^2 + \frac{c_1}{2}\vert \mathbf{F} \vert^2 - \gamma \Psi^\dagger \mathbf{S} \Psi \right] \mathrm{d} \mathbf{x} \nonumber \\
 	&=\mathcal{E}\left(\Psi(\cdot,0)\right). 
 \end{align}
 The ground state of the rotating SOC spin-1 BEC \eqref{GPE1}-\eqref{GPE3} is defined as the minimizer of the following nonconvex minimization problem
 \begin{equation}\label{groundstate}
 	\Phi_{\rm g} = \arg\min_{\Phi \in \mathcal{M}} \mathcal{E}(\Phi),
 \end{equation}
 with
 \begin{equation}\label{sphere}
 	\mathcal{M} = \left\{\Phi = (\phi_1, \phi_0,\phi_{-1})^\top \mid \mathcal{N}(\Phi) = 1, ~ \mathcal{E}(\Phi)<\infty\right\}.
 \end{equation}

 There have been extensive mathematical and numerical studies on the ground states of single-component BECs. The gradient flow with discrete normalization (GFDN) method, also known as the imaginary time evolution method \cite{ITE,ITE_jcp}, was first proposed in \cite{GFDN_2004} to compute the ground states. 
 Later, some improved gradient flow methods were proposed, e.g., Danaila et al. proposed a new Sobolev gradient method for direct minimization of the Gross-Pitaevskii energy with rotation \cite{DK2010},
 and Liu et al. designed the gradient flow with Lagrange multiplier method \cite{GFLM}. 
 These gradient-flow type methods are effective in computing the ground states of non-rotating and slow-rotating BECs, 
 but do not perform ideally for fast-rotating BECs as efficiency is concerned.
 Recently, based on optimization theory on Riemannian manifolds \cite{matrixmanifold}, Danaila et al. combined Riemannian optimization theory and Sobolev gradients to propose the Riemannian steepest descent method and conjugate gradient method \cite{RCG}. Notably, Antoine et al. proposed a preconditioned nonlinear conjugate gradient (PCG) method \cite{pcg,BEC2HPC,action,pRCG} on the unit sphere manifold. 
 This method exhibits extremely high efficiency, even for rapidly rotating BECs, and has been extended to rotating dipolar BECs \cite{pcg-dbec}. 
 In addition, Chen et al. introduced an efficient second-order damped flow method \cite{secondflow} 
 by adding a second-order time derivative to the gradient flow, which performs exceptionally well in computing the ground states of single-component rotating BECs.

 For the spinor BECs, there has been some literature on either mathematical or numerical studies,
 and we refer the readers to \cite{spinorBEC} for a comprehensive review. However, so far as we know, there is little literature devoted to the ground state of spinor BECs with SOC term.
 Along the numerical front, Bao et al. applied the GFDN method and introduced a third normalization condition \cite{spinorBEC,spin1charact,spin1_GFDN}. 
 Recently, Cai et al. extended the gradient flow with Lagrange multiplier to general spinor cases \cite{spinGFLM}. 
 For SOC spin-1 BEC without rotation, the ground states have been numerically studied in \cite{FORTRESS,SOCspin1}.

 Recently, there is a growing interest in studying the properties of rotating SOC spin-1 BECs \cite{rotsoc3,rotsoc2,rotSOC,rotsoc4}. 
 So far as we know, there are quite few mathematical and numerical studies on the ground states of spin-1 BECs 
 with both the rotation and SOC terms.
 Two major challenges arise therein. 
 First, the presence of rotation term narrows the so-called fundamental gap, i.e., the energy difference between the ground state and the first excited state,
 especially for fast rotating BECs, and it causes a significant efficiency degradation or even accuracy saturation for gradient-flow type methods. 
 Second, for both fast-rotating and strong SOC BEC, the ground states exhibit many fine structures, 
 which shall require very high accuracy and a reasonable efficiency at the same time in numerical simulations. 
 To this end, we adapt the PCG method to develop an efficient and accurate numerical scheme for computing the ground states of rotating SOC spin-1 BECs.

 \
 
 The primary objectives of this paper are threefold:

 \begin{itemize}[itemsep=0.5em, topsep=0.3em, parsep=0.3em]
 	\item[(1)] to prove the existence and derive some properties of the ground states, including the virial identity and negativity of SOC energy.
 	\item[(2)] to adapt PCG method by integrating the Fourier spectral method, optimizing an adaptive step size control strategy, developing good preconditioning techniques and utilizing the cascadic multigrid to improve the efficiency.
 	\item[(3)] to study the impacts of local interactions, rotation, spin-orbit coupling and external trapping potential on ground states in 2D and 3D, and to explore novel ground state patterns and structures.
 \end{itemize}
 
 The rest of the paper is organized as follows. In Section \ref{exist-prop}, we establish existence results and some properties of the ground states. 
 In Section \ref{numericalmethod}, we present an accurate and efficient preconditioned nonlinear conjugate gradient method, its multigrid accelerated version, as well as three different preconditioners and common stopping criteria. 
 In Section \ref{result}, we test the performance of our algorithm and apply it to investigate the ground state patterns and vortex structures in 2D and 3D under different physical parameters. Finally, we draw some conclusions in Section \ref{conclude}.

 \section{Existence and properties of the ground states}\label{exist-prop}
 In this section, we establish the existence of ground state of rotating SOC spin-1 BECs and show some properties of the ground states.
 
 \subsection{Existence of the ground states}\label{existence}
 The ground state of rotating SOC spin-1 BEC is defined as the minimizer such that 
 \begin{equation}\label{minimizer}
 	\mathcal{E}(\Phi_{\rm g})=\min_{\Phi\in\mathcal{M} } \mathcal{E}(\Phi)
 \end{equation}
 with 
 \[
 \mathcal{M}=\{(\phi_1,\phi_0,\phi_{-1})^\top|\phi_{j}\in H^1(\mathbb{R}^d)\cap L^2_{V}(\mathbb{R}^d),~j=-1,0,1, ~\mathcal{N}(\Phi)=1\},
 \]
 where $L^2_{V}(\mathbb{R}^d) := \{ \phi(\mathbf{x}) | \int_{\mathbb{R}^d} V(\mathbf{x}) \vert \phi(\mathbf{x}) \vert^2 \mathrm{d} \mathbf{x} < \infty \}$ is a weighted $L^2$ space.
 We define $ \mathbf{x}^{\perp} =(-y,x)$ for $d=2$ and $\mathbf{x}^\perp=(-y,x,0) $ for $d=3$, 
 then $L_z\phi=-i\mathbf{x}^{\perp}\cdot\nabla\phi$.
 Now we split the energy functional into three parts
 \[
 \mathcal{E}(\Phi)=\mathcal{E}_1(\Phi)+\mathcal{E}_2(\Phi)+\mathcal{E}_3(\Phi),
 \]
 where
 \begin{align}
 	\mathcal{E}_1(\Phi) &= \int_{\mathbb{R}^d} \sum_{j=-1}^{1} \left(\frac{1}{2}\vert \nabla \phi_j \vert^2 + V(\mathbf{x}) \vert \phi_j \vert^2 +i\Omega \bar{\phi}_j \mathbf{x}^{\perp}\cdot\nabla \phi_j \right) \mathrm{d} \mathbf{x}, \nonumber\\	
 	\mathcal{E}_2(\Phi)	&= \int_{\mathbb{R}^d} \frac{c_0}{2} \rho^2 + \frac{c_1}{2} \left[2(\Re((\bar\phi_1 + \bar\phi_{-1})\phi_0))^2 + 2(\Im(\phi_0(\bar\phi_{-1} - \bar\phi_1)))^2 + (|\phi_1|^2 - |\phi_{-1}|^2)^2\right] \mathrm{d} \mathbf{x}, \nonumber\\	
 	\mathcal{E}_3(\Phi)	&= -\gamma\int_{\mathbb{R}^d}\left(\bar{\phi}_1 L_0\phi_0+\bar{\phi}_0(L_0\phi_{-1}+L_1\phi_1)+\bar{\phi}_{-1}L_1\phi_0\right ) \mathrm{d}\mathbf{x}, \nonumber
 \end{align}	
 and $\Re (f)$ and $\Im (f)$ represent the real part and the imaginary part of $f$, respectively. Since $L_z$ is self-adjoint, we rewrite $\mathcal{E}_1$ as
 \begin{align}\label{E1}
 	\mathcal{E}_1(\Phi)&=\frac{1}{2} \int_{\mathbb{R}^d}\sum_{j=-1}^{1} \left(\vert \nabla \phi_j - i\Omega\mathbf{x}^{\perp}\phi_j\vert^2 + (2V(\mathbf{x})-\Omega^2|\mathbf{x}^\perp| ^2)\vert \phi_j \vert^2\right) \ \mathrm{d} \mathbf{x}.
 \end{align}
 For $d=2$, we denote by
 \[
 C_b=\inf_{0\neq \phi\in H^1(\mathbb{R}^2)}\frac{\|\nabla\phi\|_{2}^2\|\phi\|_2^2}{\|\phi\|_4^4}.
 \]
 We also have 
 \begin{equation}\label{GN}
 	\begin{aligned}
 		&\int_{\mathbb{R}^2}(|\phi_1|^2+|\phi_0|^2+|\phi_{-1}|^2)^2 \ \mathrm d\mathbf{x} \\
 		&\leq \frac{1}{C_b}\int_{\mathbb{R}^2}(|\nabla\phi_1|^2+|\nabla\phi_0|^2+|\nabla\phi_{-1}|^2)\ \mathrm d\mathbf{x}\int_{\mathbb{R}^2}(|\phi_1|^2+|\phi_0|^2+|\phi_{-1}|^2)\ \mathrm d\mathbf{x}
 	\end{aligned}
 \end{equation}
 by Lemma A.2 in \cite{gs-two-exist}.
 
 \begin{thm}
 	Suppose \begin{equation}\label{omega}
 		|\Omega| < \min\{\gamma_x,\gamma_y\}.
 	\end{equation}
 	Then there exists a ground state of \eqref{minimizer} if
 	
 	$(1)~ d=2$
 	\begin{equation}\label{beta1}
 		c_0>-C_b, \quad c_0+c_1>-C_b,
 	\end{equation}
 	
 	$(2)~ d=3$
 	\begin{equation}\label{beta2}
 		c_0,c_1\geq 0 ~~\text{or}~~c_1\leq 0,~c_0+c_1\geq 0.
 	\end{equation} 
 \end{thm}
 \begin{proof}
 	(1) $d=2$.	Since the case $c_1\geq 0$ is an easy adaption of the following argument, we will focus on the case $c_1<0$. There are positive numbers $a^*$ and $a$ such that
 	\[
 	1>a^*>\frac{a}{C_b},\quad c_0+c_1>-a>-C_b.
 	\]
 	Then by Diamagnetic inequality \cite{analysis}, we have
 	\[
 	\int_{\mathbb{R}^2} \vert \nabla \phi_j - i\Omega\mathbf{x}^{\perp}\phi_j\vert^2\ \mathrm{d} \mathbf{x}\geq \int_{\mathbb{R}^2} \vert \nabla |\phi_j|\vert^2 \ \mathrm{d} \mathbf{x}, ~~j=-1,0,1.
 	\]
 	Therefore 
 	\begin{align}
 		\mathcal{E}_1(\Phi)&=\frac{1}{2} \int_{\mathbb{R}^2}\sum_{j=-1}^{1} \vert \nabla \phi_j - i\Omega\mathbf{x}^{\perp}\phi_j\vert^2 \ \mathrm{d} \mathbf{x} +\int_{\mathbb R^2}\sum_{j=-1}^{1} \left(V(\mathbf{x})-\frac{\Omega^2}{2}|\mathbf{x}^\perp| ^2\right)\vert \phi_j \vert^2  \ \mathrm{d} \mathbf{x} \nonumber\\
 		&\geq\frac{1-a^*}{2} \int_{\mathbb R^2}\sum_{j=-1}^{1} \vert \nabla \phi_j - i\Omega\mathbf{x}^{\perp}\phi_j\vert^2 \ \mathrm{d} \mathbf{x} + \int_{\mathbb R^2}\sum_{j=-1}^{1} \left(V(\mathbf{x})-\frac{\Omega^2}{2}|\mathbf{x}^\perp| ^2\right)\vert \phi_j \vert^2  \ \mathrm{d} \mathbf{x}\nonumber\\
 		&+ \frac{a^*}{2}\int_{\mathbb R^2}(|\nabla|\phi_{-1}||^2+|\nabla|\phi_0||^2+|\nabla|\phi_1||^2) \ \mathrm{d} \mathbf{x}.
 		\nonumber
 	\end{align}
 	By Gagliardo-Nirenberg inequality \eqref{GN}, for any $\Phi\in \mathcal{M} $ we have
 	\begin{align}
 		\mathcal{E}_2(\Phi)	&\geq \frac{c_0}{2}\int_{\mathbb R^2}(|\phi_{-1}|^2+|\phi_0|^2+|\phi_1|^2)^2\ \mathrm d\mathbf{x} \nonumber\\
 		&+\frac{c_1}{2}\int_{\mathbb R^2}(2|\phi_0|^2(|\phi_{-1}|^2+|\phi_1|^2)+(|\phi_{1}|^2-|\phi_{-1}|^2)^2)\ \mathrm d\mathbf{x}\nonumber\\
 		&=\frac{c_0+c_1}{2} \int_{\mathbb R^2}(|\phi_{-1}|^2+|\phi_0|^2+|\phi_1|^2)^2 \ \mathrm d\mathbf{x} - \frac{c_1}{2} \int_{\mathbb R^2}|\phi_0|^4 \ \mathrm d\mathbf{x}- 2 c_1 \int_{\mathbb R^2}|\phi_{-1}|^2|\phi_{1}|^2 \ \mathrm d\mathbf{x} \nonumber\\
 		&\geq -\frac{a}{2} \int_{\mathbb R^2} (|\phi_{-1}|^2 + |\phi_0|^2 + |\phi_1|^2)^2 \mathrm d\mathbf{x}\geq -\frac{a}{2C_b} \int_{\mathbb R^2} (|\nabla|\phi_{-1}||^2 + |\nabla|\phi_0||^2 + |\nabla|\phi_1||^2)\ \mathrm{d} \mathbf{x}. \nonumber
 	\end{align}
 	Adding them together yields
 	\[
 	\mathcal{E}_1(\Phi)+\mathcal{E}_2(\Phi)\geq \frac{1-a^*}{2} \int_{\mathbb R^2} \sum_{j=-1}^{1} \vert \nabla \phi_j - i\Omega\mathbf{x}^{\perp}\phi_j\vert^2 \ \mathrm{d} \mathbf{x} + \int_{\mathbb R^2} \sum_{j=-1}^{1} \left(V(\mathbf{x}) - \frac{\Omega^2}{2}|\mathbf{x}^\perp|^2\right) \vert \phi_j \vert^2  \ \mathrm{d} \mathbf{x}.
 	\]
 	Therefore, by Cauchy inequality, there exists a positive constant $c$ depending only on $\Omega,k, c_0,c_1$ such that 
 	\[
 	\mathcal{E}_1(\Phi)+	\mathcal{E}_2(\Phi)\geq c(\|\nabla\Phi\|_2^2+\|\Phi\|_{L_V^2}^2).
 	\]
 	
 	For $\mathcal{E}_3(\Phi)$, we note 
 	\[
 	|\bar{\phi}_1 L_0\phi_0+\bar{\phi}_0(L_0\phi_{-1}+L_1\phi_1)+\bar{\phi}_{-1}L_1\phi_0 |\leq 2|\nabla \phi_0|(|\phi_1|+|\phi_{-1}|)+2|\phi_0|(|\nabla \phi_1|+|\nabla \phi_{-1}|).
 	\]
 	Hence, by Cauchy inequality, we obtain $\forall ~\varepsilon>0$
 	\[
 	|\mathcal{E}_3(\Phi)|\leq \varepsilon \|\nabla\Phi\|_2^2+C_{\varepsilon, |\gamma|} \|\Phi\|_2^2.
 	\]
 	Letting $\varepsilon=\frac{c}{2}$ we have 
 	\[
 	\mathcal{E}(\Phi)\geq \frac{c}{2}(\|\nabla\Phi\|_2^2+\|\Phi\|_{L_V^2}^2)- C \|\Phi\|_2^2.
 	\]
 	Therefore $\mathcal{E}$ is bounded below on $\mathcal{M}$ and we define
 	\[
 	m:=\inf_{\Phi\in\mathcal{M} }	\mathcal{E}(\Phi)>-\infty.
 	\]
 	
 	Now taking a minimizing sequence $\{\Phi^n=(\phi_{-1}^n,\phi_0^n,\phi_1^n)^\top\}_{n=1}^\infty\subset \mathcal{M}$ satisfying $\mathcal{E}(\Phi^n)\rightarrow m$, then there exists a constant $C>0$
 	\[
 	\|\Phi^n\|_2^2+	\|\nabla\Phi^n\|_2^2+\|\Phi^n\|_{L_V^2}^2\leq C, ~n\geq 1.
 	\]
 	That is to say, $\{\Phi^n\}_{n=1}^\infty$ is bounded in $H^1(\mathbb R^2)\cap L_V^2(\mathbb R^2)$. Therefore up to a subsequence we assume 
 	\[
 	\phi^n_j\rightharpoonup\phi^\infty_j ~\text{in}~ H^1(\mathbb R^2)\cap L_V^2(\mathbb R^2), ~	\phi^n_j\rightarrow \phi^\infty_j, ~\text{a.e.}~ j=-1,0,1.
 	\] 
 	Using \eqref{E1} and Fatou's lemma, we have 
 	\[
 	\liminf_{n\rightarrow\infty}\mathcal{E}_1(\Phi^n)\geq \mathcal{E}_1(\Phi^\infty).
 	\]
 	To deal with $\mathcal{E}_2$, by the well-known compactness theorem \cite{gs-two-exist}
 	\[
 	H^1(\mathbb R^2)\cap L_V^2(\mathbb R^2)\hookrightarrow L^p(\mathbb R^2), ~~2\leq p< \infty,
 	\]
 	we have 
 	\[
 	\phi^n_j\rightarrow\phi^\infty_j~\text{in}~ L^4(\mathbb R^2)\cap L^2(\mathbb R^2), ~~ j=-1,0,1.
 	\]
 	Hence 
 	\[
 	\lim_{n\rightarrow\infty}\mathcal{E}_2(\Phi^n)=\mathcal{E}_2(\Phi^\infty),\quad  \|\Phi^\infty\|_2^2=1.
 	\]
 	For $\mathcal{E}_3$, we note 
 	\[
 	L_i\phi_j^n\rightharpoonup	L_i\phi_j^\infty, \phi_j^n\rightarrow \phi_j^\infty~\text{in}~L^2(\mathbb R^2), ~i=0,1, ~j=-1,0,1.
 	\]
 	Then 
 	\[
 	\lim_{n\rightarrow\infty}\mathcal{E}_3(\Phi^n)=\mathcal{E}_3(\Phi^\infty).
 	\]
 	Hence we have 
 	\[
 	\liminf_{n\rightarrow\infty}\mathcal{E}(\Phi^n)\geq \mathcal{E}(\Phi^\infty),~\|\Phi^\infty\|_2^2=1.
 	\]
 	That is to say $\Phi^\infty=(\phi_{-1}^\infty,\phi_0^\infty,\phi_1^\infty)^\top$ is a minimizer.
 	
 	(2) $d=3$. 	By Cauchy inequality, there exists a positive constant $c$ depending only on $\Omega,k $ such that 
 	\[
 	\mathcal{E}_1(\Phi)\geq c(\|\nabla\Phi\|_2^2+\|\Phi\|_{L_V^2}^2).
 	\]
 	Under the condition \eqref{beta2}, it always holds $\mathcal{E}_2(\Phi)\geq 0$. The rest of the proof is the same as the case $d=2$ except using 
 	the well-known compactness theorem
 	\[
 	H^1(\mathbb R^3)\cap L_V^2(\mathbb R^3)\hookrightarrow L^p(\mathbb R^3), ~~2\leq p< 6.
 	\]
 	For the sake of brevity, we omit details.
 \end{proof}

 \subsection{Properties of the ground states}\label{properties}
 The Euler-Lagrange equations associated to minimization problem \eqref{groundstate} read as follows
 \begin{align}
 	\label{eigen1}
 	&\mu \phi_1(\mathbf{x}) = \left[\mathcal{H}_0 + c_1 \left(\rho_0 + \rho_1 - \rho_{-1}\right) \right] \phi_1 + c_1 \bar{\phi}_{-1}\phi_0^2 - \gamma L_0\phi_0:= H_1(\Phi), \\ 
 	\label{eigen2}
 	&\mu \phi_0(\mathbf{x}) = \left[\mathcal{H}_0 + c_1 \left(\rho_1 + \rho_{-1}\right) \right] \phi_0 + 2 c_1  \phi_{-1} \bar{\phi}_0 \phi_1 - \gamma \left(L_0\phi_{-1} + L_1\phi_1\right):= H_0(\Phi),\\
 	\label{eigen3}
 	&\mu \phi_{-1}(\mathbf{x}) = \left[\mathcal{H}_0 + c_1 \left(\rho_0 + \rho_{-1} - \rho_1\right) \right] \phi_{-1} + c_1 \phi_0^2 \bar{\phi}_{1} - \gamma L_1\phi_0 := H_{-1}(\Phi),
 \end{align}
 where $\Phi = \left(\phi_1, \phi_0, \phi_{-1} \right)^\top \in \mathcal{M}$, and $\mu$ is the Lagrange multiplier 
 (or chemical potential) of the CGPEs \eqref{GPE1}-\eqref{GPE3}, which can be computed as 
 \begin{equation}\label{muphi}
 	\mu(\Phi) = \int_{\mathbb{R}^d} \sum_{\ell=-1}^1 H_{\ell}(\Phi)\bar{\phi}_{\ell}\ \mathrm{d} \mathbf{x}.
 \end{equation}
 In fact, \eqref{eigen1}-\eqref{eigen3} is a nonlinear eigenvalue problem subject to a normalization constraint \eqref{mass}. It can also be derived from the CGPEs \eqref{GPE1}-\eqref{GPE3} by plugging in $\psi_{\ell}(\mathbf{x},t)=e^{-i\mu t}\phi_{\ell}(\mathbf{x})$, 
 and Eqn.~\eqref{eigen1}-\eqref{eigen3} is often referred to as the time-independent CGPEs. 
 In physics literature, eigenfunctions with higher energy than the ground state are called excited states. 
 
 Note $\mathcal{E}\left(\kappa ~ \Phi_{\rm g}\right) = \mathcal{E}\left(\Phi_{\rm g}\right)$ for all $\kappa \in \mathbb{C}$ with $\vert \kappa \vert = 1$.
 If $\Phi^1$ and $\Phi^2$ are two ground states differing only by a phase factor, that is, $\Phi^1 = \kappa~ \Phi^2$, 
 then the phase factor $\kappa$ is determined by
 \begin{equation}\label{phase-factor}
 	\kappa = \langle \Phi^1, \Phi^2 \rangle / \langle \Phi^2, \Phi^2 \rangle,
 \end{equation}
 where the inner product of two vector functions is defined as
 \begin{equation*}
 	\langle \mathbf{v},\mathbf{u} \rangle := \int_{\mathbb{R}^d} \sum \limits_{\ell=-1}^1 v_{\ell}(\mathbf{x}) \overline{u_{\ell}(\mathbf{x})} \ \mathrm{d} \mathbf{x}, ~\text{with} ~ \mathbf{v} = (v_1,v_0,v_{-1})^\top, ~\mathbf{u} = (u_1,u_0,u_{-1})^\top.
 \end{equation*}
 
 First, similar to \cite{spinorBEC,SOCspin1}, we have the following virial theorem for the ground states.
 \begin{thm}[\textbf{Virial identity}]\label{thm-virial}
 	Suppose that $\Phi_{\rm g} = \left(\phi_1^{\rm g}, \phi_0^{\rm g}, \phi_{-1}^{\rm g}\right)^\top \in \mathcal{M}$ is the ground state solution of the rotating spin-orbit coupled spin-1 BEC \eqref{GPE1}-\eqref{GPE3}. 
 	When $V(\mathbf{x})$ is chosen as a harmonic potential, we have
 	\begin{equation}\label{virial}
 		2 \mathcal{E}_{\rm kin}\left(\Phi_{\rm g}\right) - 2 \mathcal{E}_{\rm pot}\left(\Phi_{\rm g}\right) + d\ \mathcal{E}_{\rm spin}\left(\Phi_{\rm g}\right) 
 		+ \mathcal{E}_{\rm soc}\left(\Phi_{\rm g}\right) = 0,\qquad d=2,3,
 	\end{equation}
 	where
 	\begin{align*}
 		\mathcal{E}_{\rm kin}\left(\Phi\right) = \frac{1}{2}\int_{\mathbb{R}^d}\sum\limits_{\ell=-1}^1 |\nabla \phi_{\ell}|^2 \ \mathrm{d} \mathbf{x},& \qquad
 		\mathcal{E}_{\rm pot}\left(\Phi\right) = 
 		\int_{\mathbb{R}^d} \sum\limits_{\ell=-1}^1 V|\phi_{\ell}|^2 \ \mathrm{d} \mathbf{x}, \\
 		\mathcal{E}_{\rm spin}\left(\Phi\right) = \int_{\mathbb{R}^d} \frac{c_0}{2} \rho^2 + \frac{c_1}{2}\vert \mathbf{F} \vert^2 \ \mathrm{d} \mathbf{x},& \qquad
 		\mathcal{E}_{\rm rot}\left(\Phi\right) = -\Omega \int_{\mathbb{R}^d} \sum\limits_{\ell=-1}^1 \bar{\phi}_{\ell} L_z\phi_{\ell} \ \mathrm{d} \mathbf{x}, \\
 		\mathcal{E}_{\rm soc}\left(\Phi\right) = -\gamma& \int_{\mathbb{R}^d} \Phi^\dagger \mathbf{S} \Phi \ \mathrm{d} \mathbf{x}.
 	\end{align*}
 \end{thm}
 
 \begin{proof}
 	Let us introduce the functions 
 	$\phi_{\ell}^{\varepsilon} \left(\mathbf{x}\right) = \left(1 + \varepsilon\right)^{d/2} \phi_{\ell}^{\rm g} \left( \left(1+\varepsilon \right) \mathbf{x}\right)$. The energy $\mathcal{I}(\varepsilon) := \mathcal{E}\left(\phi_1^{\varepsilon},\phi_0^{\varepsilon},\phi_{-1}^{\varepsilon}\right)$ can be computed as 
 	\begin{equation}\label{E_epsilon}
 		\begin{aligned}
 			\mathcal{I}(\varepsilon) = (1+\varepsilon)^2 \mathcal{E}_{\rm kin}\left(\Phi_{\rm g}\right) +& \frac{1}{(1+\varepsilon)^2} \mathcal{E}_{\rm pot}\left(\Phi_{\rm g}\right) + (1+\varepsilon)^d \mathcal{E}_{\rm spin}\left(\Phi_{\rm g}\right) \\
 			+& \mathcal{E}_{\rm rot}\left(\Phi_{\rm g}\right) + (1+\varepsilon)\mathcal{E}_{\rm soc}\left(\Phi_{\rm g}\right).
 		\end{aligned}
 	\end{equation}
 	Since $\Phi_{\rm g} = \left(\phi_1^{\rm g}, \phi_0^{\rm g}, \phi_{-1}^{\rm g}\right)^\top$ is a ground state solution,
 	the energy function $\mathcal{I}(\varepsilon)$ attains a minimum at $\varepsilon = 0$, and it implies directly that $\frac{\mathrm{d} \mathcal{I}(\varepsilon)}{\mathrm{d} \varepsilon} |_{\varepsilon=0} =0.$
 	Differentiating \eqref{E_epsilon} and evaluating the resulted function at $\varepsilon =0$, we obtain the virial identity \eqref{virial}.
 \end{proof}
 
 Then, we derive the negativity of SOC energy for the ground states of the BEC under a radially symmetric trapping potential.
 
 \begin{thm}[\textbf{Negativity of the SOC energy}]
 	Suppose $\Phi_{\rm g}\left(\mathbf{x}\right) \in \mathcal{M}$ is the ground state solution of rotating spin-orbit coupled spin-1 BEC \eqref{GPE1}-\eqref{GPE3}. 
 	For radially symmetric trapping potential, i.e., $V(\mathbf{x}) = V(|\mathbf{x}|)$, we have 
 	$$\mathcal{E}_{\rm soc}\left(\Phi_{\rm g}\right) \leq 0.$$
 	Furthermore, if $\mathcal{E}_{\rm soc}\left(\Phi_{\rm g}\right) = 0$, the rotated ground state function, $\Phi^{\theta}(\bx) = \Phi_{\rm g}(\mathcal R(\theta)\bx)$ 
 	where the rotation matrix $\mathcal R(\theta)$ is defined as 
 	\begin{equation*}
 		\mathcal R(\theta) = \begin{pmatrix}
 			\cos\theta & -\sin\theta \\
 			\sin\theta & \cos\theta
 		\end{pmatrix}, ~d = 2,
 		\qquad \mathcal R(\theta) = \begin{psmallmatrix}
 			\cos\theta & -\sin\theta & 0 \\
 			\sin\theta & \cos\theta & 0 \\
 			0 & 0 & 1
 		\end{psmallmatrix}, ~d = 3,
 	\end{equation*}
 	is also a ground state. For the case $\mathcal{E}_{\rm soc}\left(\Phi_{\rm g}\right) < 0$, 
 	any rotated function $\Phi^{\theta}$ with $\theta \in (0,2\pi)$ is not any more a ground state.
 \end{thm}
 
 \begin{proof} 
 	The energy $\mathcal{J}(\theta) := \mathcal{E}\left(\Phi^{\theta}(\bx)\right)$ of the rotated ground state function is given as follows
 	\begin{equation*}
 		\mathcal{J}(\theta) = \mathcal{E}_{\rm kin}\left(\Phi_{\rm g}\right) + \mathcal{E}_{\rm pot}\left(\Phi_{\rm g}\right) + \mathcal{E}_{\rm spin}\left(\Phi_{\rm g}\right) + \mathcal{E}_{\rm rot}\left(\Phi_{\rm g}\right) + \cos\left(\theta\right) \mathcal{E}_{\rm soc}\left(\Phi_{\rm g}\right).
 	\end{equation*}
 	Since $\Phi_{\rm g}$ is a ground state solution, the energy function $\mathcal{J}(\theta)$ attains a minimum at $\theta = 0$. Therefore, we can prove that $\mathcal{E}_{\rm soc}\left(\Phi_{\rm g}\right) \leq 0.$ 
 	
 	If $\mathcal{E}_{\rm soc}\left(\Phi_{\rm g}\right) = 0$, then $\mathcal{J}(\theta) \equiv \mathcal{E}\left(\Phi_{\rm g}\right)$, so for any $\theta \in (0,2\pi)$, $\Phi^{\theta}$ is also a ground state. If $\mathcal{E}_{\rm soc}\left(\Phi_{\rm g}\right) < 0$, then $\mathcal{J}(\theta) > \mathcal{E}\left(\Phi_{\rm g}\right)$ for any $\theta \in (0,2\pi)$, that is, $\Phi^{\theta}$ is not a ground state.
 \end{proof}

 \section{Numerical method}\label{numericalmethod}
 In this section, we propose an efficient and accurate numerical method to compute the ground states of rotating SOC spin-1 BECs. 
 
 \subsection{Preconditioned nonlinear conjugate gradient method}\label{pcgmethod}
 In this subsection, we adapt the PCG method to compute the ground states of rotating SOC spin-1 BECs. 
 Let $\Phi^n=\left(\phi_1^n,\phi_0^n,\phi_{-1}^n\right)^\top$ be the $n$-th $(n \geq 0)$ wave function, $D^n=\left(d_1^n,d_0^n,d_{-1}^n\right)^\top$ and $\tau_n$ 
 be the $n$-th descent direction and step size, respectively. 
 Generally, the subsequent wave function $\Phi^{n+1}$ is updated as
 \begin{equation}\label{update}
 	\Phi^{\left(\#\right)} = \Phi^n + \tau_n D^n,\quad
 	\Phi^{n+1} = \Phi^{\left(\#\right)}/\Vert \Phi^{\left(\#\right)} \Vert. 
 \end{equation}
 Actually, the above formula can be reformulated as a linear combination of $\Phi^n$ and the descent direction $D^n$, that is,
 \begin{equation}\label{formula-pcg}
 	\Phi^{n+1} = \cos\left(\theta_n\right) \Phi^n + \sin\left(\theta_n\right) \hat{P}^n, \quad \hat{P}^n:= P^n/\Vert P^n\Vert, 
 \end{equation}
 where $P^n=\left(p_1^n,p_0^n,p_{-1}^n\right)^\top$ is the projection of $D^n$ onto the tangent subspace of the manifold $\mathcal{M}$ at $\Phi^n$, i.e.,
 \begin{equation}\label{formula-project}
 	P^n = D^n - \Re \langle D^n,\Phi^n \rangle \Phi^n.
 \end{equation}
 It is easy to check that $\Re \langle P^n,\Phi^n \rangle = 0$ and $\Phi^{n+1}\in\mathcal{M}$. Now we take the descent direction $D^n = \left(d_1^n,d_0^n,d_{-1}^n\right)^\top$ 
 as the following nonlinear conjugate gradient direction with preconditioner
 \begin{equation}
 	D^n=\begin{cases}
 		-\mathscr{P} r^n,& \quad n = 0, \\
 		-\mathscr{P} r^n + \beta^{n} P^{n-1},& \quad n \geq 1.
 	\end{cases}
 \end{equation}
 Here, $\mathscr{P}$ is a symmetric positive definite preconditioner, which will be discussed later in subsection \ref{precondition},
 the residual 
 $$r^n = \left(r_1^n,r_0^n,r_{-1}^n\right)^\top:=\left(H_1(\Phi^n) - \mu_n \phi_1^n,H_0(\Phi^n) - \mu_n \phi_0^n,H_{-1}(\Phi^n) - \mu_n \phi_{-1}^n\right)^\top$$ 
 is computed according to \eqref{eigen1}-\eqref{eigen3} and the ``momentum" term $\beta^n$ can be computed using the Polak-Ribi\`ere-Polyak formula:
 \begin{equation*}
 	\beta^n = \max \left\{ \frac{ \Re \langle r^n-r^{n-1},\mathscr{P} r^n \rangle} {\Re \langle r^{n-1},\mathscr{P} r^{n-1} \rangle}, 0 \right\}.
 \end{equation*}
 
 A detailed step-by-step algorithm description is summarized in the following algorithm.
 
 \begin{algorithm}[htbp] 
 	\caption{Preconditioned nonlinear conjugate gradient (PCG) method.}
 	\label{algorithm1}
 	\begin{algorithmic}
 		\State $n = 0$, given initial guess $\Phi^0$
 		\While{not converged}
 		\State $\mu_n = \mu(\Phi^n)$
 		\State $r_{\ell}^n = H_{\ell}(\Phi^n) - \mu_n \phi_{\ell}^n \ (\ell=1,0,-1)$
 		\If{$n = 0$}
 		\State $D^n = -\mathscr{P} r^n$
 		\Else
 		\State $\beta^n = \Re \langle r^n-r^{n-1},\mathscr{P} r^n \rangle / \Re \langle r^{n-1},\mathscr{P} r^{n-1} \rangle$
 		\State $\beta^n = \max\left(\beta^n,0\right), \ D^n = -\mathscr{P} r^n + \beta^n P^{n-1}$
 		\EndIf
 		\State $P^n = D^n - \Re \langle D^n,\Phi^n \rangle \Phi^n$
 		\State $\hat{P}^n = P^n \big/ \Vert P^n \Vert$
 		\State $\theta_n = \mathop{\arg\min}\limits_{\theta > 0} \mathcal{E}\left(\cos\left(\theta\right) \Phi^n + \sin\left(\theta\right) \hat{P}^n \right)$
 		\State $\Phi^{n+1} = \cos(\theta_n) \Phi^n + \sin(\theta_n) \hat{P}^n$
 		\State $n = n + 1$
 		\EndWhile
 	\end{algorithmic}
 \end{algorithm}

 To find the optimal step size $\theta_n$, we can implement the line search to solve the following one-dimensional minimization problem
 \begin{equation}\label{theta}
 	\theta_n = \mathop{\arg\min}\limits_{\theta > 0}  ~ \mathcal{E}\left(\cos\left(\theta\right) \Phi^n + \sin\left(\theta\right) \hat{P}^n\right).
 \end{equation}
 Since $\mathcal{E}(\theta) := \mathcal{E}\left(\cos\left(\theta\right) \Phi^n + \sin\left(\theta\right) \hat{P}^n\right)$ is not a quadratic function,
 it generally requires many evaluations of the energy functional at given variables $\theta$ with the minimization algorithm, for example, Brent's method \cite{Brent1,Brent2},
 which is often time-consuming, especially in high space dimensions.
 Fortunately, we can explicitly rewrite $\mathcal{E}(\theta)$ in terms of $\sin\left(\theta\right)$ and $\cos\left(\theta\right)$, 
 and all associated coefficients can be pre-computed with FFT efficiently. 
 As a result, equation \eqref{theta} is reduced to a standard one-dimensional problem with explicit expressions, and one can easily apply a minimization algorithm to obtain the minimum efficiently.\\
 
 Alternatively, we can obtain a simple approximation of $\theta_n$ 
 by first expanding $\mathcal{E}(\theta)$ up to second order of $\theta$ since it usually gets close to zero as the solution approaches the ground state 
 and then solving the approximation polynomial problem. That is to say, by a direct calculation, we have
 \begin{equation}\label{approeng}
 	\begin{aligned}
 		\mathcal{E}(\theta)
 		\approx \ a_n\, \theta^2 + b_n\, \theta + c_n:= \ f_n \left(\theta\right),\quad  ~\theta \sim 0 
 	\end{aligned}
 \end{equation}
 with coefficients
 \begin{align}
 	a_n =& -\mu(\Phi^n) + \frac{1}{\Vert P^n \Vert^2} \int_{\mathbb{R}^d} \sum_{\ell=-1}^1 H_\ell(P^n) \bar{p}_\ell^n \ \mathrm{d}\mathbf{x} \label{an} \\
 	& + \frac{2}{\Vert P^n \Vert^2} \int_{\mathbb{R}^d} c_0 \left( \Re \left( (\Phi^n)^\dagger P^n \right) \right)^2 + c_1 \sum_\alpha \left[ \Re \left( \tilde{F}_\alpha(\Phi^n, P^n) \right) \right]^2 \ \mathrm{d}\mathbf{x}, \notag \\
 	b_n =& \frac{2}{\| P^n \|} \Re \sum_{\ell=-1}^1 \int_{\mathbb{R}^d} H_\ell(\Phi^n) \bar{p}_\ell^n \, \mathrm{d}\mathbf{x}, \qquad c_n = \mathcal{E}(\Phi^n), \label{bn}
 \end{align}
 where $\tilde{F}_{\alpha}$ $(\alpha = x, y,z)$ are the following conjugate bilinear forms 
 \begin{align*}
 	\tilde{F}_x(\mathbf{v},\mathbf{u}) =& \frac{1}{\sqrt{2}}\left[\bar{u}_1 v_0 + \bar{u}_0 (v_1 + v_{-1}) + \bar{u}_{-1} v_0\right], \\
 	\tilde{F}_y(\mathbf{v},\mathbf{u}) =& \frac{i}{\sqrt{2}}\left[-\bar{u}_1 v_0 + \bar{u}_0 (v_1 - v_{-1}) + \bar{u}_{-1} v_0\right], \\
 	\tilde{F}_z(\mathbf{v},\mathbf{u}) =& \bar{u}_1 v_1 - \bar{u}_{-1} v_{-1}, ~\text{with}~ \mathbf{v} = (v_1,v_0,v_{-1})^\top, \mathbf{u} = (u_1,u_0,u_{-1})^\top.
 \end{align*}
 
 Unlike the general steepest descent algorithm, the preconditioned conjugate gradient direction $P^n$ might not always be a descent direction, that is to say,
 the energy does not necessarily decrease at each iteration even with a very small step size.
 To guarantee the energy decreasing property at each iteration, we adopt the following remedies.
 We first compute the coefficients $a_n$ and $b_n$ of Eqn.~\eqref{approeng}. 
 If $b_n > 0$ or $b_n = 0~\&\&~a_n \geq 0$, indicating that $P^n$ can not decrease $f_n\left(\theta\right)$, we switch to the preconditioned steepest descent (PSD) direction, i.e., set $\beta^n = 0$. With this choice, noticing that $\mathscr{P}$ is positive definite and $\langle \Phi^n, r^n \rangle = 0$, we can easily check that
 \begin{equation*}
 	b_n = \frac{2 \Re \langle P^n, r^n + \mu_n \Phi^n \rangle}{\Vert P^n \Vert} = \frac{2 \Re \langle D^n - \Re \langle D^n, \Phi^n \rangle \Phi^n, r^n + \mu_n \Phi^n \rangle}{\Vert P^n \Vert} 
 	= -\frac{2 \Re \langle \mathscr{P} r^n, r^n \rangle}{\Vert P^n \Vert} < 0,
 \end{equation*} 
 and it automatically falls into other cases. If $a_n > 0~\&~b_n < 0$, $f_n\left(\theta\right)$ attains a minimum within the feasible region and we choose $\theta_n=-\frac{b_n}{2a_n}$ as the iteration step size. For case $a_n<0 ~\&~ b_n <0$, $f_n\left(\theta\right)$ decreases monotonically for positive $\theta$, and we choose a prescribed small positive number $\theta_{\rm trial}$ as the initial step and accept it as the optimal step if the energy decreases, otherwise, we halve it until acceptance.
 The details are summarized in Algorithm \ref{algorithm2}.
 
 \begin{algorithm}[htbp] 
 	\caption{Line search and backtracking algorithm.}\label{algorithm2}
 	\begin{algorithmic}
 		\State Given a trial step size $\theta_{\rm trial}$ and a backtracking factor $\sigma \in (0,1)$.
 		\State \textbf{Step 1:} Compute $a_n$ and $b_n$ according to \eqref{an}-\eqref{bn}.
 		\State \textbf{Step 2:} Determine the step size $\theta_n$:
 		\If{$(b_n > 0)$ or $(b_n = 0 ~\& ~a_n \geq 0)$}
 		\State set $\beta^n = 0$ and go to \textbf{Step 1}
 		\ElsIf{$(a_n > 0~\&~b_n < 0)$}
 		\State $\theta_n = -\frac{b_n}{2 a_n}$
 		\Else
 		\State $\theta_n = \theta_{\rm trial}$
 		\EndIf
 		\State \textbf{Step 3:} If the energy decreases, then stop; otherwise, go to \textbf{Step 4}.
 		\State \textbf{Step 4:} Update $\theta_n := \sigma \theta_n$ and go to \textbf{Step 3}.
 	\end{algorithmic}
 \end{algorithm}

 \subsection{Spatial discretization by Fourier spectral method}\label{discretize}
 In the presence of the external potential $V(\mathbf{x})$, the wave function is smooth and decays exponentially at the far field. Therefore, we can reasonably truncate the whole space to a large enough bounded rectangular domain $\mathcal{D}:=[-L,L]^d$, and impose periodic boundary conditions. Then we can readily apply the Fourier spectral method \cite{ShenBook,NickBook} to discretize the wave function $\Phi(\mathbf{x})=\left(\phi_1(\mathbf{x}),\phi_0(\mathbf{x}),\phi_{-1}(\mathbf{x})\right)^\top$.
 For simplicity of presentation, we shall illustrate the numerical scheme in 2D.
 
 We first discretize the computational domain $\mathcal{D}=[-L,L]^2$ with uniform grid size $h = 2L/N$, $N \in 2\mathbb{Z}^{+}$,
 then we define the physical, Fourier index and grid point sets as follows
 \begin{equation*}
 	\begin{aligned}
 		\mathcal{O}_{N} =& \left\{(m, n) \in \mathbb{N}^2 \mid 0 \leq m \leq N-1,~ 0 \leq n \leq N-1 \right\},\\
 		\Lambda_{N} =& \left\{(p, q) \in \mathbb{Z}^2 \mid -N / 2 \leq p \leq N / 2-1,~ -N / 2 \leq q \leq N / 2-1\right\},\\
 		\mathcal{Q}_{N} =& \left\{(x_m, y_n) := (-L+m h, -L+n h)\mid (m, n) \in \mathcal{O}_{N}\right\}.
 	\end{aligned}
 \end{equation*}
 and the basis of Fourier series
 \begin{equation*}
 	T_{p q}(x, y)=e^{i v_p^x\left(x+L\right)} e^{i v_q^y\left(y+L\right)},\quad
 	v_p^x=\pi p/L, ~ v_q^y=\pi q/L, \quad (p, q) \in \Lambda_{N}.
 \end{equation*}
 
 For the $\ell$-th component $\phi_{\ell}$ of the wave function, the Fourier spectral approximations of $\phi_{\ell}$, $\Delta \phi_{\ell}$, $L_z \phi_{\ell}$, $L_0 \phi_{\ell}$ and $L_1 \phi_{\ell}$ read as
 \begin{align*}
 	&\left(\phi_{\ell} \right)_{m n} \approx (\widetilde{\phi_{\ell}})_{m n} =\sum_{(p,q)\in \Lambda_{N}} \widehat{(\widetilde{\boldsymbol{\phi_{\ell}}})}_{p q} ~T_{p q}\left(x_m, y_n\right), \quad(m, n) \in \mathcal{O}_{N}, \\
 	&\left(\Delta \phi_{\ell} \right)_{m n} \approx(\llbracket \Delta \rrbracket \widetilde{\phi_{\ell}})_{m n} 	 =-\sum_{(p,q)\in \Lambda_{N}}\left(\left(v_p^x\right)^2+\left(v_q^y\right)^2\right) \widehat{(\widetilde{\boldsymbol{\phi_{\ell}}})}_{p q} ~ T_{p q}\left(x_m, y_n\right),\\
 	&\left(L_z \phi_{\ell} \right)_{m n} \approx\left(\llbracket L_z \rrbracket 	\widetilde{\phi_{\ell}}\right)_{m n} =\sum_{(p,q)\in \Lambda_{N}}\left(x_m v_q^y-y_n v_p^x\right) \widehat{(\widetilde{\boldsymbol{\phi_{\ell}}})}_{p q} ~ T_{p q}\left(x_m, y_n\right), \\
 	&\left(L_0 \phi_{\ell} \right)_{m n} \approx\left(\llbracket L_0 \rrbracket 	\widetilde{\phi_{\ell}}\right)_{m n} =\sum_{(p,q)\in \Lambda_{N}}\left(- v_p^x + i v_q^y \right) \widehat{(\widetilde{\boldsymbol{\phi_{\ell}}})}_{p q} ~T_{p q}\left(x_m, y_n\right), \\
 	&\left(L_1 \phi_{\ell} \right)_{m n} \approx\left(\llbracket L_1 \rrbracket 	\widetilde{\phi_{\ell}}\right)_{m n} =\sum_{(p,q)\in \Lambda_{N}}-\left(v_p^x + i v_q^y\right) \widehat{(\widetilde{\boldsymbol{\phi_{\ell}}})}_{p q} ~T_{p q}\left(x_m, y_n\right),
 \end{align*}
 where the Fourier coefficients $\widehat{(\widetilde{\boldsymbol{\phi_{\ell}}})}_{p q}$, defined below 
 \begin{equation*}
 	\begin{aligned}
 		\widehat{(\widetilde{\boldsymbol{\phi_{\ell}}})}_{p q} :=& \sum_{(m,n)\in\mathcal{O}_N}\mkern-12mu \phi_\ell(x_m,y_n) ~ \overline{T_{pq}(x_m,y_n)} \\
 		=& \sum_{(m,n)\in\mathcal{O}_N} \mkern-12mu \phi_\ell(x_m,y_n) ~ e^{-i v_p^x\left(x_m + L\right)} e^{-iv_q^y\left(y_n + L\right)}, ~(p,q)\in \Lambda_N,
 	\end{aligned}
 \end{equation*}
 can be accelerated by the discrete Fast Fourier transform (FFT) within $O(N^2\log N)$ float operations.

 \subsection{The multigrid PCG method}\label{multigrid}
 It is known that computation on a very fine fixed mesh leads to a large-size optimization problem and shall consume a significant amount of time, especially for high-dimensional cases. 
 Moreover, it is important to provide appropriate initial guesses so as to achieve better convergence. 
 Better initial guesses could be chosen as the commonly-used initial guesses or generated by interpolating numerical ground state obtained on coarser mesh grid.
 Here, in this article, we adopt the cascadic multigrid method, similar to that for single component BECs in \cite{pcg-dbec}, to accelerate the convergence, and shall name it as cascadic multigrid PCG (CM-PCG) method. 
 
 For sake of presentation simplicity, we define hierarchical index and grid point sets as follows
 \begin{equation*}
 	\begin{aligned}
 		\mathcal{O}_{N}^p =& \left\{(m, n) \in \mathbb{N}^2 \mid 0 \leq m \leq N_p-1,~ 0 \leq n \leq N_p-1 \right\},\quad \text{ with } \quad  N_p = 2^p, \\[0.3em]
 		\mathcal{Q}_{N}^p =& \left\{(x_m, y_n) := (-L+m h, -L+n h)\mid (m, n) \in \mathcal{O}_{N}^p\right\}.
 	\end{aligned}
 \end{equation*} The approach is as follows: 
 \begin{enumerate}
 	\item[a)] Start from initial guess $\Phi_0^{p}$ on mesh grid $\mathcal{Q}_{N}^p$ to obtain ground state $\Phi_{\rm g}^p$.
 	\item[b)] Refine the mesh $\mathcal{Q}_{N}^p$ by halving mesh size to get finer mesh $\mathcal{Q}_{N}^{p+1}$.
 	\item[c)] Interpolate $\Phi_{\rm g}^p$ on $\mathcal{Q}_{N}^{p+1}$ to obtain new initial guess $\Phi_0^{p+1}$.
 	\item[d)] Apply the PCG with $\Phi_0^{p+1}$ to get $\Phi_{\rm g}^{p+1}$, and repeat the process until the finest grid.
 \end{enumerate}
 A detailed algorithm description is summarized in Algorithm \ref{algorithm3}.
 
 \begin{algorithm}[h]
 	\caption{The CM-PCG method.}
 	\label{algorithm3}
 	\begin{algorithmic}
 		\State Given initial grid size $N_p$, initial mesh $\mathcal{Q}_{N}^{p}$ and initial guess $\Phi_0^p$.
 		\While{not reach the finest mesh}
 		\State Implement Algorithm \ref{algorithm1} to compute the ground state at level $p$: $\Phi_{\rm g}^p$
 		\State Interpolate $\Phi_{\rm g}^p$ at the refined mesh $\mathcal{Q}_{N}^{p+1}$, denoted as $\tilde{\Phi}_{\rm g}^p$
 		\State Set $\Phi_0^{p+1} = \tilde{\Phi}_{\rm g}^p$
 		\State $p = p + 1$
 		\EndWhile
 	\end{algorithmic}
 \end{algorithm}
 
 Notably, a ground state obtained at the coarser level usually provides a better initial guess for the refined level than the commonly-used initial guess, which is numerically confirmed in Example \ref{efficiency}. The multigrid strategy proved to be more efficient than directly computing on the fixed fine mesh. In addition, the cascadic multigrid technique can be extended to help accelerate other type of numerical methods, such as GFDN \cite{GFDN_2004}, the second-order flow \cite{secondflow}, etc.

 \subsection{Stopping criteria and Preconditioners}\label{precondition}
 There are three types of stopping criteria. A commonly
 used one is based on the wave function difference
 \begin{equation}\label{stop1}
 	\Phi_{\rm err}^{\infty} := \max \limits_{-1 \leq \ell \leq 1} \Vert \phi_{\ell}^{n+1} - \phi_{\ell}^n \Vert_{\infty} < \varepsilon_{\rm tol}.
 \end{equation}
 The second is based on the symmetric-covariant residual 
 \begin{equation}\label{stop2}
 	r_{\rm err}^{n,\infty} := \max \limits_{-1 \leq \ell \leq 1} \Vert r_{\ell}^n \Vert_{\infty} < \varepsilon_{\rm tol},
 \end{equation}
 and the third one is based on the symmetry-invariant energy difference 
 \begin{equation}\label{stop3}
 	\mathcal{E}_{\rm err}^n := \vert \mathcal{E}(\Phi^{n+1}) - \mathcal{E}(\Phi^n) \vert < \varepsilon_{\rm tol}.
 \end{equation}
 Stopping criterion \eqref{stop1} is commonly adopted in the gradient-flow type methods. However, it could be problematic if the minima are not isolated but form a continuum due to symmetries (for instance, complex phase). Criterion \eqref{stop3} converges faster than the other two for the same tolerance $\varepsilon_{\rm tol}$. Extensive numerical results show that it takes relatively long time to reach \eqref{stop1} or \eqref{stop2}, especially for large rotation speed and/or strong SOC interaction. \\
 
 A good preconditioner $\mathscr{P}$ is both necessary and important in order to accelerate the iteration by selecting appropriate descent direction near the minimum. Generally speaking, the preconditioner should approximate the inverse of Hessian matrix as close as possible. Similar to \cite{pcg}, we choose the following three preconditioners.
 
 \textbf{Kinetic energy preconditioner}. The first preconditioner relates only to the kinetic energy term:
 \begin{equation}\label{PreCondKin}
 	\mathscr{P}_{\Delta} = \textbf{diag} \left( \left(\alpha_{\Delta}^{1} - \frac{\Delta}{2}\right)^{-1}, \left(\alpha_{\Delta}^{0} - \frac{\Delta}{2}\right)^{-1}, \left(\alpha_{\Delta}^{-1} - \frac{\Delta}{2}\right)^{-1} \right),
 \end{equation}
 where $\alpha_{\Delta}^{\ell}$ $(\ell=1,0,-1)$ is a positive shifting constant so as to make the operator $\left(\alpha_{\Delta}^{\ell} - \frac{\Delta}{2}\right)^{-1}$ invertible. 
 When we choose $\alpha_{\Delta}^{\ell}=\frac{1}{2}$, the preconditioned descent direction $\mathscr{P}_{\Delta} r$ is equivalent to the Sobolev gradient of the Lagrangian \cite{Sobolevgrad}. Numerically, we found that the following choice
 \begin{equation}\label{PreCnstAlpha}
 	\alpha_{\Delta}^{\ell} = \int_{\mathbb{R}^d} \left( \frac{1}{2} \vert \nabla \phi_{\ell}^n \vert^2 + V \vert \phi_{\ell}^n \vert^2 + c_0 \rho^n \vert \phi_{\ell}^n \vert^2 \right) \mathrm{d} \mathbf{x} > 0,
 \end{equation} 
 is quite efficient.
 On a fixed domain with preconditioner \eqref{PreCondKin}, the iteration number is independent of mesh size but deteriorates for large domain size $L$ and/or spin-independent interaction strength $c_0$. 
 
 \textbf{Potential energy preconditioner}. The second one is related to the potential energy:
 \begin{equation}\label{PreCondPot}
 	\mathscr{P}_{V} = \textbf{diag} \left( \left(\alpha_{V}^{1} + V + c_0 \rho^n \right)^{-1}, \left(\alpha_{V}^{0} + V + c_0 \rho^n \right)^{-1}, \left(\alpha_{V}^{-1} + V + c_0 \rho^n \right)^{-1} \right).
 \end{equation}
 Similarly, $\alpha_{V}^{\ell}$ is a positive shifting constant to obtain an invertible operator. 
 We found that it is efficient to choose $\alpha_{V}^{\ell}=\alpha_{\Delta}^{\ell}$ as detailed above in Eqn.~\eqref{PreCnstAlpha}. Different from preconditioner \eqref{PreCondKin}, when the above preconditioner \eqref{PreCondPot} is applied, the performance of the preconditioner, measured in iteration number, deteriorates as the mesh size $h$ decreases but becomes stable as either the domain size $L$ or spin-independent interaction strength $c_0$ increases.
 
 \textbf{Combined energy preconditioner}. The third one is a combination of the previous two preconditioners:
 \begin{equation}\label{PreCondMix}
 	\mathscr{P} = \mathscr{P}_{C} = \mathscr{P}_{V}^{1/2} \mathscr{P}_{\Delta} \mathscr{P}_{V}^{1/2}.
 \end{equation}
 It achieves a stable performance independent of domain size $L$ and mesh size $h$. The output of operator $\mathscr{P}_{V}^{1/2}$ acting on function $\Phi(\mathbf{x})$ reads as 
 \begin{equation*}
 	\mathscr{P}_{V}^{1/2}  \Phi(\mathbf{x}) = \left(\phi_1(\mathbf{x})/ \sqrt{\alpha_{V}^{1} + V + c_0 \rho^n}, \phi_0(\mathbf{x})/ \sqrt{\alpha_{V}^{0} + V + c_0 \rho^n}, \phi_{-1}(\mathbf{x})/ \sqrt{\alpha_{V}^{-1} + V + c_0 \rho^n} \right),
 \end{equation*}
 and the implementation of $\mathscr{P}_{\Delta}$ requires solving a linear system of constant coefficients in frequency space.
 
 As shown in literature \cite{pcg,pcg-dbec,action}, the combined energy preconditioner \eqref{PreCondMix} outperforms the other two, especially for the fast rotation and strong spin-orbit coupling cases, 
 which is indispensable for the small mesh size and large domain size.
 Therefore, we always adopt the last preconditioner \eqref{PreCondMix} in our computation.

 \section{Numerical results}\label{result}
 In this section, we present numerical results to confirm the spatial accuracy and efficiency, and investigate the effects of local interaction, rotation, and spin-orbit coupling on the ground states. We also show some interesting phenomena of 3D ground states. To compute the ground states, there are ten commonly-used initial guesses for the single-component BEC \cite{pcg,pRCG} listed as follows
 \begin{align*}
 	&(a)\, \phi_{a}(\mathbf{x})=\frac{1}{\pi^{d/4}} e^{-\frac{\vert \mathbf{x} \vert^2}{2}}, \quad
 	(b)\, \phi_{b}(\mathbf{x})=(x+i y) \phi_{a}(\mathbf{x}), \quad
 	(\bar{b})\, \phi_{\bar{b}}(\mathbf{x})=\bar{\phi}_{b}(\mathbf{x}), \\
 	&(c)\, \phi_c(\mathbf{x})=\frac{\phi_a(\mathbf{x})+\phi_b(\mathbf{x})}{\left\|\phi_a(\mathbf{x})+\phi_b(\mathbf{x})\right\|}, \quad
 	(\bar{c})\, \phi_{\bar{c}}(\mathbf{x})=\bar{\phi}_c(\mathbf{x}), \\
 	&(d)\, \phi_d(\mathbf{x})=\frac{(1-\Omega) \phi_a(\mathbf{x})+\Omega \phi_b(\mathbf{x})}{\left\|(1-\Omega) \phi_a(\mathbf{x})+\Omega \phi_b(\mathbf{x})\right\|}, \quad
 	(\bar{d})\, \phi_{\bar{d}}(\mathbf{x})=\bar{\phi}_d(\mathbf{x}), \\
 	&(e)\, \phi_e(\mathbf{x})=\frac{\Omega \phi_a(\mathbf{x})+(1-\Omega) \phi_b(\mathbf{x})}{\left\|\Omega \phi_a(\mathbf{x})+(1-\Omega) \phi_b(\mathbf{x})\right\|}, \quad
 	(\bar{e})\, \phi_{\bar{e}}(\mathbf{x})=\bar{\phi}_e(\mathbf{x}), \quad (f)\ \phi_f(\mathbf{x})=\frac{\phi_{\rm g}^{\text{TF}}(\mathbf{x})}{\left\|\phi_{\rm g}^{\text{TF}}(\mathbf{x})\right\|},
 \end{align*}
 where
 \begin{equation*}
 	\phi_{\rm g}^{\text{TF}}(\mathbf{x})= \begin{cases}
 		\sqrt{(\mu_{\rm g}^{\text{TF}} - V(\mathbf{x}))/c_0}, & V(\mathbf{x}) < \mu_{\rm g}^{\text{TF}}, \\
 		0, & \text{otherwise},
 	\end{cases}
 \end{equation*}
 with \begin{equation*}
 	\mu_{\rm g}^{\text{TF}} = \frac{1}{2} \begin{cases}
 		(4 c_0 \gamma_x \gamma_y / \pi)^{1/2}, & d = 2, \\
 		(15 c_0 \gamma_x \gamma_y \gamma_z / 4\pi)^{2/5}, & d = 3.
 	\end{cases}
 \end{equation*}
 
 In Example \ref{accuracy}, \ref{Effect-local}-\ref{3d-result}, to compute the ground states, the initial guesses can be selected as 
 $\Phi^0 = (\phi_1^0, \phi_0^0,\phi_{-1}^0)^\top/\sqrt{3}$ where $\phi_1^0$, $\phi_0^0$ and $\phi_{-1}^0$ 
 are chosen independently from the above 10 types, resulting in 1000 combinations. Each of these initial data leads to a numerical stationary state by our numerical methods. And, the one with the lowest energy is viewed as the numerical ground state. 
 
 In the following experiments, unless stated otherwise, we choose the harmonic potential \eqref{harmonic} as the trapping potential with $\gamma_x = \gamma_y = 1$ for 2D and $\gamma_x = \gamma_y =\gamma_z = 1$ for 3D. 
 The computational domain $\mathcal D = [-L,L]^2$ is partitioned with the uniform mesh size in each direction. Unless otherwise specified, we adopt criterion \eqref{stop3} as the default stopping condition and fix $\varepsilon_{\rm tol} = 10^{-14}$. The algorithms were implemented in FORTRAN and run on a 3.00GH Intel(R) Xeon(R) Gold 6248R CPU with 36 MB cache in Ubuntu GNU/Linux.

 \begin{exmp}\label{accuracy}
 	Here, we test the spatial accuracy using the following three cases: 
 	\begin{flalign*}
 		&\textbf{ Case 1. } V\left(\mathbf{x}\right) = \frac{1}{2} \left(x^2 + y^2\right), c_0 = 100, c_1 = 1, \Omega = 0.1, \gamma = 0.3 \text{ in 2D}. && \\ 
 		&\textbf{ Case 2. } V\left(\mathbf{x}\right) = \frac{1}{2} \left(x^2 + y^2\right), c_0 = 100, c_1 = 1, \Omega = 0.3, \gamma = 0.3 \text{ in 2D}. && \\
 		&\textbf{ Case 3. } V\left(\mathbf{x}\right) = \frac{1}{2} \left(x^2 + y^2 + z^2 \right), c_0 = 100, c_1 = 1, \Omega = 0.1, \gamma = 0.3 \text{ in 3D}. && 
 	\end{flalign*}	
 \end{exmp}
 
 We choose domain $\mathcal{D} = [-16,16]^d$. Let $\Phi_{\rm g} = \left(\phi_1^{\rm g},\phi_0^{\rm g},\phi_{-1}^{\rm g}\right)^\top$ be the numerical ``exact" ground state obtained by Algorithm \ref{algorithm3} using 5-level cascadic multigrid with finest mesh being $h = \frac{1}{16}$, and $\Phi_{\rm g}^{h}$ be the numerical solution obtained with mesh size $h$. Denote the energy and chemical potential as $\mathcal{E}_{\rm g} := \mathcal{E}\left(\Phi_{\rm g}\right)$ and $\mu_{\rm g} := \mu\left(\Phi_{\rm g}\right)$. We use the following relative maximum norm to measure the error of the wave function:
 $$E_h := \left\Vert \bar{\kappa} \Phi_{\rm g}^h - \Phi_{\rm g} \right\Vert_{\infty}/\left\Vert \Phi_{\rm g} \right\Vert_{\infty},$$
 where $\kappa = \langle \Phi_{\rm g}^h, \Phi_{\rm g} \rangle/\langle \Phi_{\rm g}, \Phi_{\rm g} \rangle$ is derived from \eqref{phase-factor}. Additionally, we compute the residual of virial identity \eqref{virial} as
 $$
 I_h := 2\mathcal{E}_{\rm kin}\left(\Phi_{\rm g}^h\right) - 2 \mathcal{E}_{\rm pot}\left(\Phi_{\rm g}^h\right) + d\ \mathcal{E}_{\rm spin}\left(\Phi_{\rm g}^h\right) + \mathcal{E}_{\rm soc}\left(\Phi_{\rm g}^h\right),\qquad d=2,3.
 $$
 In this example, to compute numerical solutions with mesh size $h = 2^{-p}$ $(p = 0, 1, 2, 3)$, we always start with coarsest mesh $h=1$ and end up with finest mesh $h = 2^{-p}$, and the stopping criterion \eqref{stop2} is set with $\varepsilon_{\rm tol} = 10^{-12}$.
 
 \begin{table}[h]
 	\centering
 	\tabcolsep=0.7cm
 	\renewcommand\arraystretch{1.4}
 	\caption{Spatial discretization errors for \textbf{Case 1} (upper), \textbf{Case 2} (middle) and \textbf{Case 3} (lower) in \textbf{Example \ref{accuracy}}.}
 	\begin{tabular}{cccccc} 
 		\hline
 		Mesh size & $h=1$ & $h=\frac{1}{2}$ & $h=\frac{1}{4}$ & $h=\frac{1}{8}$ \\
 		\hline
 		$E_h$ &3.1010E-03&1.6363E-05&3.9117E-09&3.0701E-12 \\
 		$\left\vert \mathcal{E}_{\rm g} - \mathcal{E}\left(\Phi_{\rm g}^{h}\right) \right\vert$ &9.0058E-04&6.0350E-08&2.3181E-13&3.9524E-14 \\
 		$\left\vert \mu_{\rm g} - \mu\left(\Phi_{\rm g}^{h}\right) \right\vert$ &6.7070E-04&3.1782E-06&4.9827E-13&5.5955E-14 \\
 		$I_h$&1.2630E-02&1.0721E-05&7.6072E-13&8.0491E-16 \\
 		\hline
 		$E_h$ &6.5353E-03&1.9980E-05&4.5145E-10&1.0198E-11 \\
 		$\left\vert \mathcal{E}_{\rm g} - \mathcal{E}\left(\Phi_{\rm g}^{h}\right) \right\vert$ &7.9070E-04&2.4096E-07&7.1054E-15&9.7700E-15 \\
 		$\left\vert \mu_{\rm g} - \mu\left(\Phi_{\rm g}^{h}\right) \right\vert$ &2.1377E-03 &2.2019E-06&5.7732E-14&1.5099E-14 \\
 		$I_h$&1.7800E-02&3.1102E-06&3.6637E-14&7.7715E-15 \\
 		\hline
 		$E_h$ &4.8538E-03&3.3376E-05&4.0351E-08&9.9374E-10 \\
 		$\left\vert \mathcal{E}_{\rm g} - \mathcal{E}\left(\Phi_{\rm g}^{h}\right) \right\vert$ &4.2872E-04&1.5452E-07&1.4540E-11&1.3480E-11 \\
 		$\left\vert \mu_{\rm g} - \mu\left(\Phi_{\rm g}^{h}\right) \right\vert$ &3.0351E-03&4.6866E-08&9.1300E-12&7.5100E-12 \\
 		$I_h$&2.0030E-02&4.0704E-06&1.7062E-10&4.2660E-11 \\
 		\hline 
 	\end{tabular}
 	\label{table_acc}
 \end{table}
 
 Tab.~\ref{table_acc} lists the spatial errors of the ground state for \textbf{Case 1}, \textbf{Case 2} in 2D and \textbf{Case 3} in 3D.
 For brevity, we do not show the ground state here. In fact, both the ground states of \textbf{Case 1} and \textbf{Case 2} have vortices, while there is no vortex in \textbf{Case 3}.
 From Tab.~\ref{table_acc}, we can conclude that: our method is spectrally accurate for the ground state, regardless of space dimension and whether vortices exist or not.

 \begin{exmp}\label{perf-comp}
 	In this example, we compare our algorithm with existing numerical methods. Specifically, we consider the \textbf{Case 2} in Example \ref{accuracy}, and we apply Algorithm \ref{algorithm1} and the known projected gradient flow (PGF) method \cite{PGF} to solve the ground state.
 \end{exmp}
 
 \begin{figure}[ht]
 	\centering
 	\includegraphics[scale=0.4]{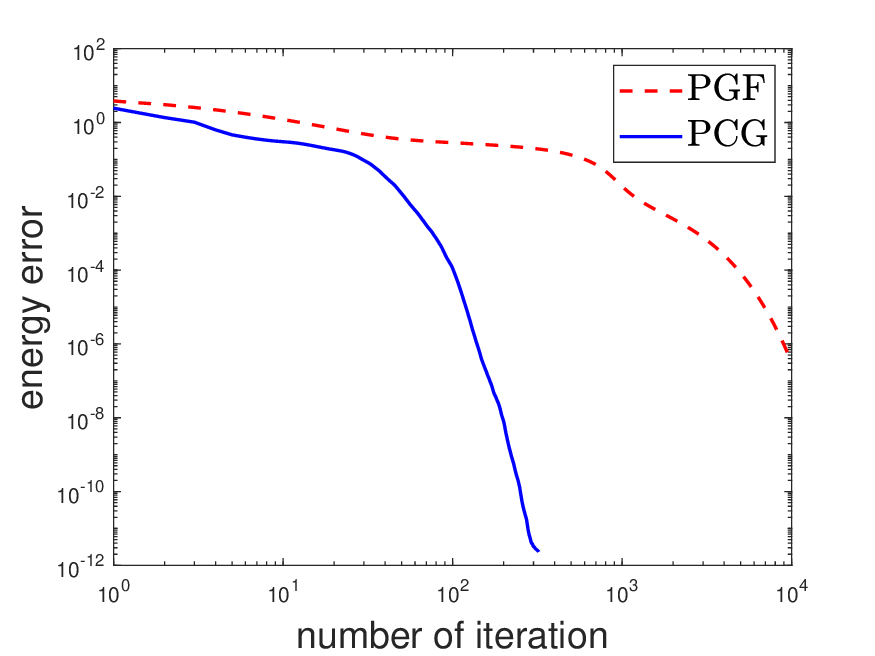}
 	\hspace{-0.22cm}
 	\includegraphics[scale=0.4]{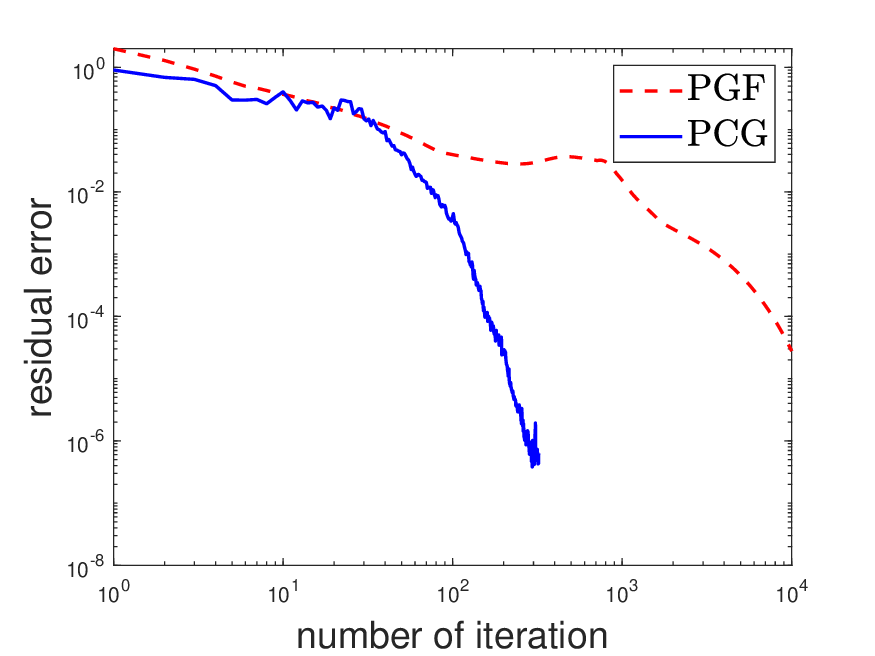}
 	\hspace{-0.22cm}
 	\caption{Residual error and the energy error against the number of iterations for PCG and PGF for 2D BEC in \textbf{Example \ref{perf-comp}}.}
 	\label{fig-comparison}
 \end{figure}
 
 We take computational domain $\mathcal{D}=[-12,12]^2$ and the grid number $N = 128$. In the PGF method, we apply backward/forward Euler Fourier spectral scheme with time step $\Delta t = 0.1$ and the same stopping criterion as in Algorithm \ref{algorithm1}. The residual error and the energy error $\vert \mathcal{E}(\Phi^n) - \mathcal{E}_{\rm g} \vert$ against the number of iterations are plotted in Fig.~\ref{fig-comparison} for the two methods, where $\mathcal{E}_{\rm g}$ is obtained as in Example \ref{accuracy}. From Fig.~\ref{fig-comparison}, it can be seen that our algorithm converges far more rapidly than the PGF method, and the former reaches convergence in 323 iterations (CPU time: 1.998 s) and the latter in 18168 iterations (CPU time: 78.608 s).

 \begin{exmp}\label{efficiency}
 	Here, we investigate the effect of different initial guesses and compare the performance of PCG (Algorithm \ref{algorithm1}) and CM-PCG (Algorithm \ref{algorithm3}). We specify the following parameters: $\gamma = 0.3$, $c_0 = 50$ and $c_1 = 0.5$. 
 \end{exmp}
 
 We take computational domain $\mathcal{D}=[-12,12]^2$. For PCG method, we set the mesh fixed with the grid number $N = 256$. For CM-PCG, we set the coarsest and finest meshes with grid number $N = 64$ and $N = 256$, respectively. The stopping criterion \eqref{stop3} is applied for both Algorithm \ref{algorithm1} and Algorithm \ref{algorithm3} (for all meshes) with $\varepsilon_{\rm tol} = 10^{-14}$. 
 
 \begin{table}[h]
 	\centering
 	\renewcommand\arraystretch{1.4}
 	\caption{PCG with fixed mesh ($N = 256$) and CM-PCG ($N = 64$ to $N = 256$): converged energies and the CPU time (in seconds) for the solution with the lowest energy (underlined) in \textbf{Example \ref{efficiency}}.}
 	\label{table_cpu}
 	\begin{tabular}{cccccccccccc}
 		\hline
 		\multicolumn{12}{c}{\textbf{PCG with fixed mesh}} \\
 		\hline
 		$\Omega$ & (a) & (b) & ($\bar{b}$) & (c) & ($\bar{c}$) & (d) & ($\bar{d}$) & (e) & ($\bar{e}$) & (f) & CPU \\
 		\hline
 		0.6 & 2.2851 & \uline{2.2833} & 2.2851 & 2.2833 & 2.2833 & 2.2833 & 2.2833 & 2.2833 & 2.2833 & 2.2851 & 1058 \\
 		0.7 & 2.0437 & 2.0437 & 2.0437 & 2.0437 & 2.0437 & 2.0437 & 2.0437 & \uline{2.0437} & 2.0437 & 2.0538 & 5740 \\
 		0.8 & 1.6806 & 1.6806 & \uline{1.6806} & 1.6806 & 1.6806 & 1.6806 & 1.6808 & 1.6808 & 1.6808 & 1.6806 & 1964 \\
 		0.9 & 1.0520 & 1.0507 & 1.0520 & 1.0507 & 1.0492 & 1.0493 & 1.0492 & \uline{1.0492} & 1.0492 & 1.0507 & 11090 \\
 		\hline
 		\multicolumn{12}{c}{\textbf{CM-PCG}} \\
 		\hline
 		$\Omega$ & (a) & (b) & ($\bar{b}$) & (c) & ($\bar{c}$) & (d) & ($\bar{d}$) & (e) & ($\bar{e}$) & (f) & CPU \\
 		\hline
 		0.6 & 2.2851 & \uline{2.2833} & 2.2851 & 2.2833 & 2.2833 & 2.2833 & 2.2833 & 2.2833 & 2.2833 & 2.2833 & 7 \\
 		0.7 & 2.0437 & 2.0437 & 2.0437 & 2.0437 & 2.0437 & 2.0437 & 2.0437 & 2.0437 & 2.0437 & \uline{2.0437} & 116 \\
 		0.8 & 1.6806 & 1.6806 & 1.6808 & 1.6806 & 1.6806 & \uline{1.6806} & 1.6806 & 1.6808 & 1.6806 & 1.6806 & 46 \\
 		0.9 & 1.0492 & 1.0492 & \uline{1.0492} & 1.0516 & 1.0492 & 1.0492 & 1.0492 & 1.0492 & 1.0492 & 1.0492 & 1340 \\
 		\hline
 	\end{tabular}
 \end{table} 
 For the sake of simplicity in presentation here, we only present 10 types of initial guesses, where each component is identical. Tab.~\ref{table_cpu} lists the energies obtained by the PCG with a fixed mesh and the CM-PCG for different initial guesses $(a)$-$(f)$ and rotational frequency $\Omega$. The stationary states $\Phi_{\rm g}(\mathbf{x})$ with the lowest energies are underlined and the corresponding CPU time is listed in the last column. From this table and additional numerical results not shown here for brevity, we conclude that: different initial guesses affect the final converged stationary states and some may lead to local minima. The CM-PCG approach is more efficient than the PCG with a fixed mesh, and showcases a remarkable speedup ratio of about 8 to 150 times.

 \begin{exmp}\label{Effect-local}
 	In this example, we consider the effects of spin-independent and spin-exchange interactions. We set $\Omega = 0.4$ and $\gamma = 0.8$. 
 	The effects of local interactions are studied in the following six cases:\\
 	\begin{tabular}{ll}
 		\textbf{ Case 1. } $c_0 = 100$ and $c_1 = 1$. & \textbf{ Case 2. } $c_0 = 200$ and $c_1 = 1$. \\
 		\textbf{ Case 3. } $c_0 = 600$ and $c_1 = 1$. & \textbf{ Case 4. } $c_0 = 200$ and $c_1 = -20$. \\
 		\textbf{ Case 5. } $c_0 = 200$ and $c_1 = 10$. & \textbf{ Case 6. } $c_0 = 200$ and $c_1 = 30$. \\
 	\end{tabular}
 \end{exmp}
 
 \begin{figure}[h]
 	\centering
 	\includegraphics[scale=0.35]{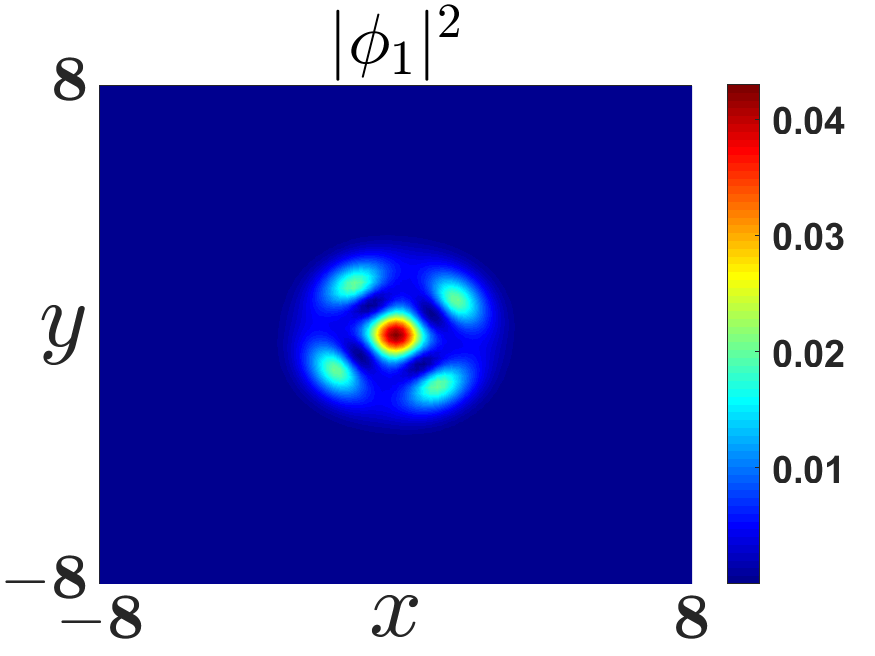}
 	\hspace{-0.22cm}
 	\includegraphics[scale=0.35]{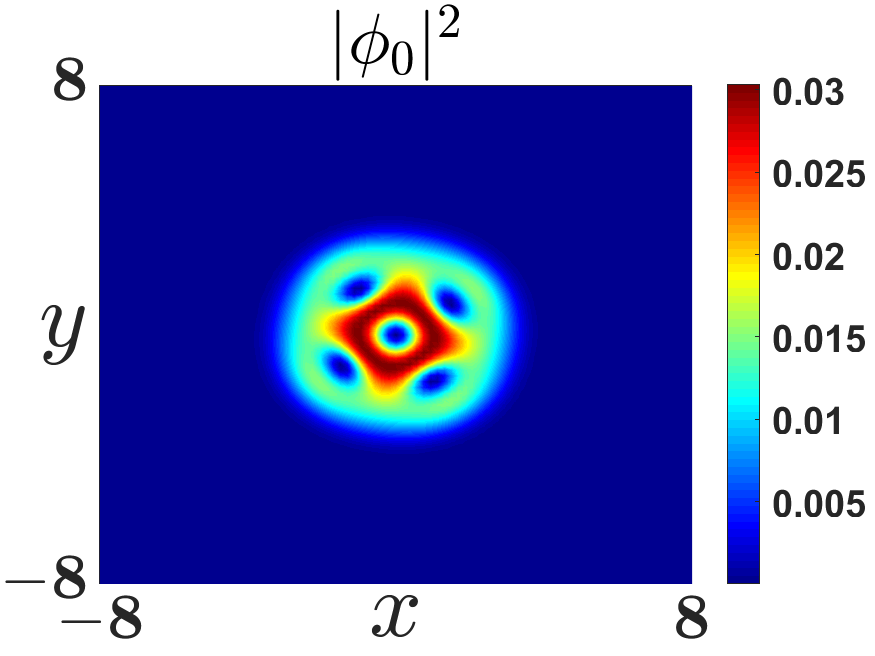}
 	\hspace{-0.22cm}
 	\includegraphics[scale=0.35]{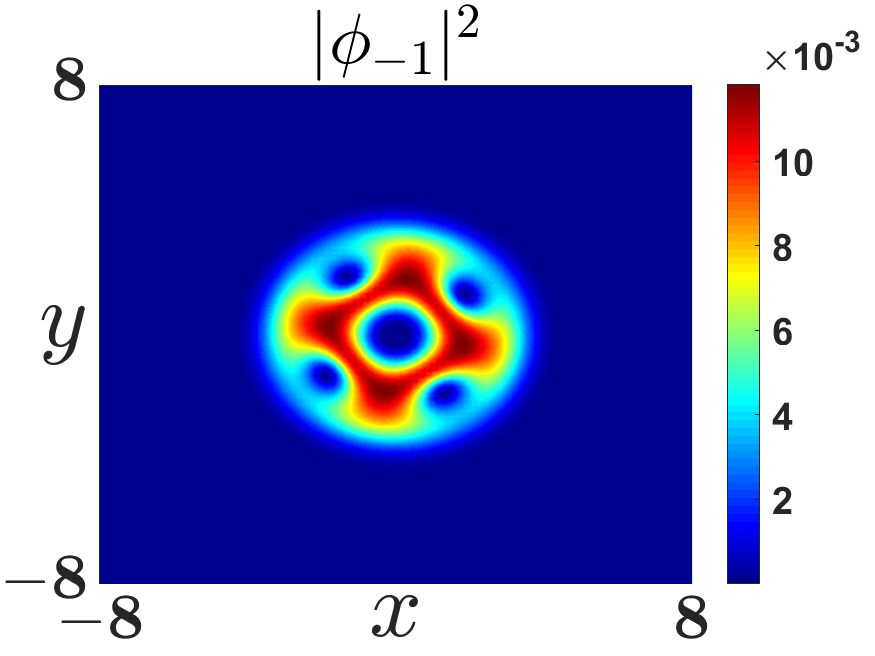}
 	\\
 	\centering
 	\includegraphics[scale=0.35]{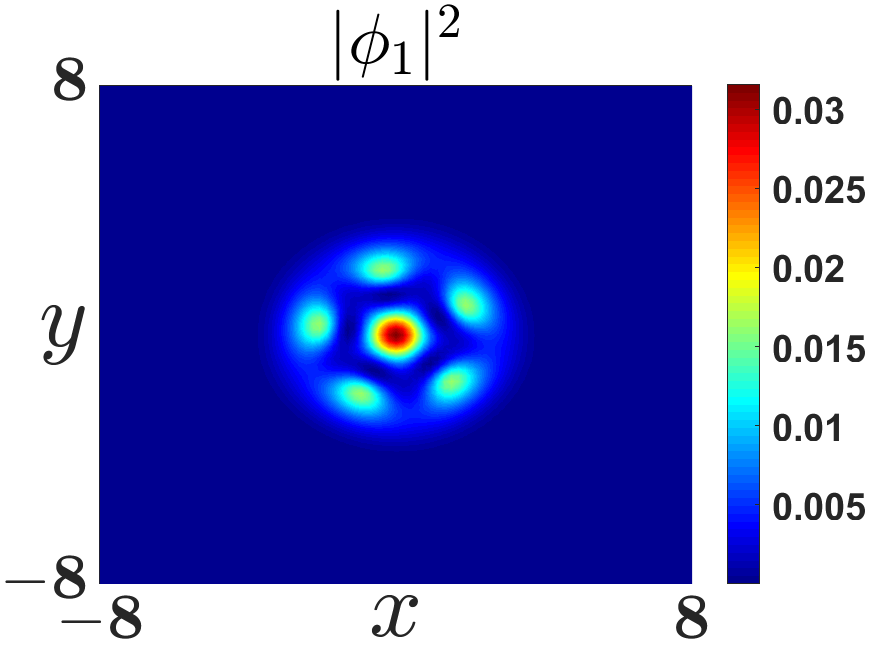}
 	\hspace{-0.22cm}
 	\includegraphics[scale=0.35]{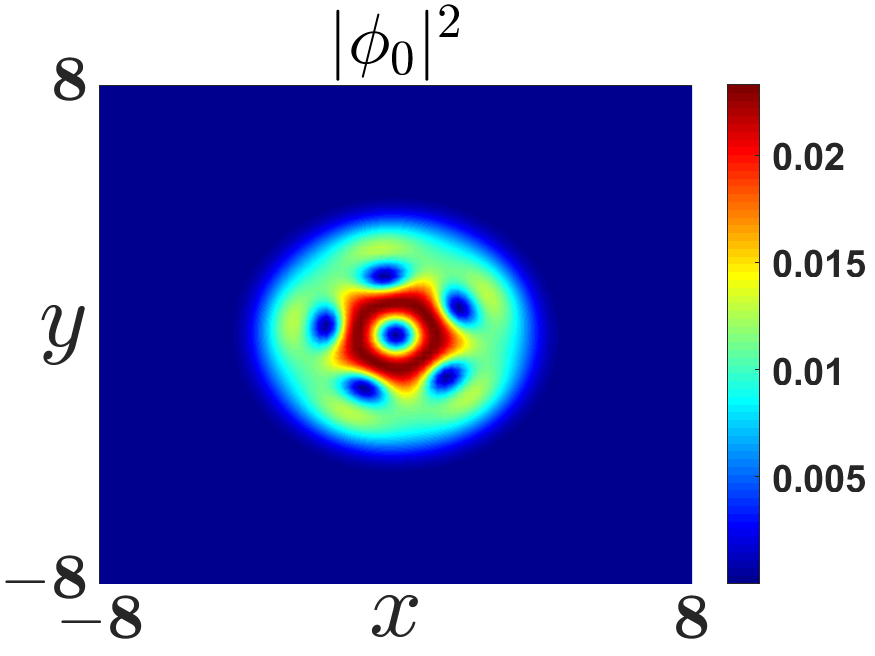}
 	\hspace{-0.22cm}
 	\includegraphics[scale=0.35]{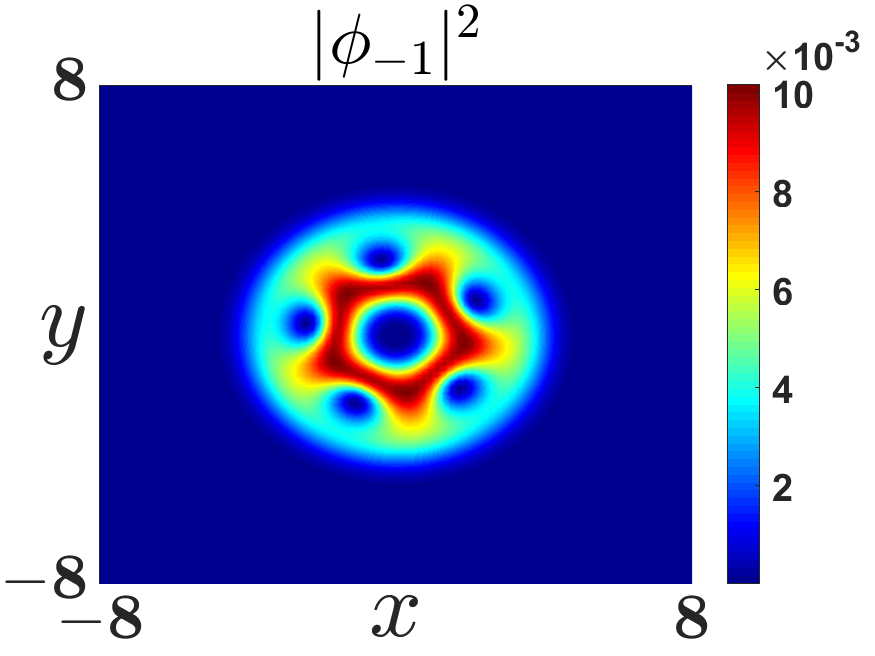}
 	\\
 	\centering
 	\includegraphics[scale=0.35]{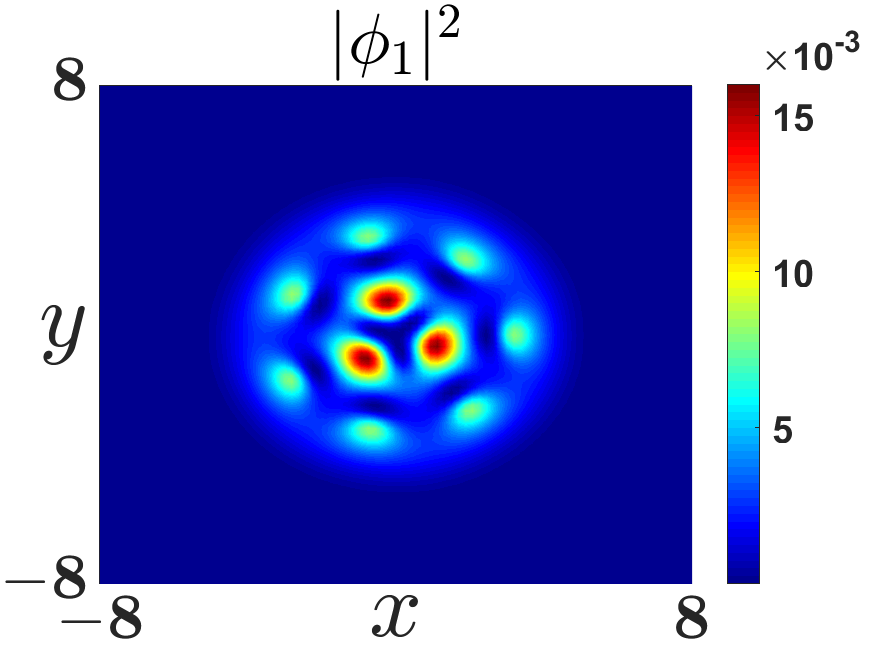}
 	\hspace{-0.22cm}
 	\includegraphics[scale=0.35]{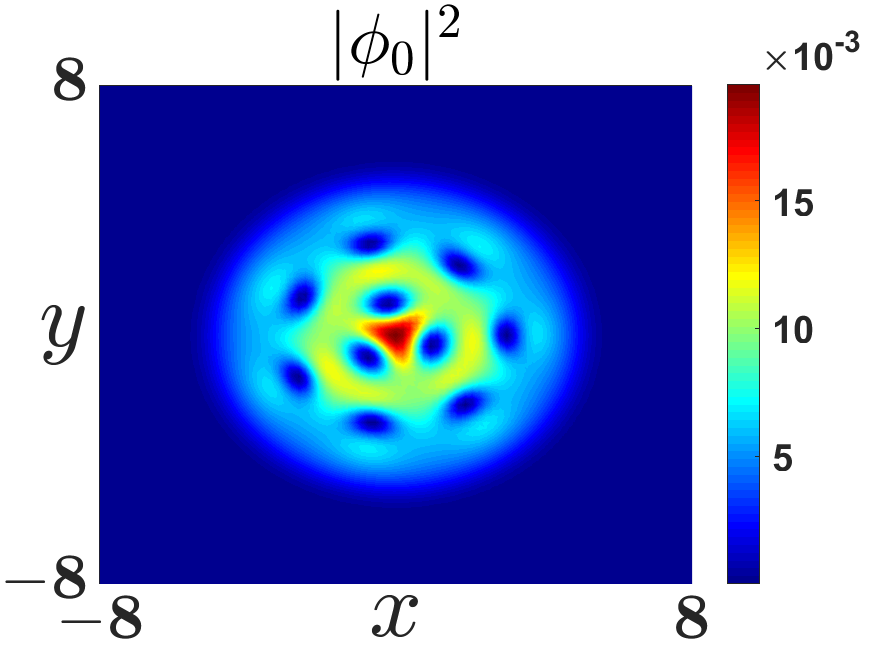}
 	\hspace{-0.22cm}
 	\includegraphics[scale=0.35]{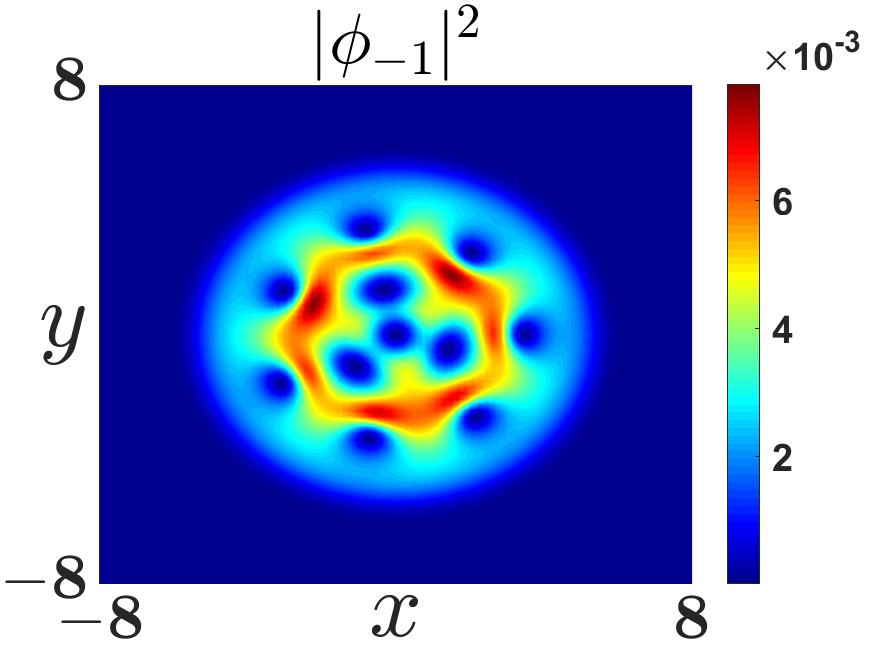}
 	\\
 	\caption{Contour plots of the densities for \textbf{Case 1, 2, 3} (from top to bottom, $c_0$=100,200,600) in \textbf{Example \ref{Effect-local}} to investigate the effect of spin-independent interaction.}
 	\label{fig_int_c0}
 \end{figure}\vspace{18pt}
 
 \begin{figure}[h]
 	\centering
 	\includegraphics[scale=0.35]{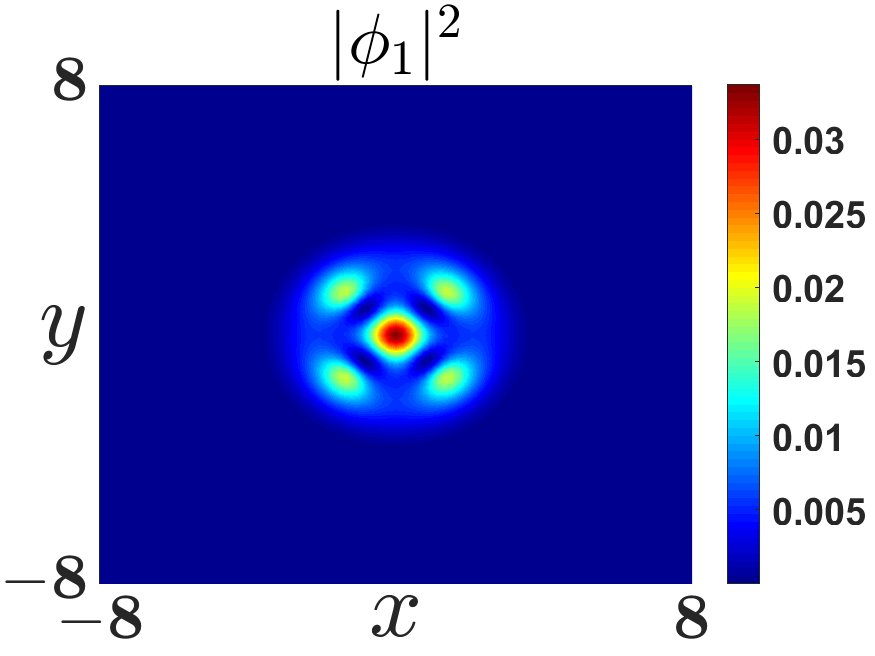}
 	\hspace{-0.22cm}
 	\includegraphics[scale=0.35]{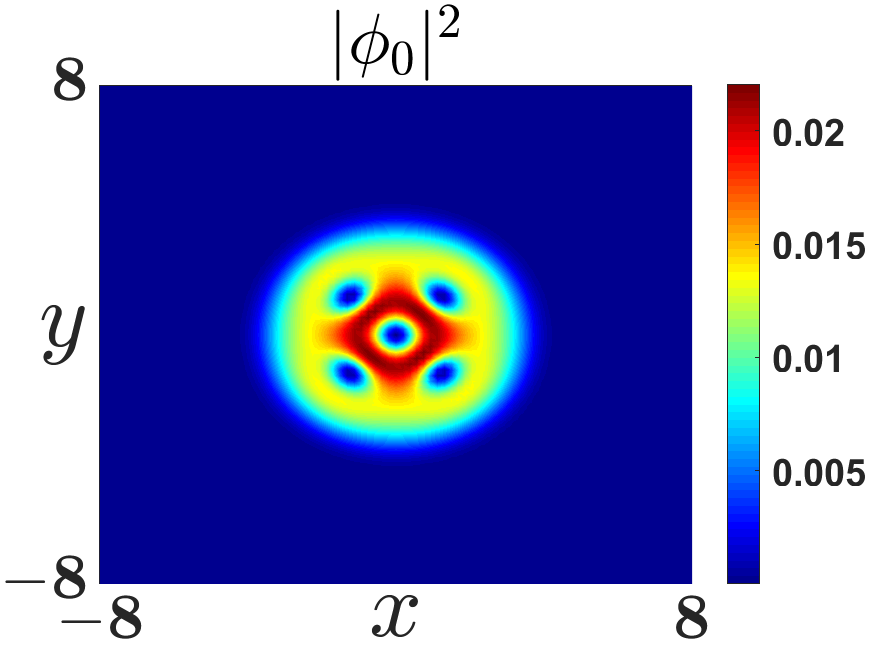}
 	\hspace{-0.22cm}
 	\includegraphics[scale=0.35]{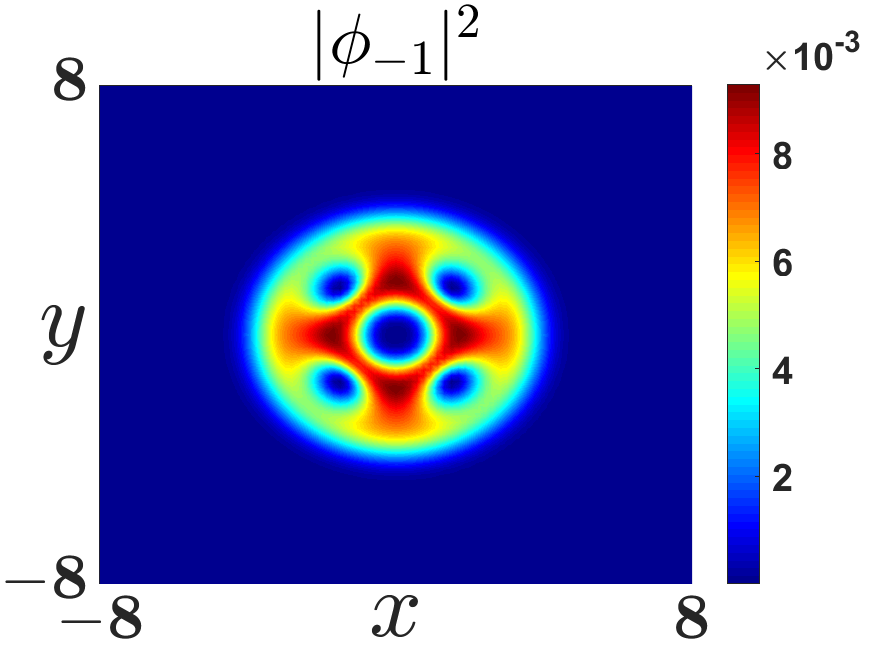}
 	\\
 	\centering
 	\includegraphics[scale=0.35]{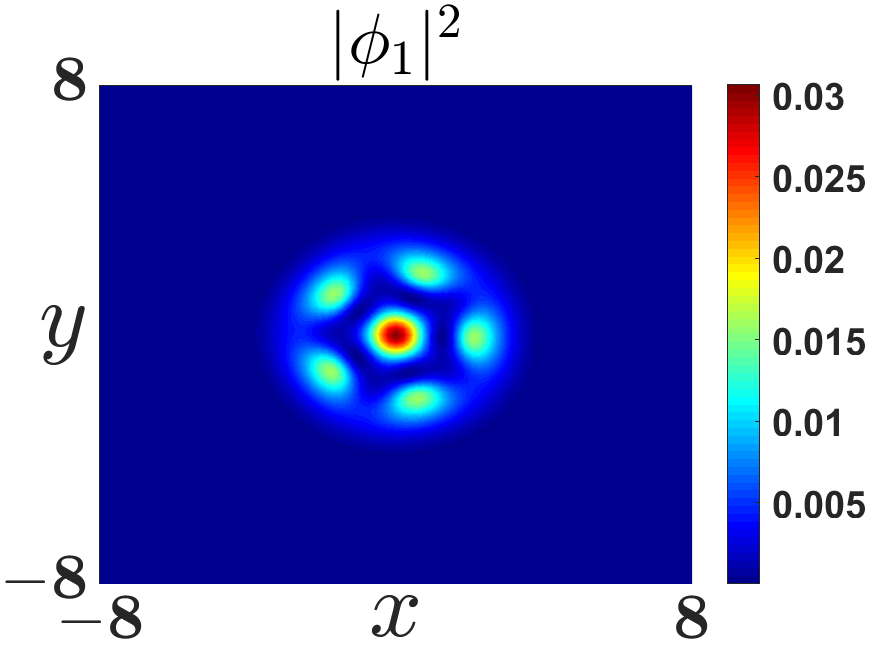}
 	\hspace{-0.22cm}
 	\includegraphics[scale=0.35]{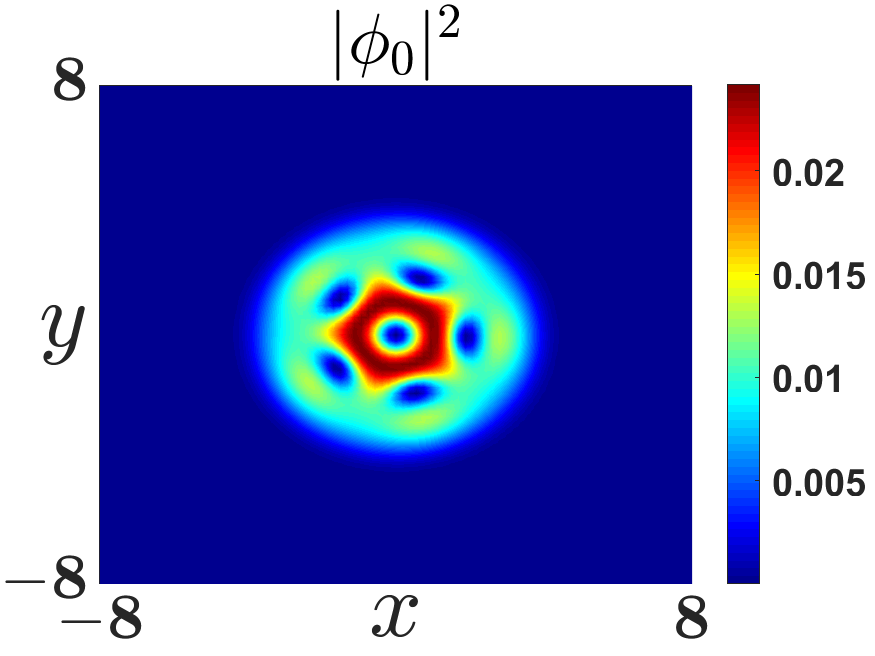}
 	\hspace{-0.22cm}
 	\includegraphics[scale=0.35]{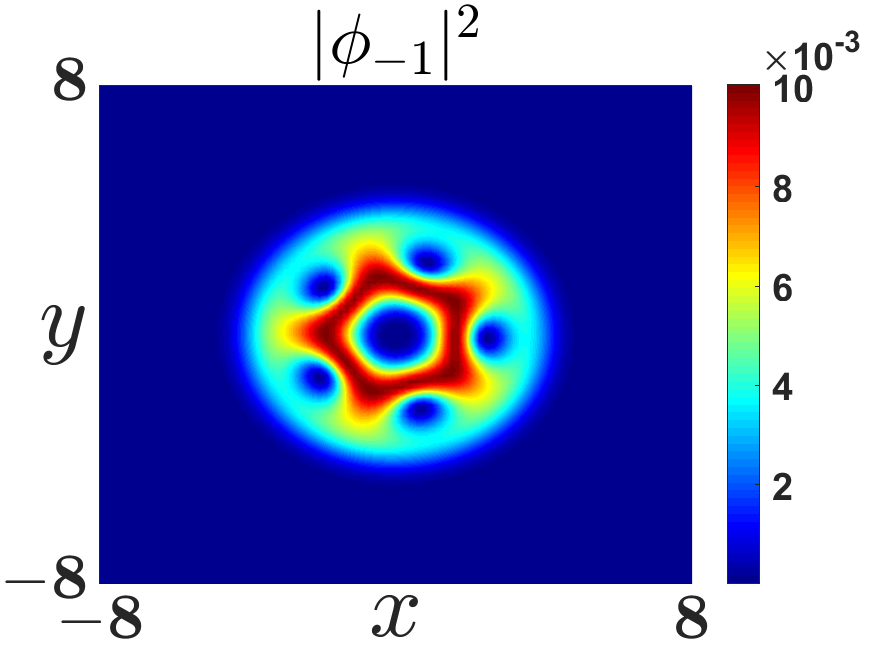}
 	\\
 	\centering
 	\includegraphics[scale=0.35]{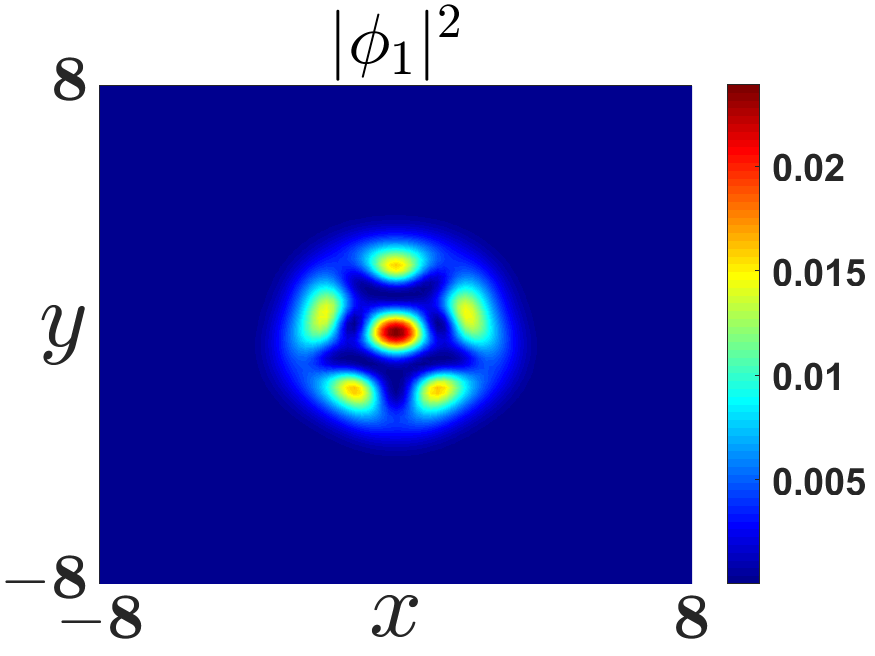}
 	\hspace{-0.22cm}
 	\includegraphics[scale=0.35]{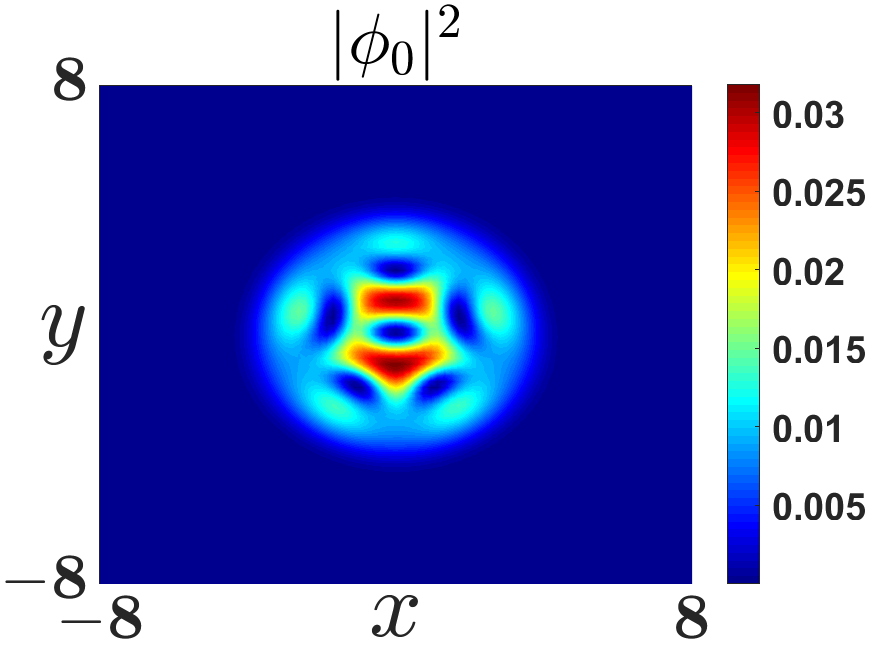}
 	\hspace{-0.22cm}
 	\includegraphics[scale=0.35]{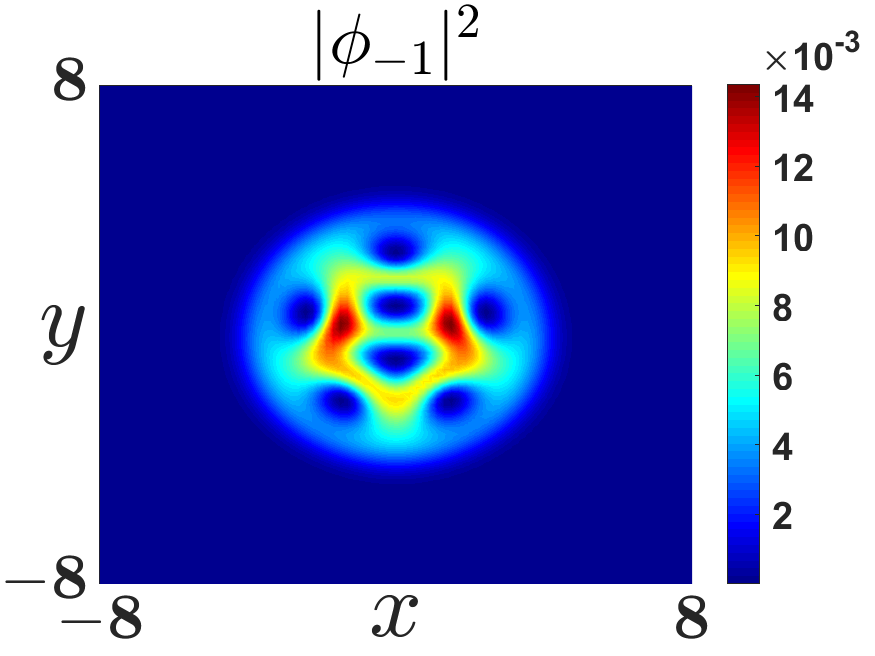}
 	\\
 	\caption{Contour plots of the densities for \textbf{Case 4, 5, 6} (from top to bottom, $c_1$=-20,10,30) in \textbf{Example \ref{Effect-local}} to investigate the effect of spin-exchange interaction.}
 	\label{fig_int_c1}
 \end{figure}
 
 The computational domain is $\mathcal{D} = [-12,12]^2$ and the grid number $N = 256$. 
 Fig.~\ref{fig_int_c0} presents the density plots for \textbf{Case 1-3}, illustrating the influence of the spin-independent interaction parameter $c_0$. Similarly, Fig.~\ref{fig_int_c1} displays the results for \textbf{Case 4-6}, highlighting the effects of the spin-exchange interaction parameter $c_1$. From Fig.~\ref{fig_int_c0} and Fig.~\ref{fig_int_c1}, several interesting phenomena can be observed. In particular, as shown in Fig.~\ref{fig_int_c0}, an increase in spin-independent interaction strength $c_0$ leads to a greater number of vortices, which is attributed to the enhanced repulsive interactions among atoms. Moreover, as demonstrated in Fig.~\ref{fig_int_c1}, increasing the spin-exchange interaction strength $c_1$ also results in an increase in the number of vortices. Additionally, numerical results show that all the ground states exhibit axial symmetry.
 
 \begin{exmp}\label{Effect-Rot}
 	Here, we consider the effect of rotation on the ground state in harmonic potential and harmonic-plus-quartic potential, respectively. The harmonic-plus-quartic potential is defined as
 	\begin{equation}\label{harm-quart}
 		V\left(\mathbf{x}\right) = -0.2 \left(x^2 + y^2\right) + 0.5\left(x^2 + y^2\right)^2.
 	\end{equation} 
 	We set the parameters as $c_1 = 1$, $\gamma = 0.8$, and test the following two cases:
 	\begin{enumerate}[label=\textbf{Case \arabic*. }, leftmargin=*, align=left]
 		\item Fix harmonic potential \eqref{harmonic} with $\gamma_x = \gamma_y = 1$, $c_0 = 100$, and vary $\Omega$ among $0$, $0.45$, $0.90$.
 		\item Fix harmonic-plus-quartic potential \eqref{harm-quart}, $c_0 = 1000$, and vary $\Omega$ among $3.0$, $6.0$, $14.0$.
 	\end{enumerate}
 \end{exmp}
 
 \begin{figure}[h]
 	\centering
 	\includegraphics[scale=0.35]{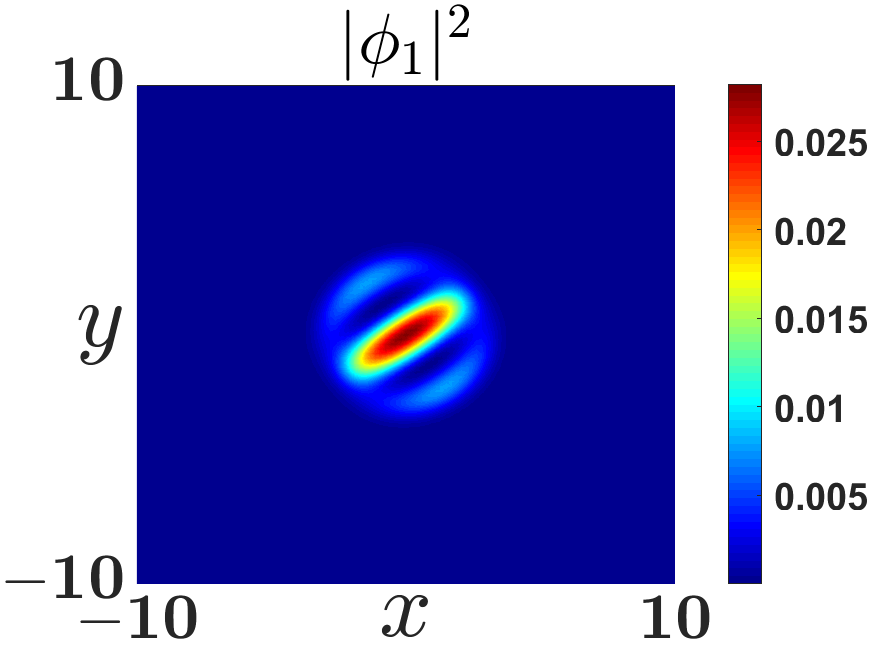}
 	\hspace{-0.22cm}
 	\includegraphics[scale=0.35]{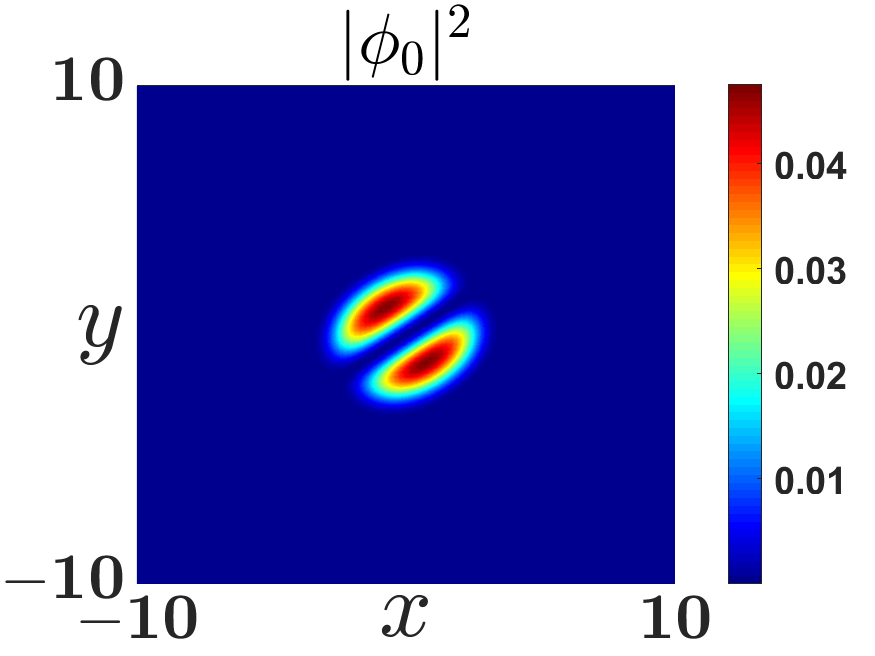}
 	\hspace{-0.22cm}
 	\includegraphics[scale=0.35]{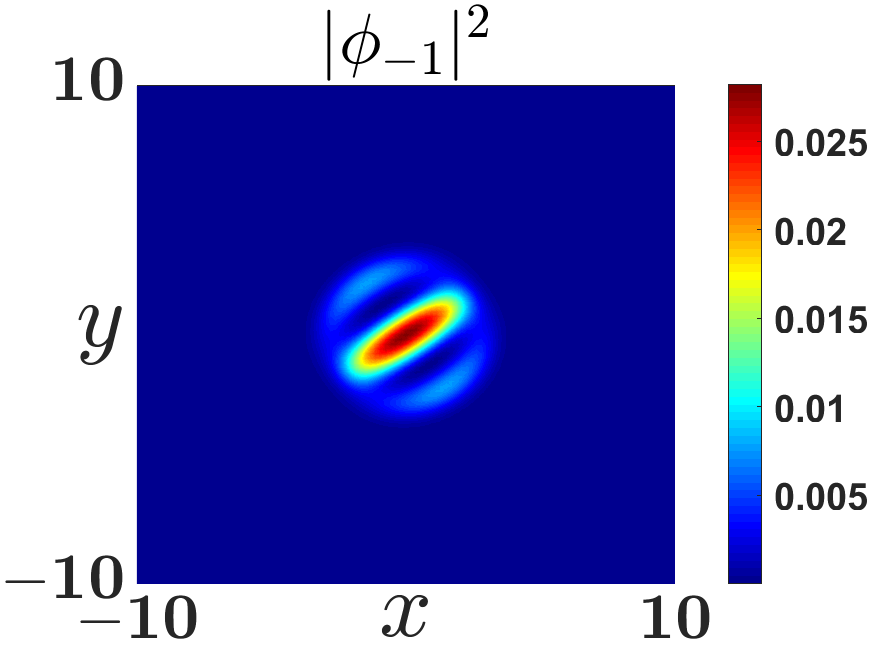}
 	\\
 	\centering
 	\includegraphics[scale=0.35]{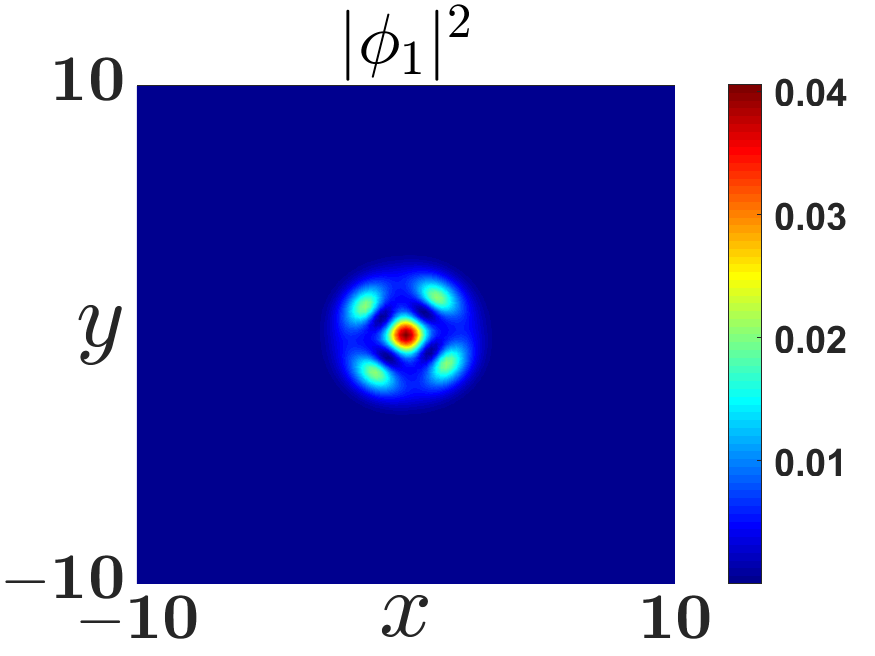}
 	\hspace{-0.22cm}
 	\includegraphics[scale=0.35]{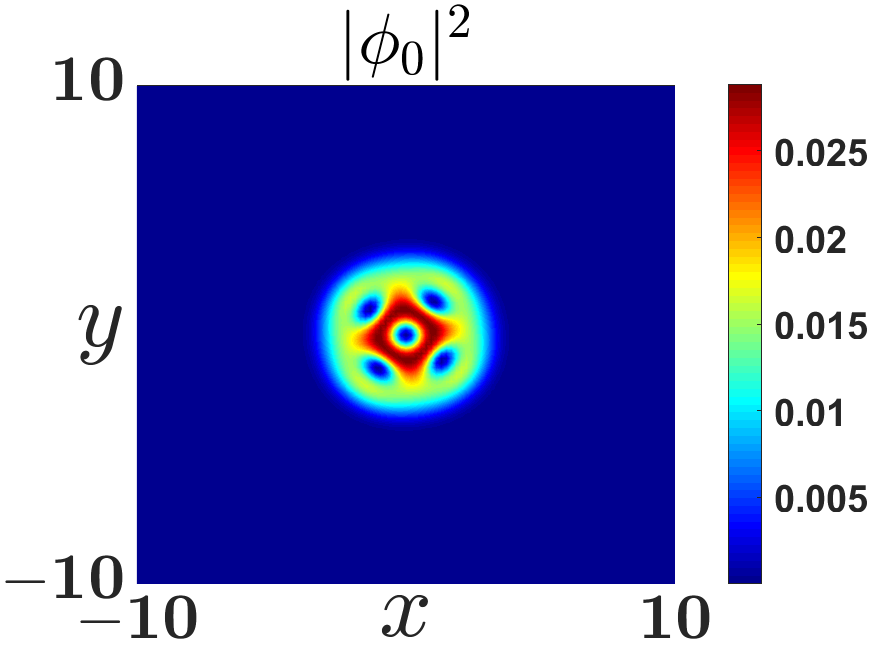}
 	\hspace{-0.22cm}
 	\includegraphics[scale=0.35]{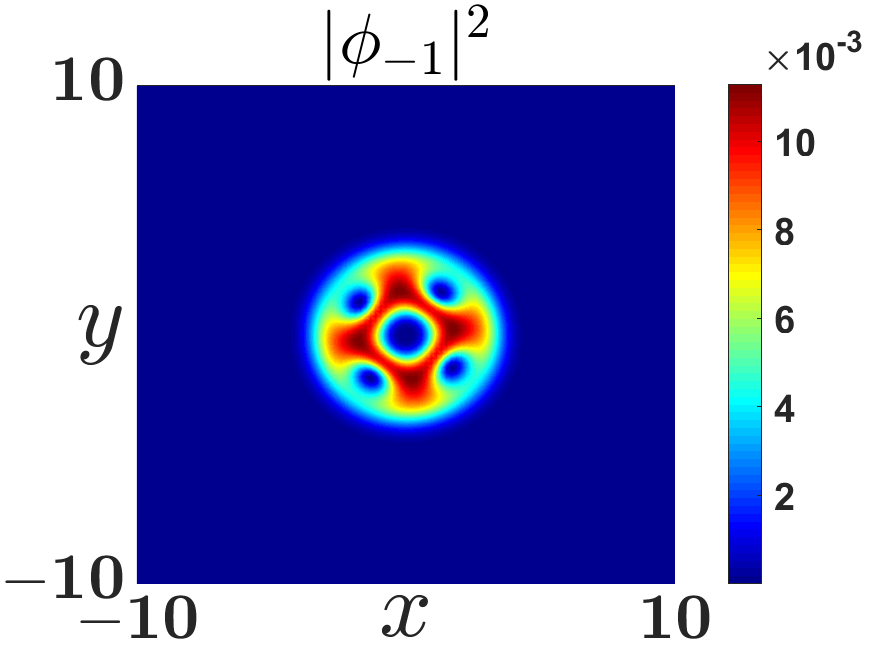}
 	\\
 	\centering
 	\includegraphics[scale=0.35]{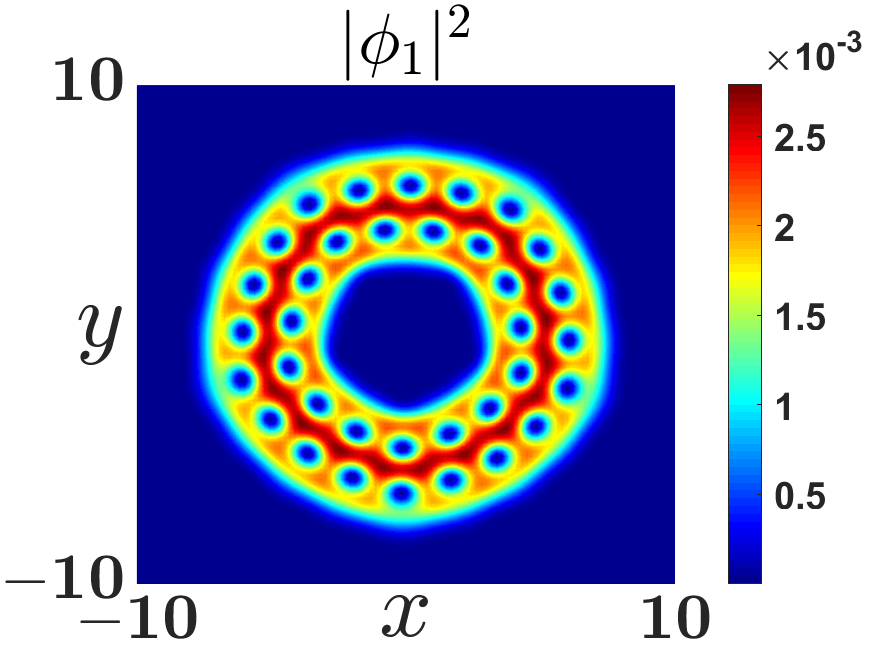}
 	\hspace{-0.22cm}
 	\includegraphics[scale=0.35]{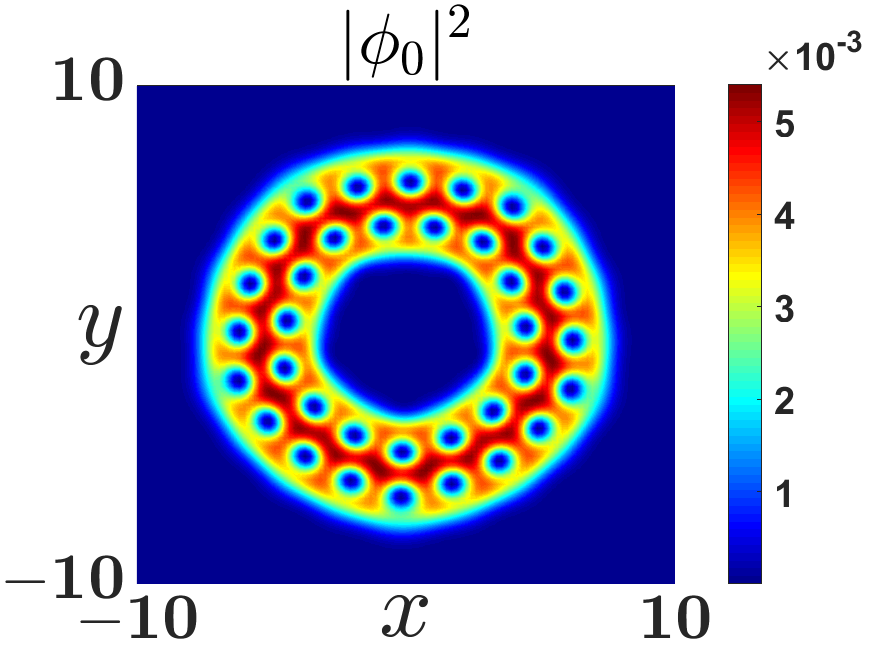}
 	\hspace{-0.22cm}
 	\includegraphics[scale=0.35]{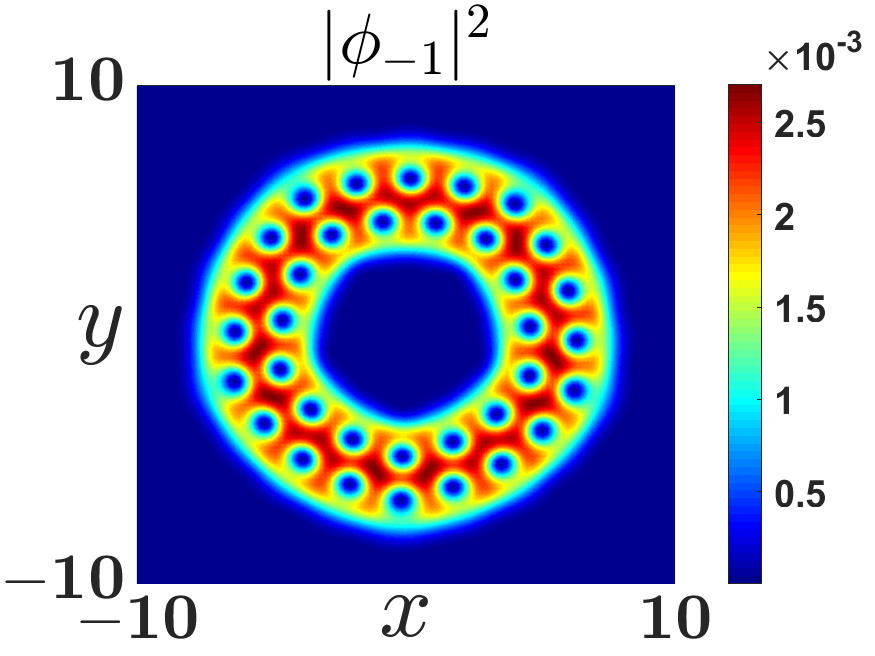}
 	\\
 	\caption{Contour plots of the densities for \textbf{Case 1} (from top to bottom, $\Omega = 0, 0.45, 0.90$) in \textbf{Example \ref{Effect-Rot}} in harmonic potential to investigate the effect of rotation.}
 	\label{fig_rot}
 \end{figure}
 
 \begin{figure}[h]
 	\centering
 	\includegraphics[scale=0.35]{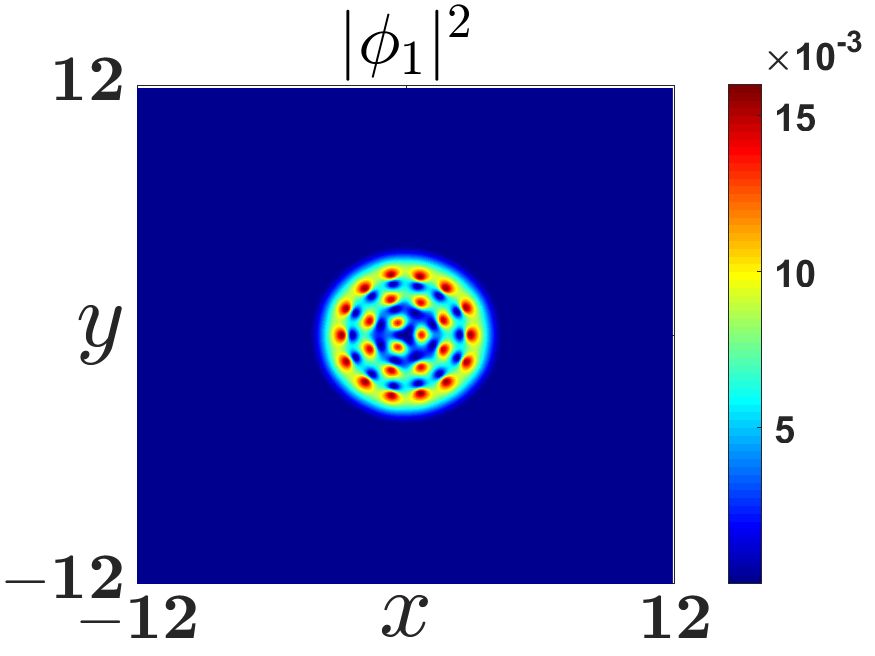}
 	\hspace{-0.22cm}
 	\includegraphics[scale=0.35]{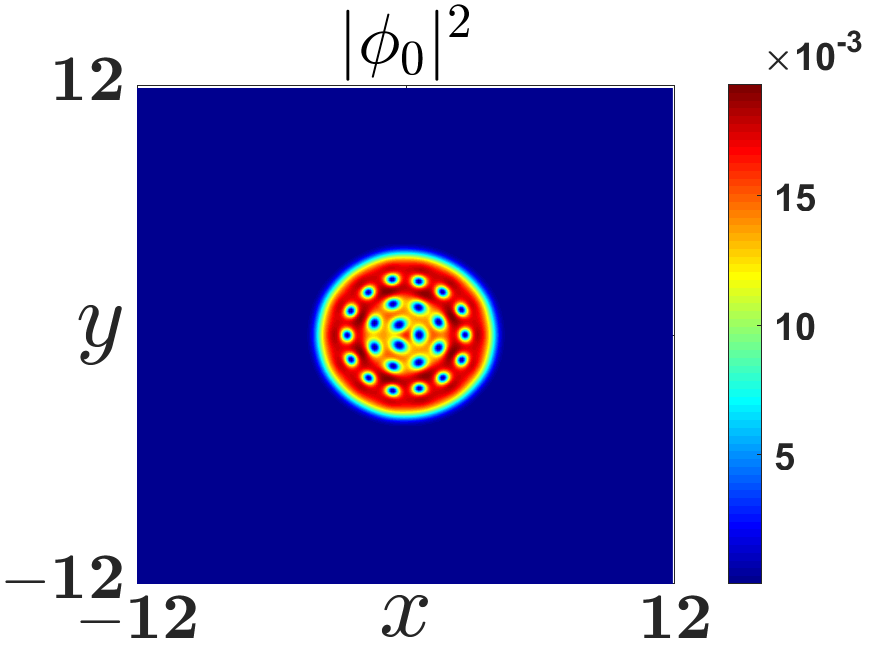}
 	\hspace{-0.22cm}
 	\includegraphics[scale=0.35]{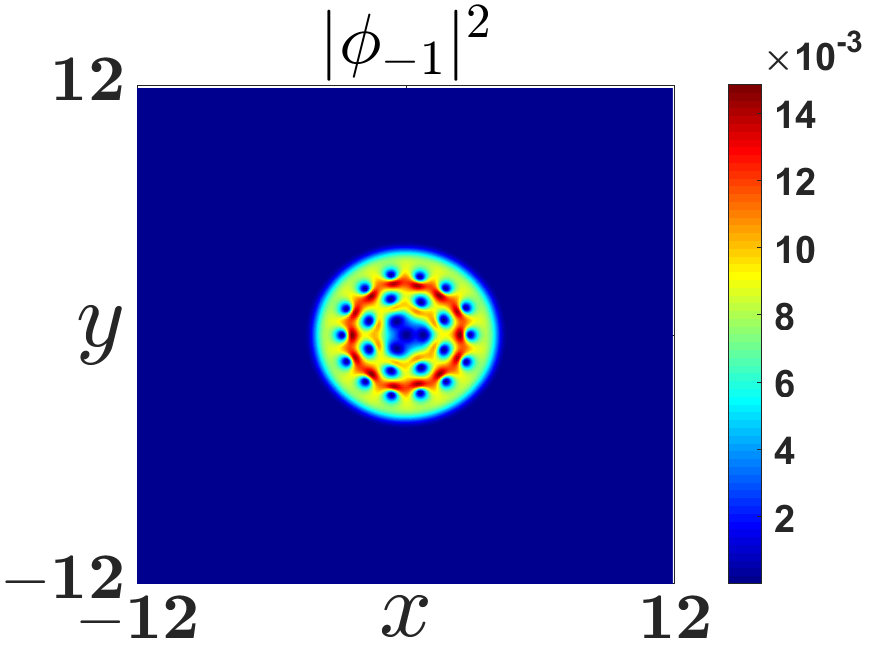}
 	\\
 	\centering
 	\includegraphics[scale=0.35]{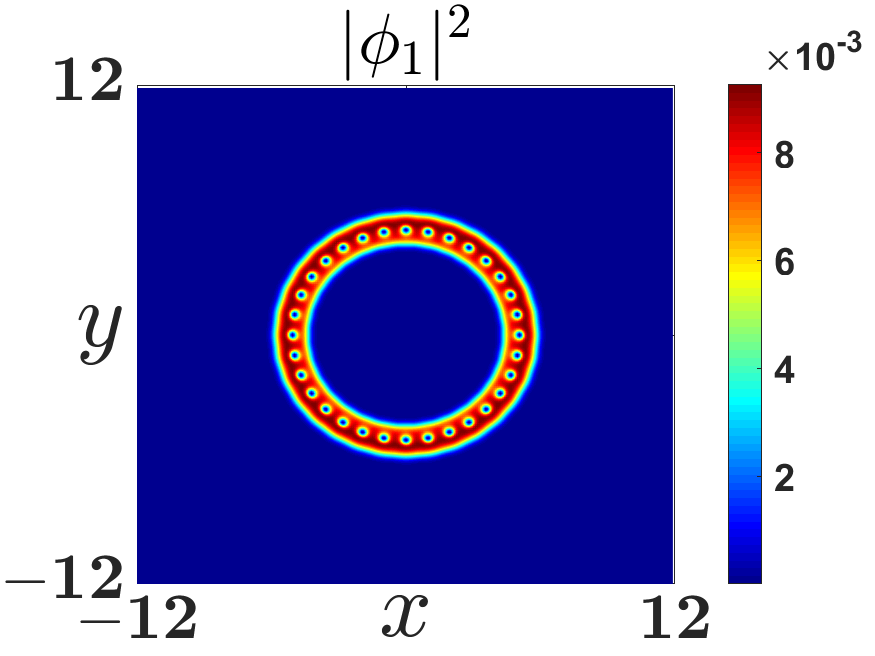}
 	\hspace{-0.22cm}
 	\includegraphics[scale=0.35]{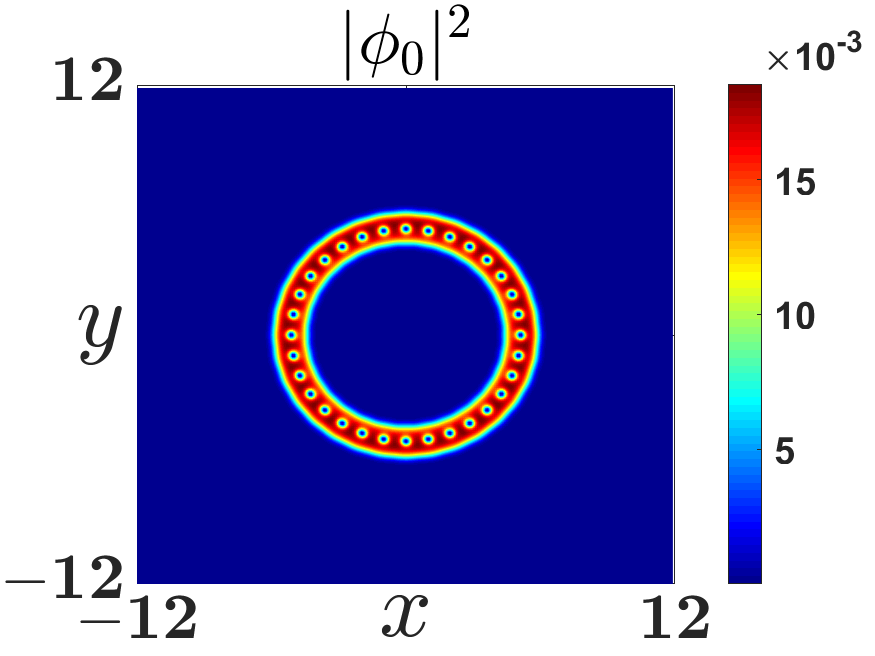}
 	\hspace{-0.22cm}
 	\includegraphics[scale=0.35]{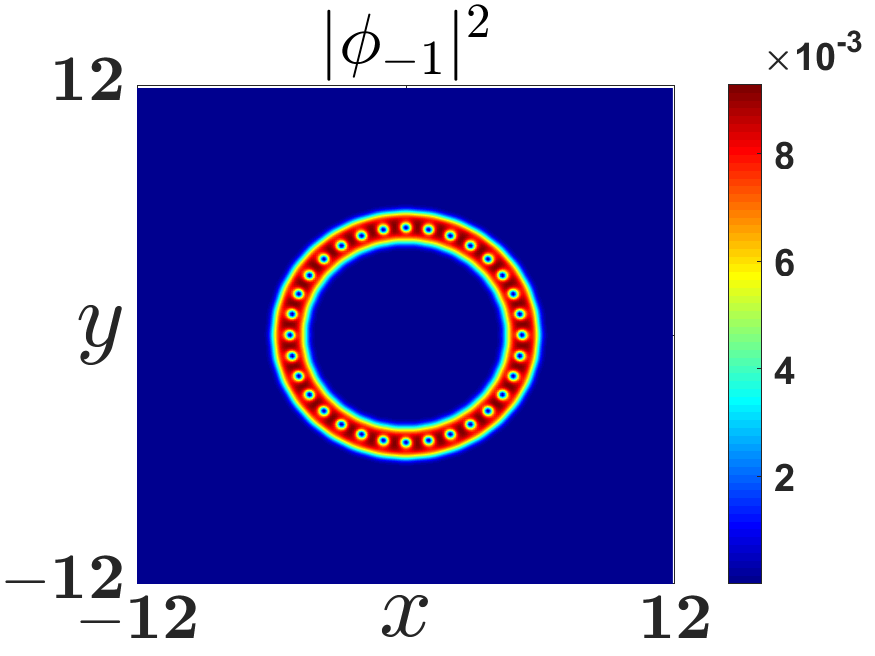}
 	\\
 	\centering
 	\includegraphics[scale=0.35]{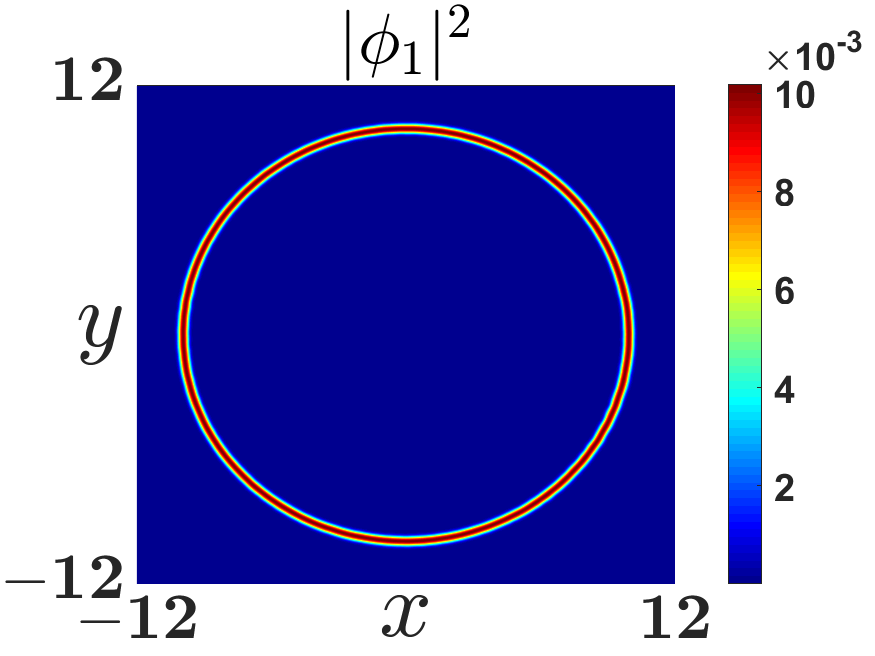}
 	\hspace{-0.22cm}
 	\includegraphics[scale=0.35]{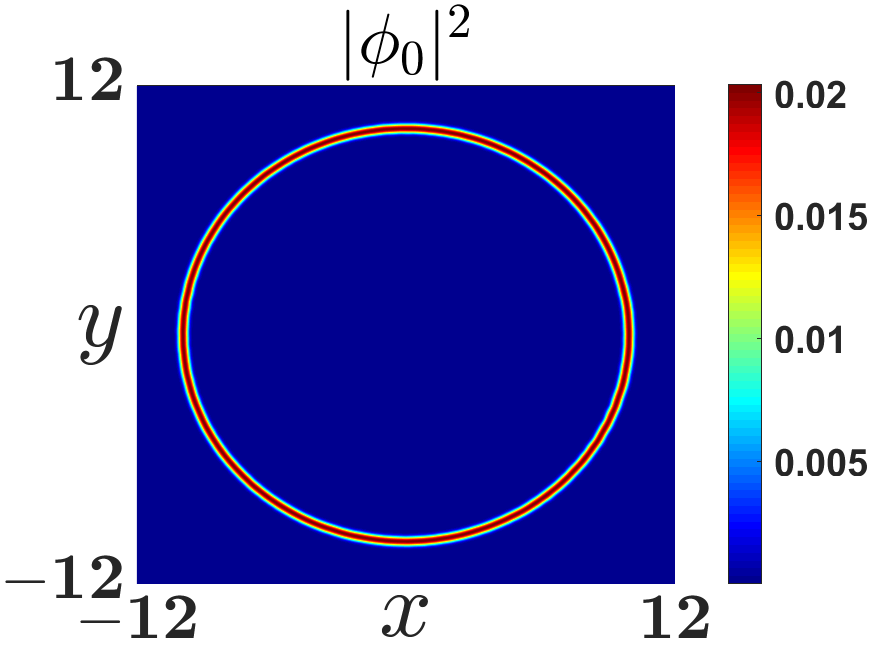}
 	\hspace{-0.22cm}
 	\includegraphics[scale=0.35]{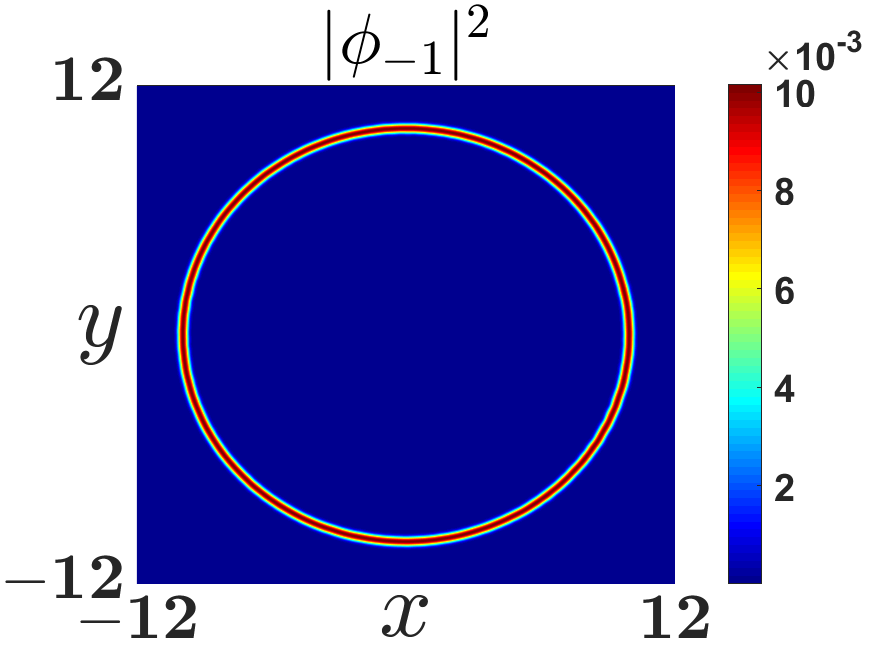}
 	\\
 	\caption{Contour plots of the densities for \textbf{Case 2} (from top to bottom, $\Omega = 3.0, 6.0, 14.0$) in \textbf{Example \ref{Effect-Rot}} in harmonic-plus-quartic potential to investigate the effect of rotation.}
 	\label{fig_rot2}
 \end{figure}
 
 For \textbf{Case 1}, we set $\mathcal{D} = [-12,12]^2$ with $N = 256$; for \textbf{Case 2}, we take $\mathcal{D} = [-20,20]^2$ with $N = 2400$. Fig.~\ref{fig_rot} and Fig.~\ref{fig_rot2} present the densities $\left(\vert \phi_1^{\rm g} \vert^2,\vert \phi_0^{\rm g} \vert^2,\vert \phi_{-1}^{\rm g} \vert^2\right)$ for \textbf{Case 1} and \textbf{Case 2}, respectively. 
 From Fig.~\ref{fig_rot}, we observe that in the harmonic potential, no vortex exists in the condensate for small rotating frequency $\Omega$ and the number of vortices increases as $\Omega$ turns larger. When $\Omega$ is large enough, a giant hole emerges at the center of the condensate \cite{rotSOCvortex}.
 
 While, in the harmonic-plus-quartic potential as shown in Fig.~\ref{fig_rot2}, 
 it is evident that many vortices exist in the condensate due to the large centrifugal force induced by the rotation. As $\Omega$ increases, the structure of the condensate becomes a perfect annulus, with uniformly distributed vortices. When $\Omega$ is increased further, a giant vortex emerges within the condensate.

 \begin{exmp}\label{Effect-SOC}
 	Here we investigate the effect of SOC on the ground state. The parameters are set as: $\Omega = 0.2$, $c_0 = 100$, and $c_1 = -1$. 
 \end{exmp}
 
 \begin{figure}[h]
 	\centering
 	\includegraphics[scale=0.35]{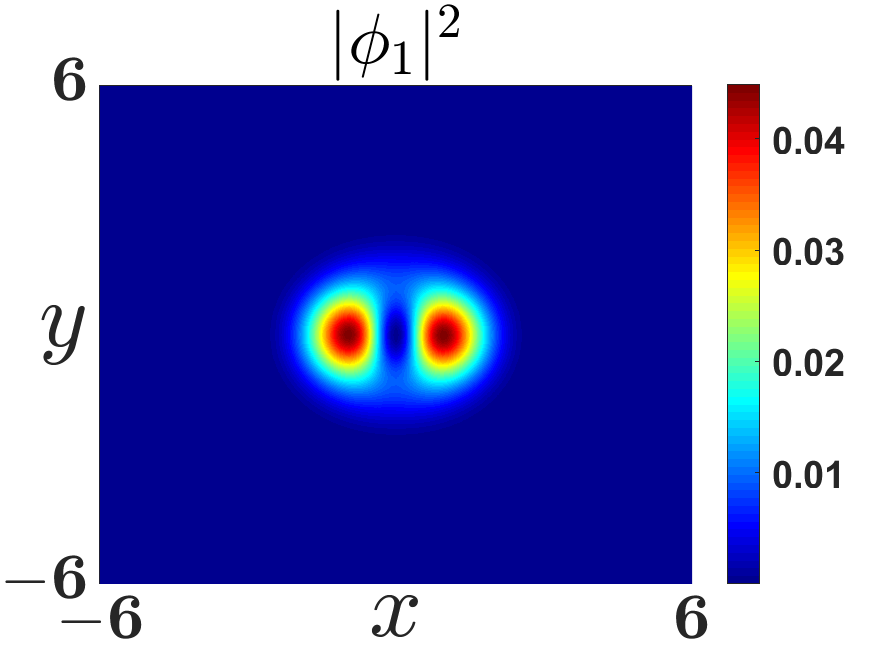}
 	\hspace{-0.22cm}
 	\includegraphics[scale=0.35]{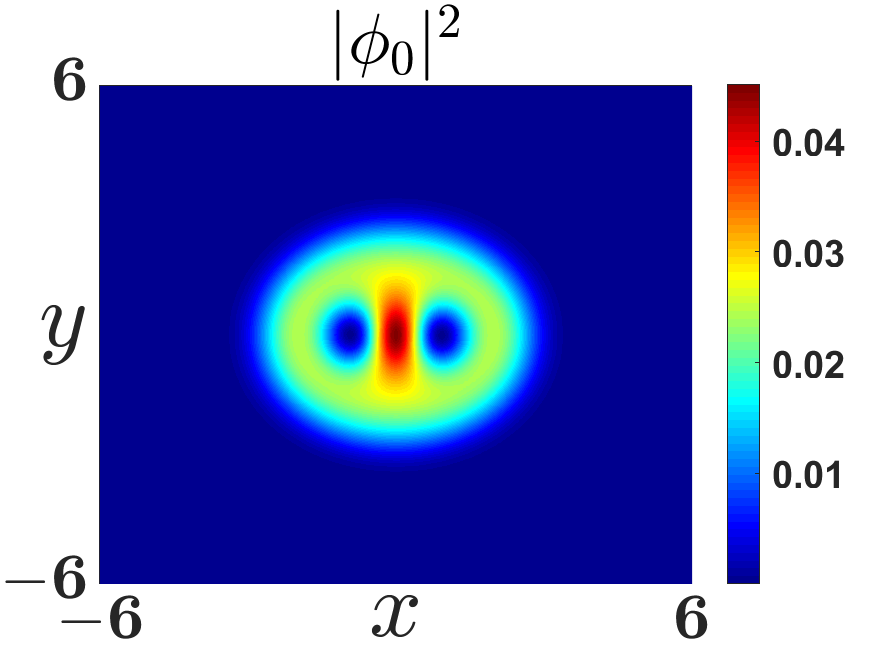}
 	\hspace{-0.22cm}
 	\includegraphics[scale=0.35]{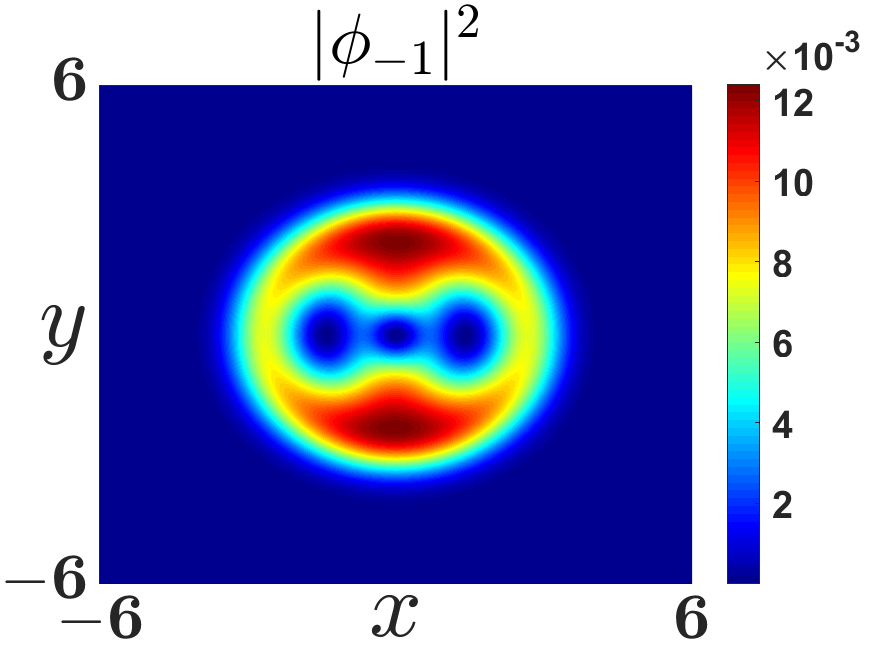}
 	\\
 	\centering
 	\includegraphics[scale=0.35]{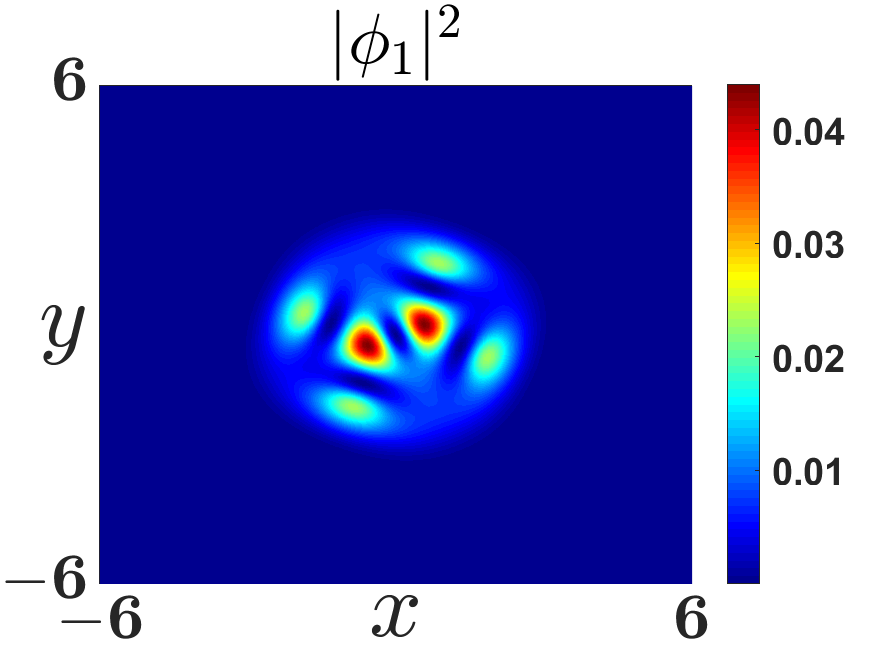}
 	\hspace{-0.22cm}
 	\includegraphics[scale=0.35]{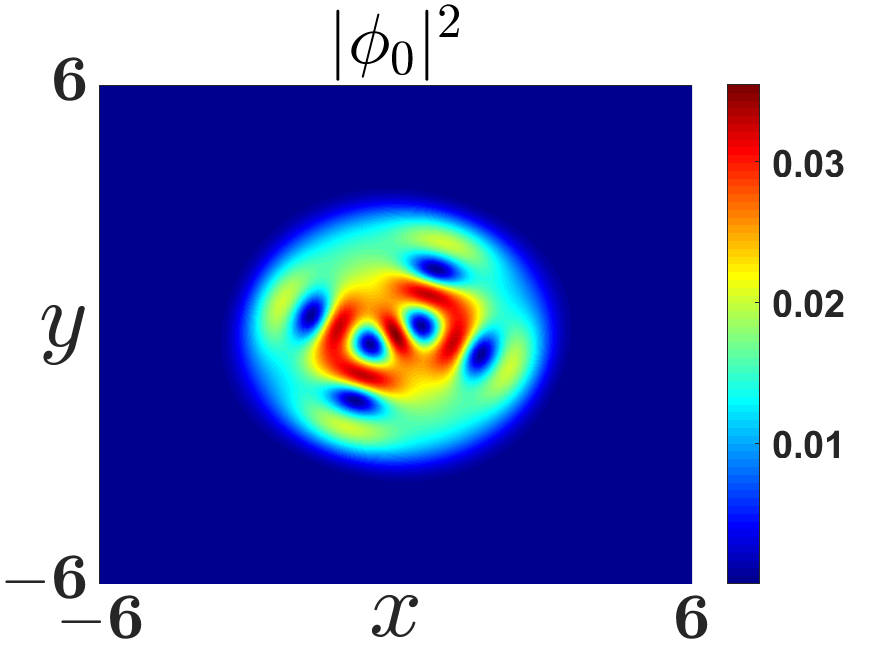}
 	\hspace{-0.22cm}
 	\includegraphics[scale=0.35]{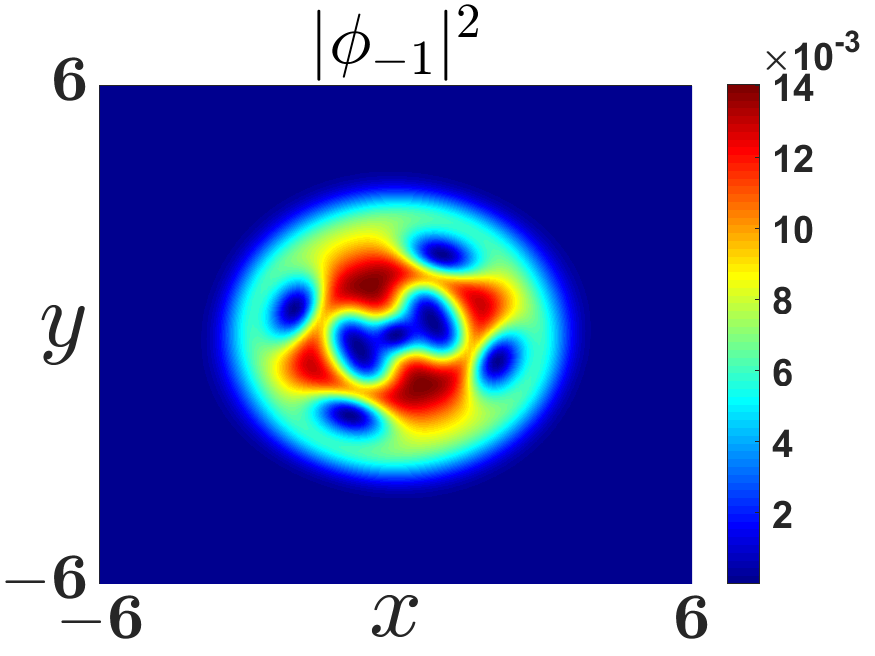}
 	\\
 	\centering
 	\includegraphics[scale=0.35]{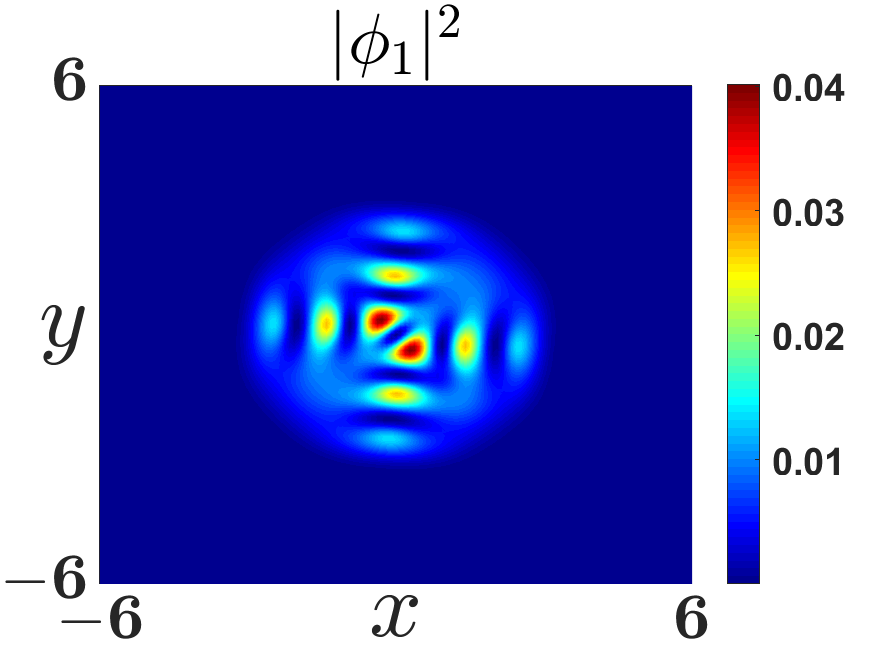}
 	\hspace{-0.22cm}
 	\includegraphics[scale=0.35]{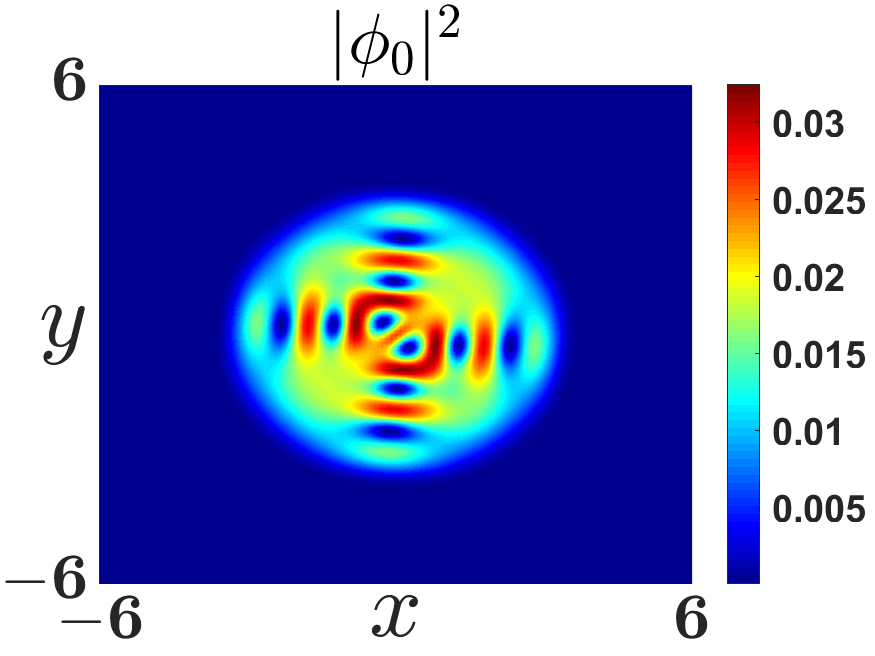}
 	\hspace{-0.22cm}
 	\includegraphics[scale=0.35]{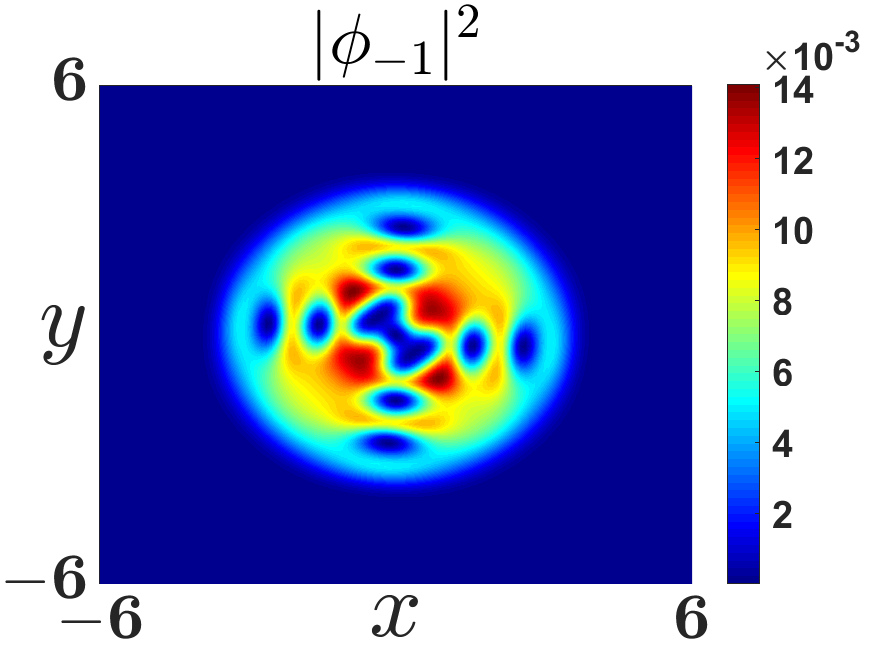}
 	\\
 	\caption{Contour plots of the densities for different SOC strength $\gamma$ (from top to bottom, $\gamma = 0.8, 1.6, 2.4$) in \textbf{Example \ref{Effect-SOC}} to investigate the effect of SOC.}
 	\label{fig_soc}
 \end{figure}
 
 We set $\mathcal{D} = [-12,12]^2$ with $N = 256$. Fig.~\ref{fig_soc} displays the density plots when $\gamma$ is equal to $0.8$, $1.6$, $2.4$, respectively. From Fig.~\ref{fig_soc}, we clearly observe that, as the SOC strength $\gamma$ increases, the number of vortices increases, and the shapes of vortices transition from circular to strip-shaped.

 \begin{exmp}\label{3d-result} 
 	In the example, we simulate some ground states in 3D isotropic/anisotropic harmonic potential. Here, we take $c_1 = 1$ and test the following six cases: 
 	\begin{flalign*}
 		&\textbf{ Case 1. } \gamma_x = \gamma_y =\gamma_z = 1, c_0 = 100, \Omega=0.6 \text{ and } \gamma=0.4. && \\ 
 		&\textbf{ Case 2. } \gamma_x = \gamma_y =\gamma_z = 1, c_0 = 100, \Omega=0.6 \text{ and } \gamma=0.6. && \\
 		&\textbf{ Case 3. } \gamma_x = \gamma_y =\gamma_z = 1, c_0 = 100, \Omega=0.8 \text{ and } \gamma=0.4. && \\
 		&\textbf{ Case 4. } \gamma_x = 0.97, \gamma_y = 1.03, \gamma_z = 0.05, c_0=3000, \Omega=0.5 \text{ and } \gamma=0.4. && \\ 
 		&\textbf{ Case 5. } \gamma_x = 0.97, \gamma_y = 1.03, \gamma_z = 0.05, c_0=3000, \Omega=0.5 \text{ and } \gamma=0.8. && \\
 		&\textbf{ Case 6. } \gamma_x = 0.97, \gamma_y = 1.03, \gamma_z = 0.05, c_0=3000, \Omega=0.7 \text{ and } \gamma=0.4. &&
 	\end{flalign*}
 \end{exmp}
 
 \begin{figure}[h!]
 	\centering
 	\includegraphics[scale=0.35]{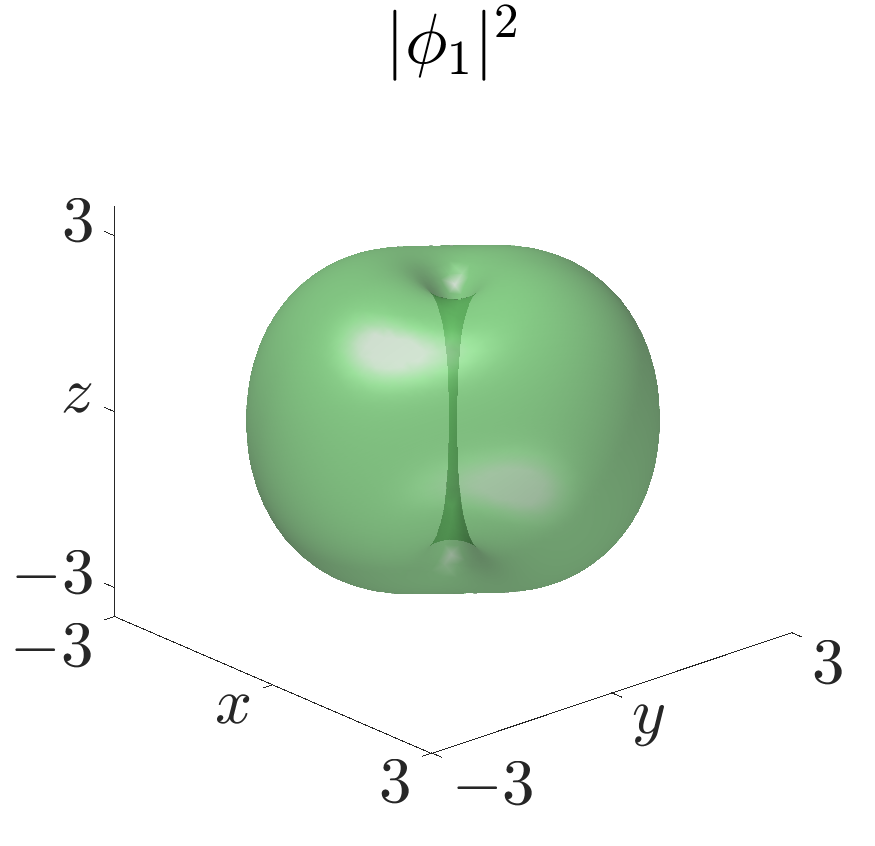}
 	\hspace{-0.22cm}
 	\includegraphics[scale=0.35]{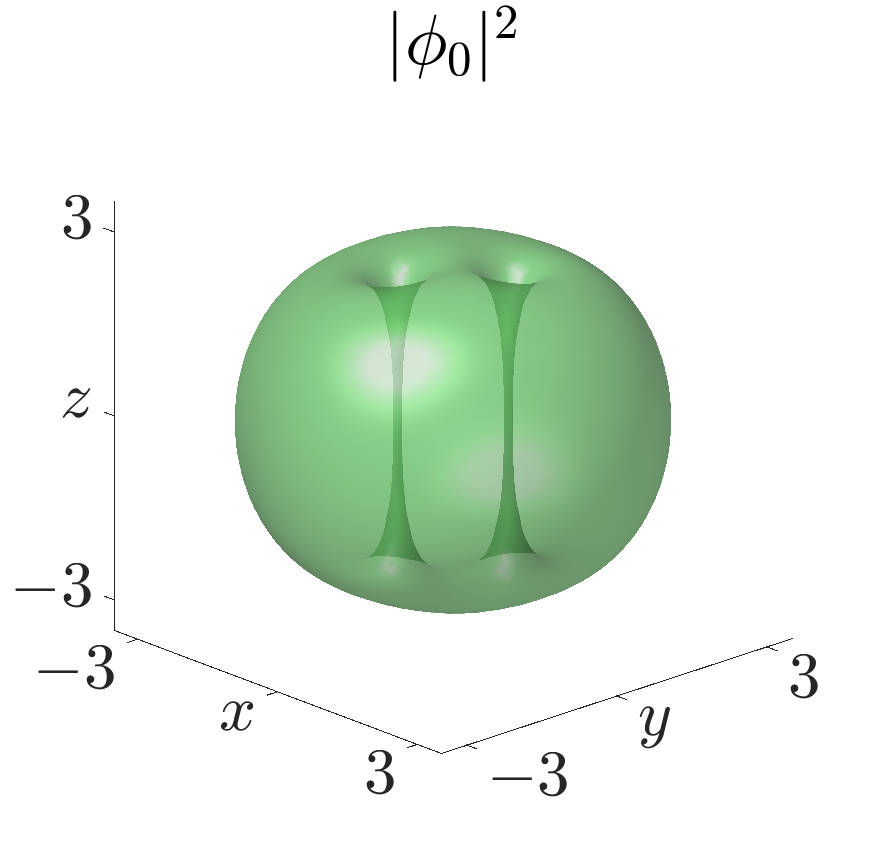}
 	\hspace{-0.22cm}
 	\includegraphics[scale=0.35]{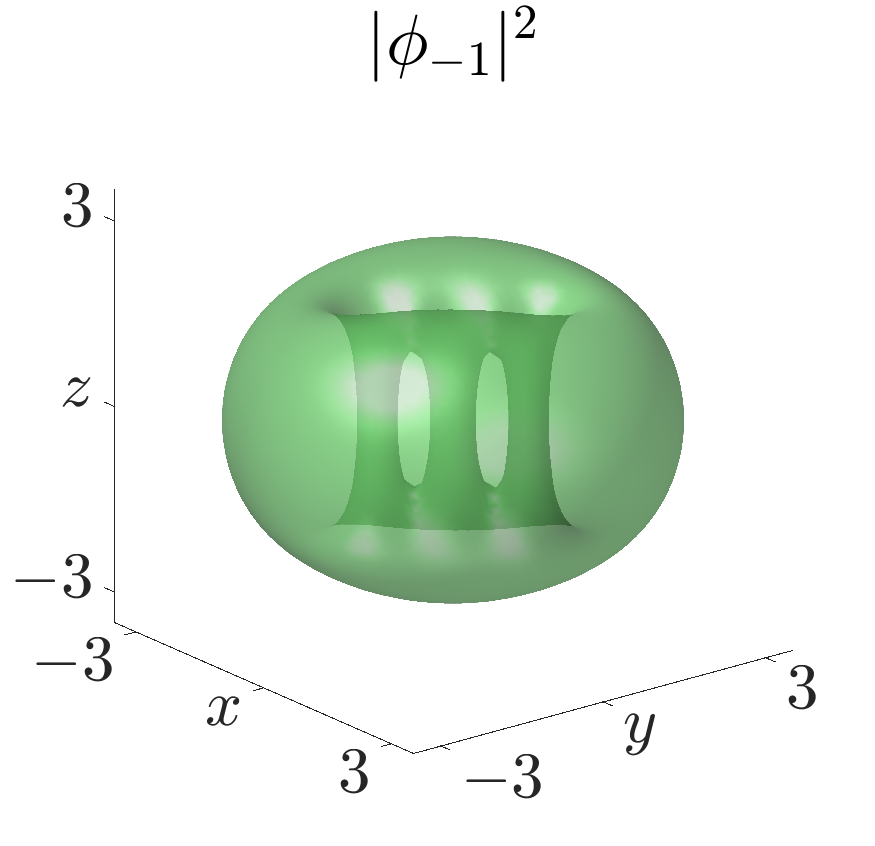}
 	\\
 	\centering
 	\includegraphics[scale=0.35]{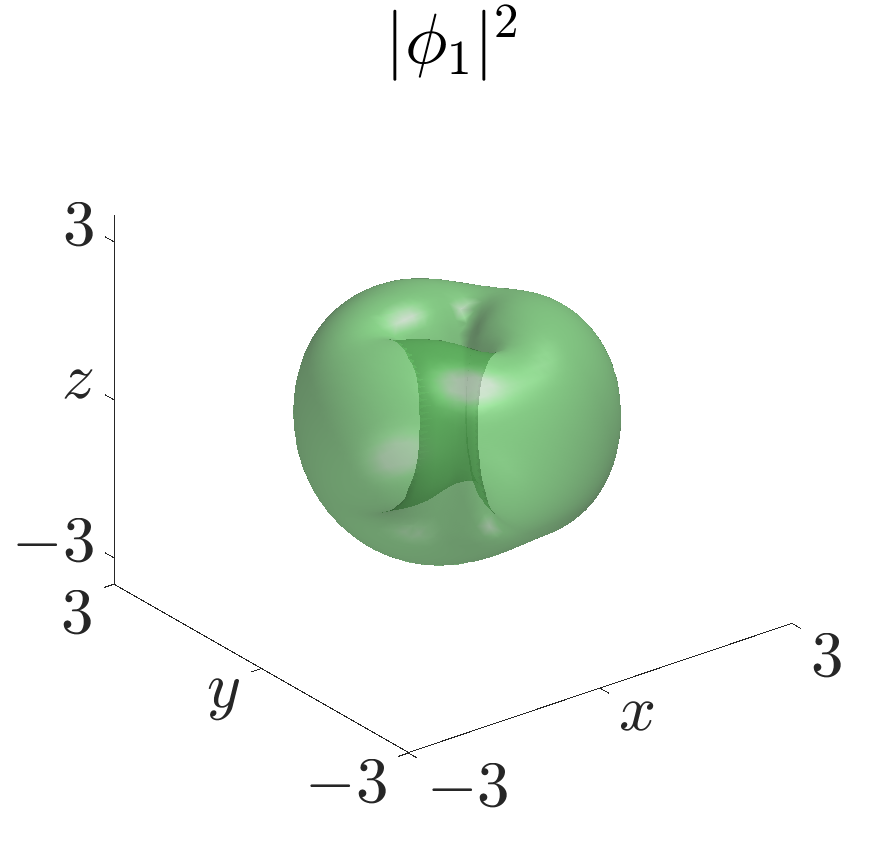}
 	\hspace{-0.22cm}
 	\includegraphics[scale=0.35]{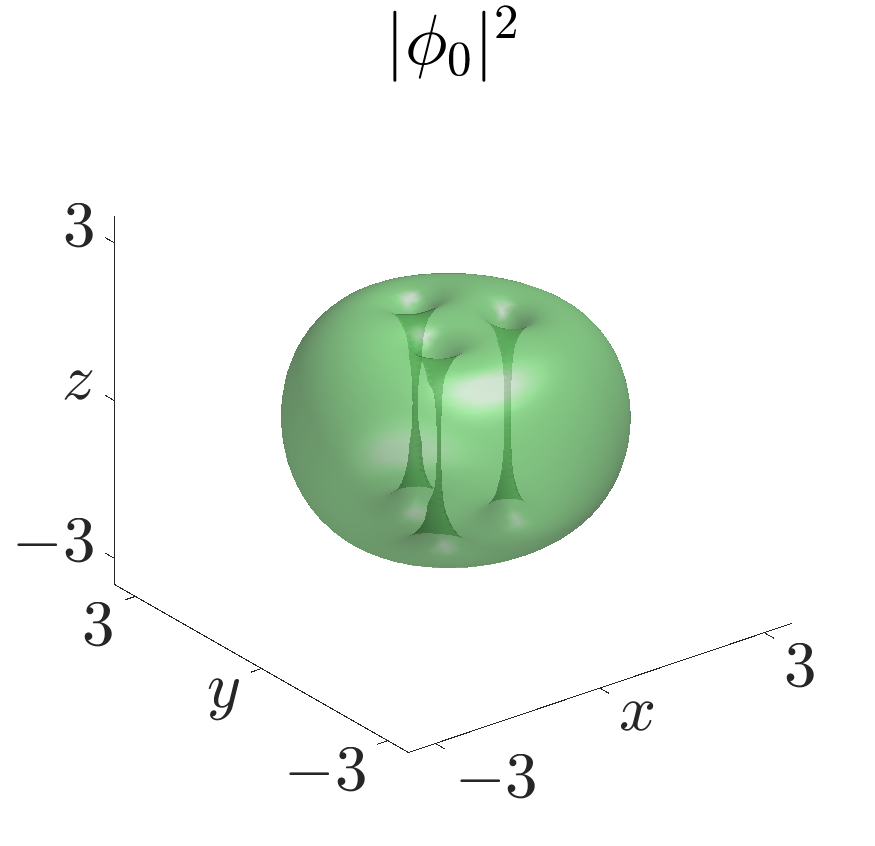}
 	\hspace{-0.22cm}
 	\includegraphics[scale=0.35]{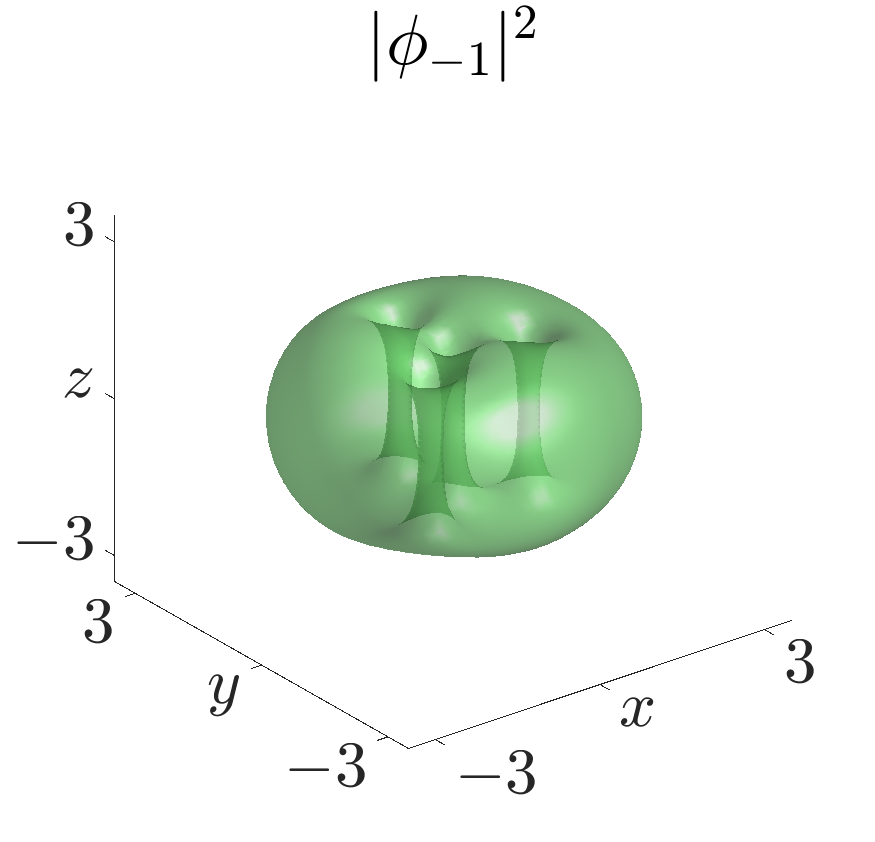}
 	\\
 	\centering
 	\includegraphics[scale=0.35]{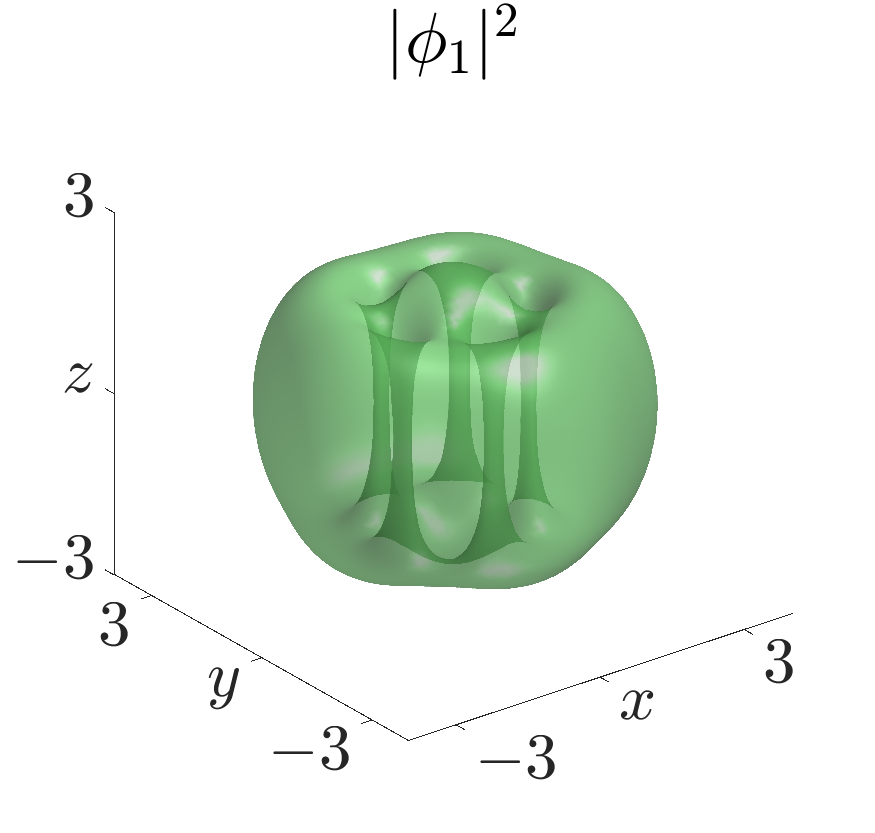}
 	\hspace{-0.22cm}
 	\includegraphics[scale=0.35]{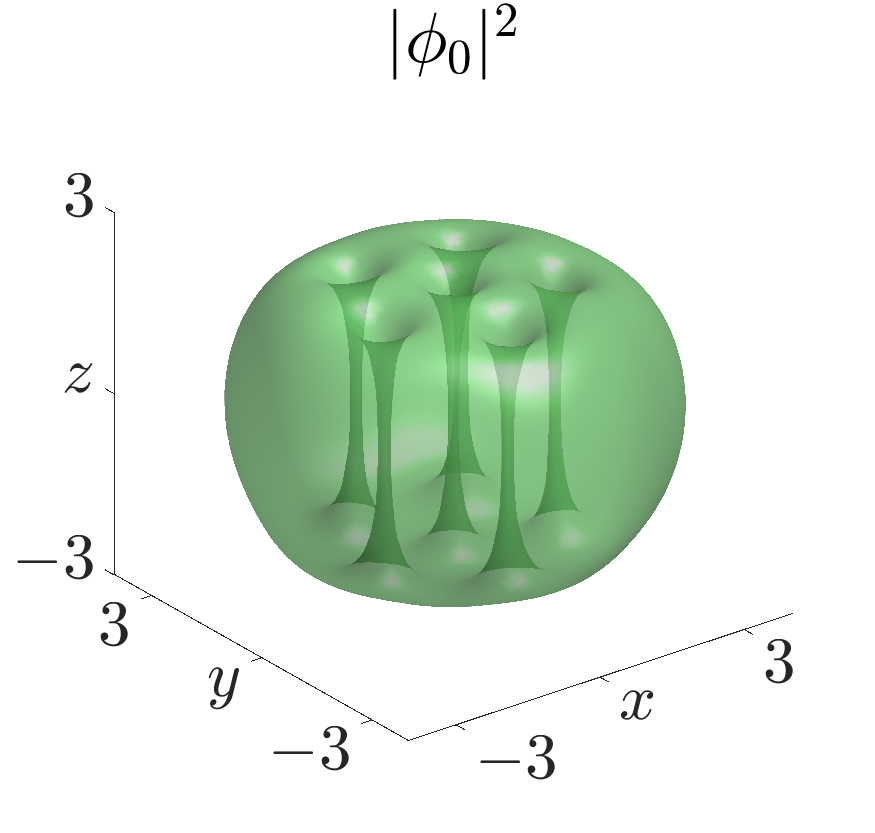}
 	\hspace{-0.22cm}
 	\includegraphics[scale=0.35]{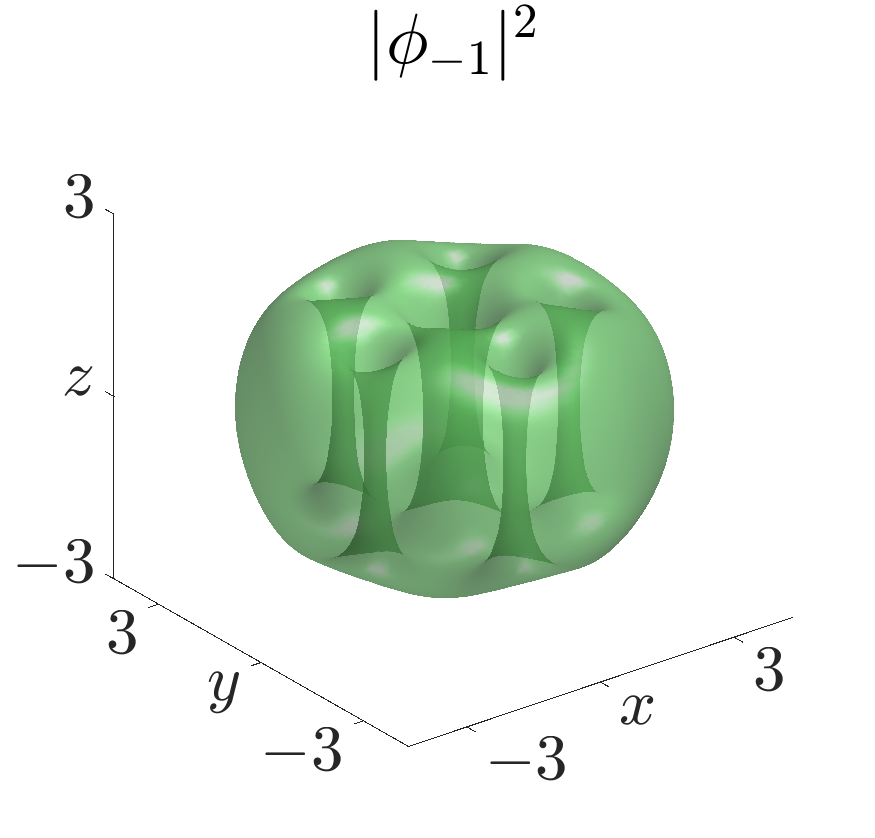}
 	\\
 	\caption{The isosurfaces $\vert\phi_{\ell}^{\rm g}\vert^2 = 10^{-4}$ for \textbf{Case 1, 2, 3} (top to bottom) in \textbf{Example \ref{3d-result}} in 3D isotropic harmonic potential.}
 	\label{fig_isosurface1}
 \end{figure}
 
 \begin{figure}[h!]
 	\centering
 	\includegraphics[scale=0.35]{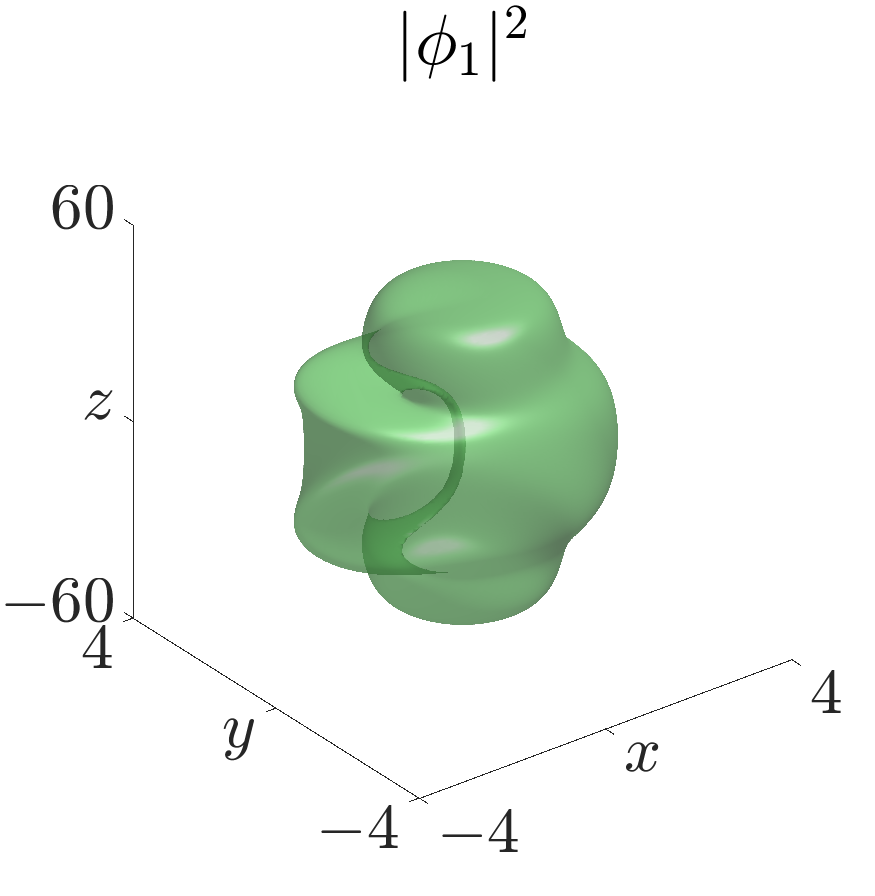}
 	\hspace{-0.22cm}
 	\includegraphics[scale=0.35]{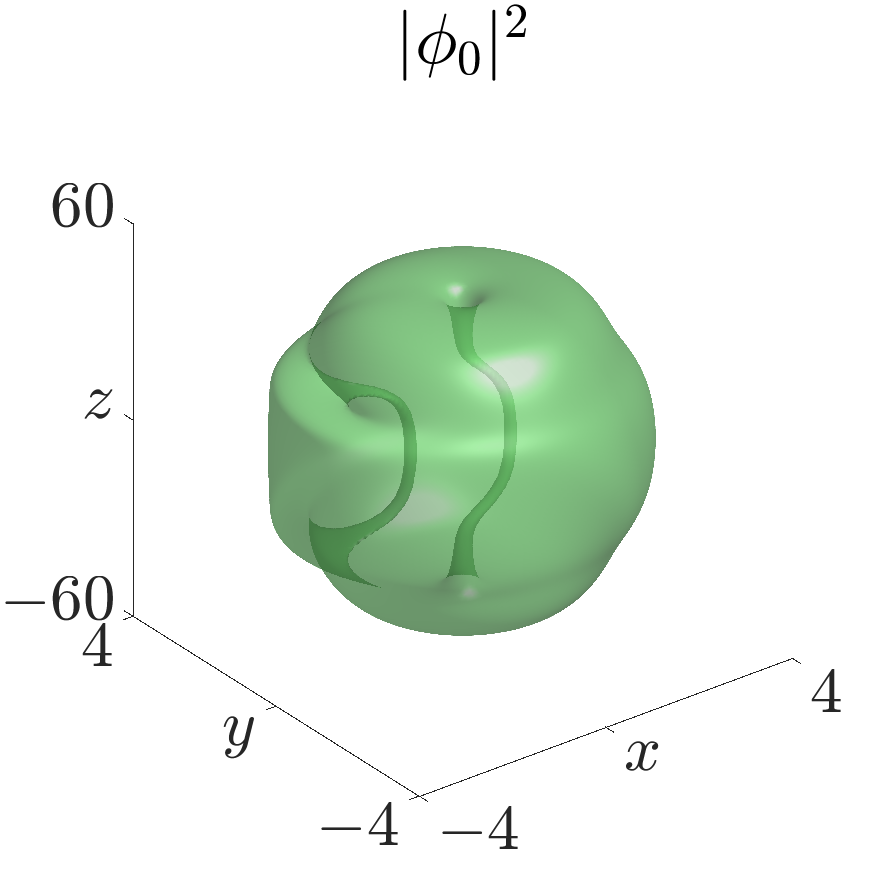}
 	\hspace{-0.22cm}
 	\includegraphics[scale=0.35]{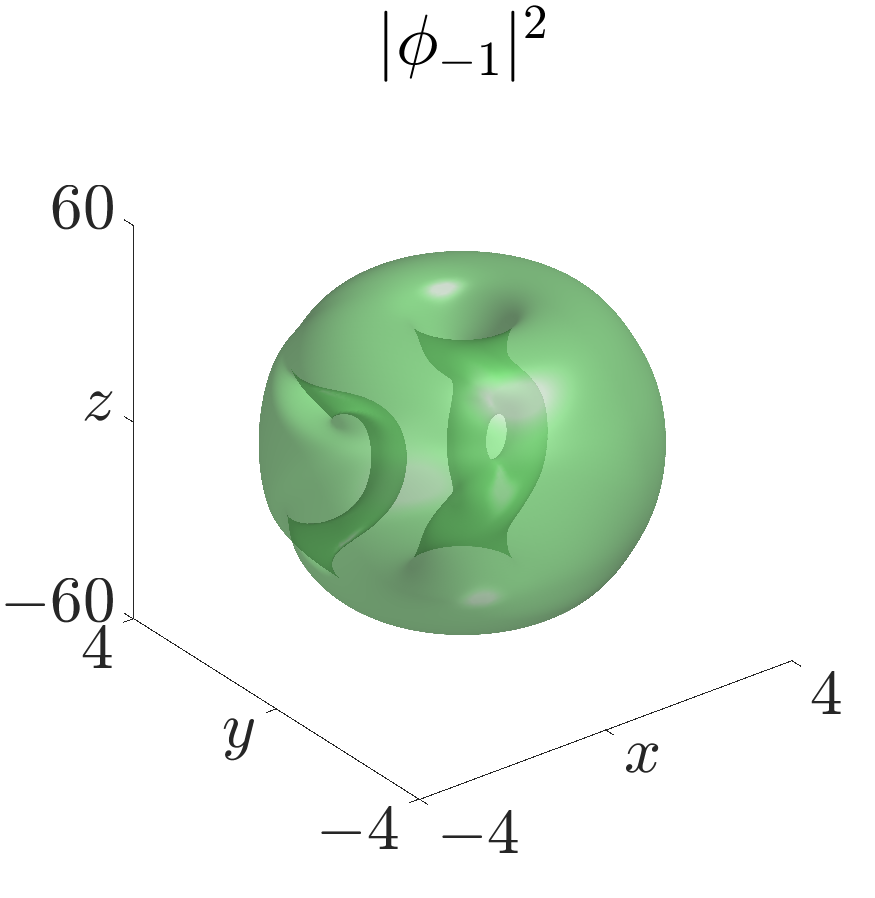}
 	\\
 	\centering
 	\includegraphics[scale=0.35]{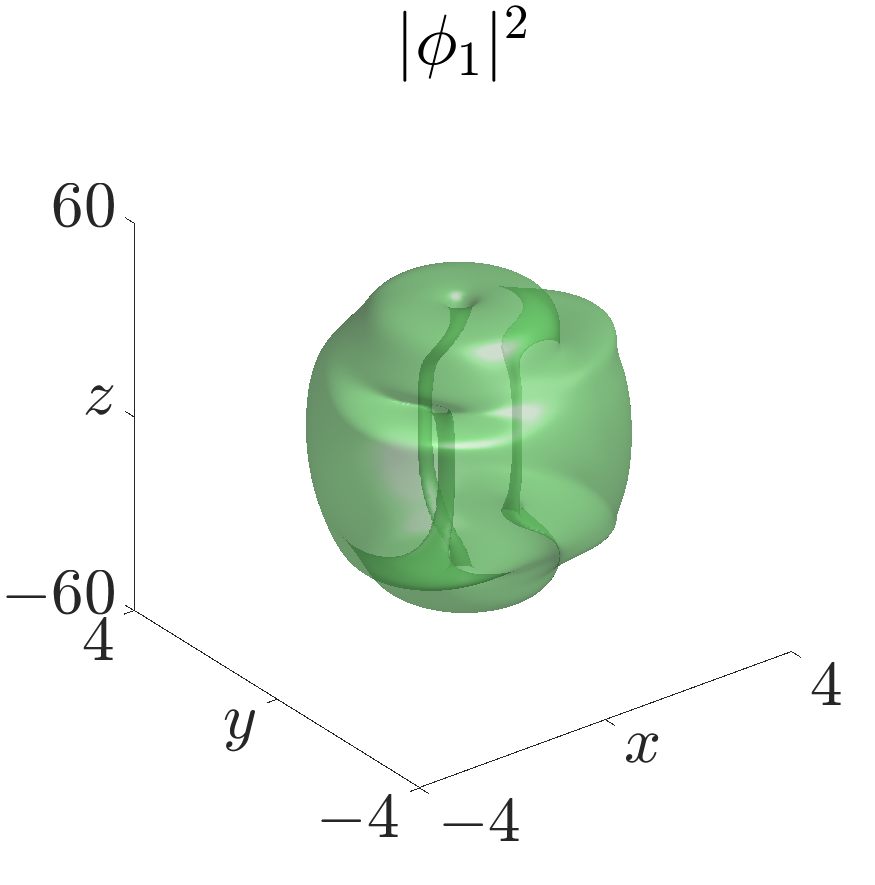}
 	\hspace{-0.22cm}
 	\includegraphics[scale=0.35]{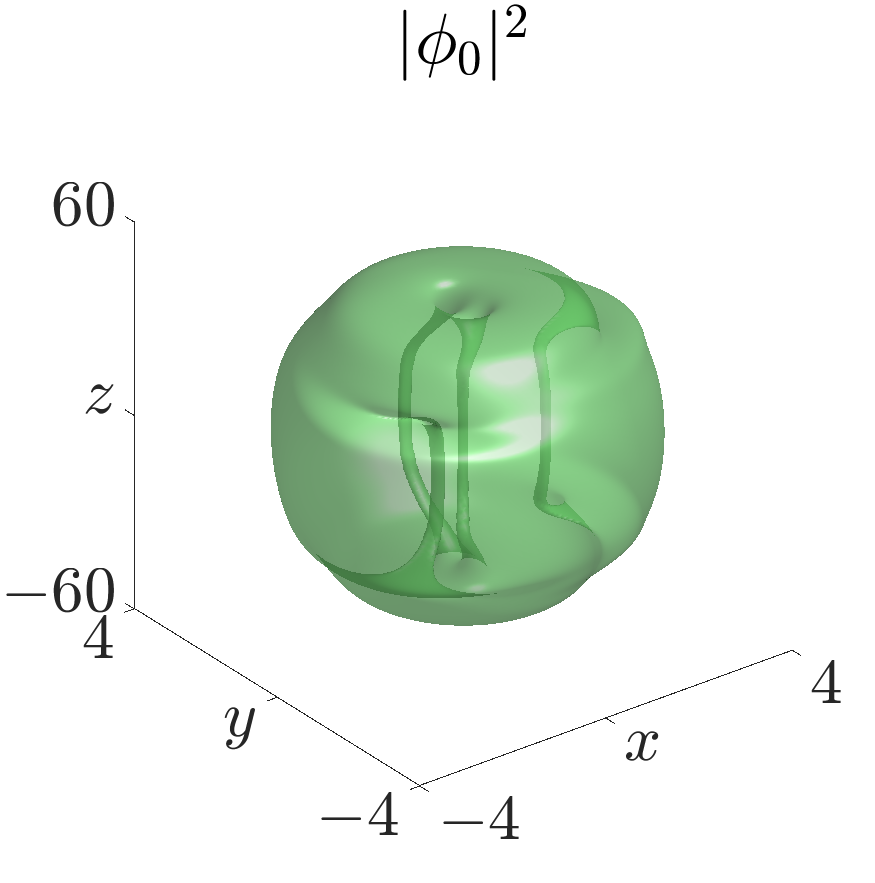}
 	\hspace{-0.22cm}
 	\includegraphics[scale=0.35]{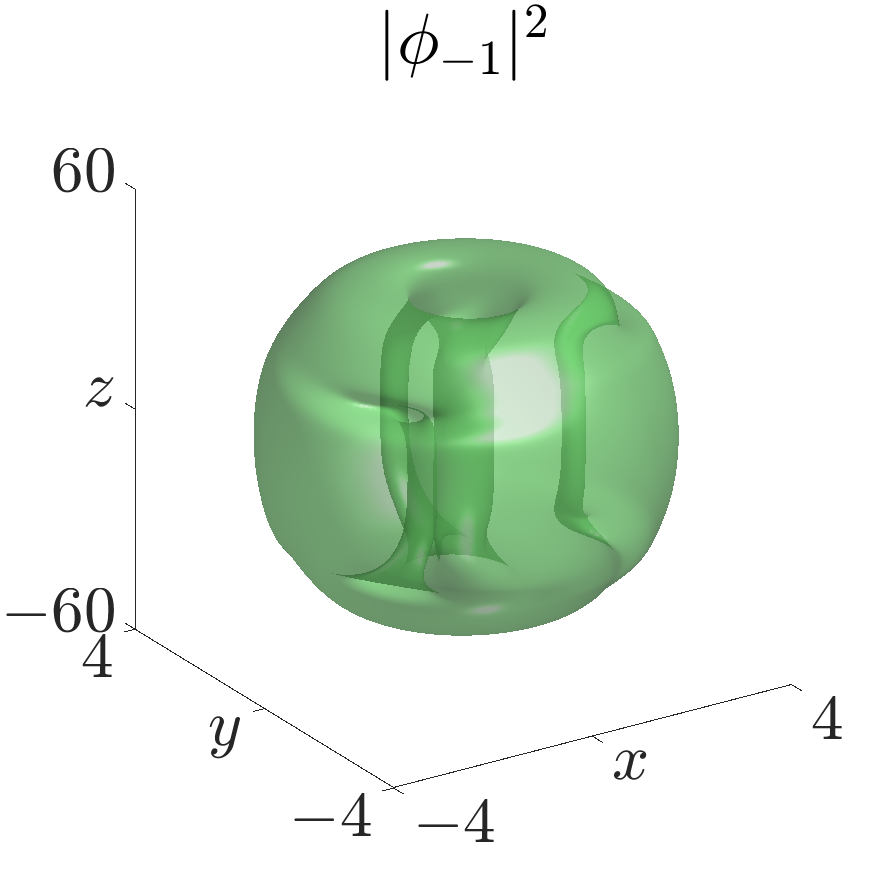}
 	\\
 	\centering
 	\includegraphics[scale=0.35]{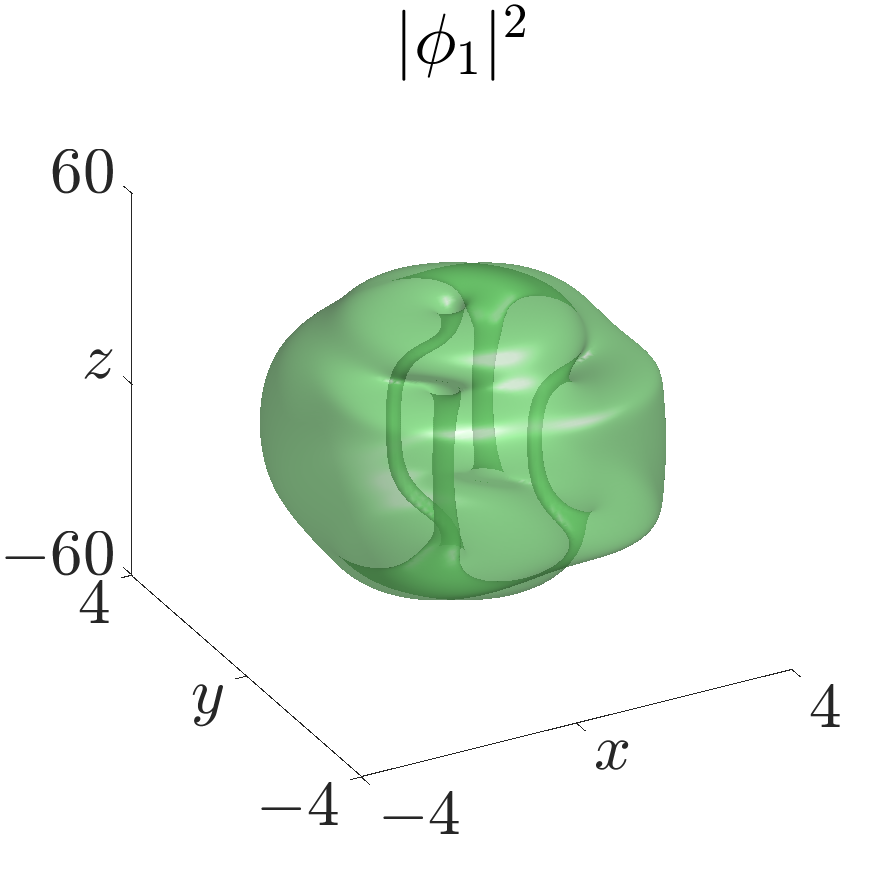}
 	\hspace{-0.22cm}
 	\includegraphics[scale=0.35]{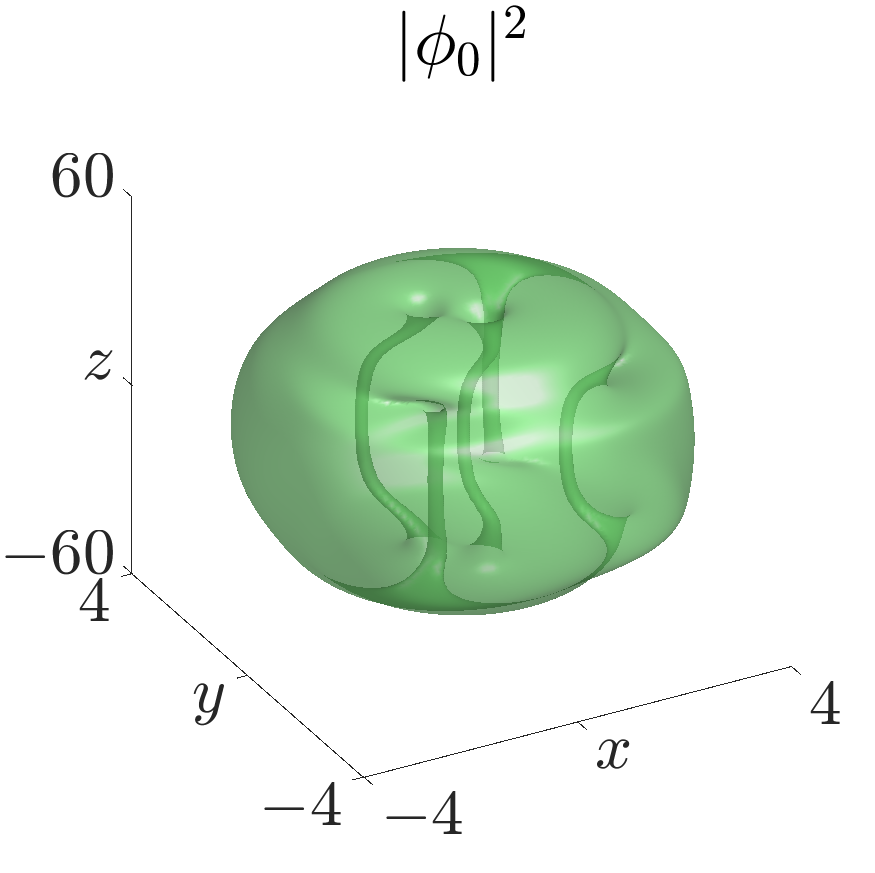}
 	\hspace{-0.22cm}
 	\includegraphics[scale=0.35]{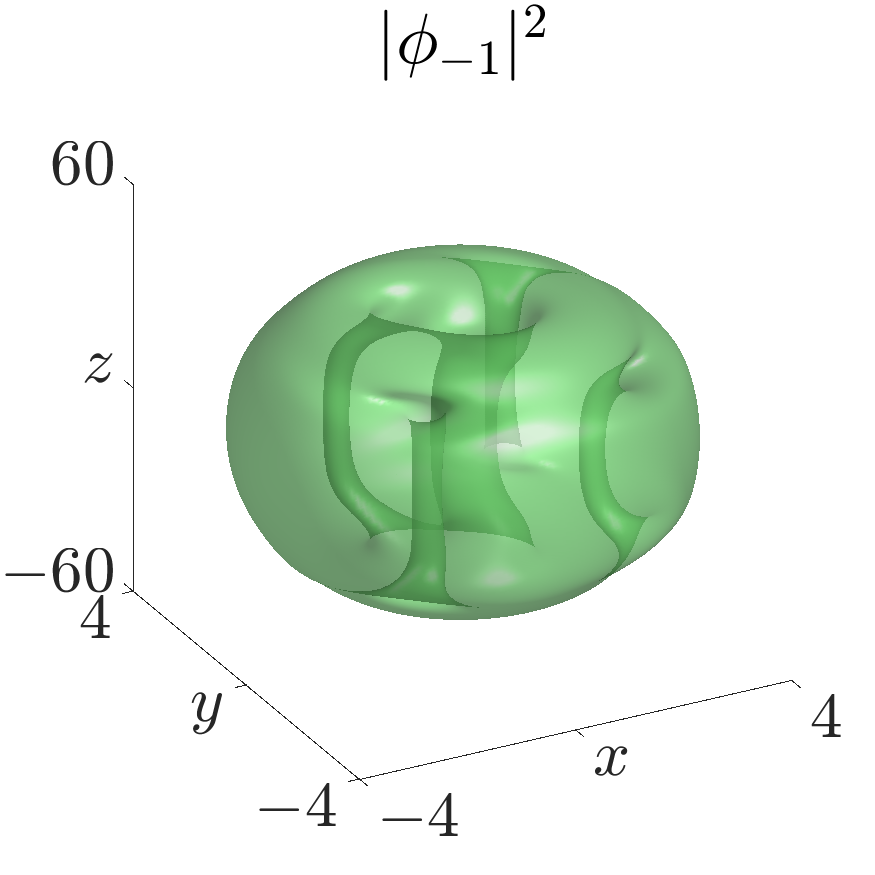}
 	\caption{The isosurfaces $\vert\phi_{\ell}^{\rm g}\vert^2 = 2 \times 10^{-5}$ for \textbf{Case 4, 5, 6} (top to bottom) in \textbf{Example \ref{3d-result}} in 3D anisotropic harmonic potential.}
 	\label{fig_isosurface2}
 \end{figure}
 
 For \textbf{Case 1-3}, the computational domain $\mathcal{D} = [-10,10]^3$, and for \textbf{Case 4-6}, $\mathcal{D} = [-10,10]^2 \times [-80,80]$. We apply Algorithm \ref{algorithm3} with the coarsest grid size $N = 64$ and the finest one $N = 256$.
 Fig.~\ref{fig_isosurface1} and Fig.~\ref{fig_isosurface2} depict the isosurfaces $\vert\phi_{\ell}^{\rm g}\vert^2 = 1 \times 10^{-4}$ $(\ell = 1, 0, -1)$ for the \textbf{Case 1-3} in isotropic harmonic potential and $\vert\phi_{\ell}^{\rm g}\vert^2 = 2 \times 10^{-5}$ $(\ell = 1, 0, -1)$ for the \textbf{Case 4-6} in anisotropic harmonic potential, respectively. From the two figures, we conclude that the number of vortex lines rises as the rotational speed $\Omega$ or SOC strength $\gamma$ increases. Furthermore, the vortex lines remain parallel to the axis of rotation in the isotropic harmonic potential. However, when $\gamma_x/\gamma_z$ and $\gamma_y/\gamma_z$ turn large enough in the anisotropic harmonic potential, the vortex lines are bent and will possess a U shape.

 \section{Conclusion}\label{conclude}
 In this paper, we first study properties of the ground state for the rotating spin-orbit coupled spin-1 BEC, including the existence, virial identity, and negativity of the spin-orbit coupling energy. Then, we propose an efficient and accurate preconditioned nonlinear conjugate gradient method to compute the ground states. 
 Our algorithm is spectrally accurate and performs with great efficiency, especially for the fast-rotating or strongly repulsive interaction, by virtue of the adaptive step size control mechanism and the well-designed preconditioners. The efficiency is further enhanced by cascadic multigrid strategy. Finally, through extensive numerical experiments, we validate the spatial spectral accuracy by traversing the most commonly-used initial guesses for each component, and investigate the effects of local interactions, rotation, spin-orbit coupling, and external trapping potential on the ground states. Furthermore, we unveil some interesting structures of the ground states, such as giant vortex and U-shape vortex line, which, we believe, may inspire more interesting follow-up research such as the ground state patterns.
 
 \section*{Acknowledgements}
This work was partially supported by the National Key R\&D Program of China No. 2024YFA1012803 and basic research fund of Tianjin University under grant 2025XJ21-0010 (W. Yang and Y. Zhang), by the National Natural Science Foundation of China No. 12471375, 11971007 (Y. Yuan) and No. 12271400 (Y. Zhang).

 \end{document}

%% file: command.tex
%%%%%%%%%%%%%%%%% author macros %%%%w%%%%%%%%%
\newcommand{\be}{\begin{equation}}
\newcommand{\ee}{\end{equation}}
\newcommand{\ba}{\begin{array}}
\newcommand{\ea}{\end{array}}
\newcommand{\bea}{\begin{eqnarray}}
\newcommand{\eea}{\end{eqnarray}}
\newcommand{\beas}{\begin{eqnarray*}}
\newcommand{\eeas}{\end{eqnarray*}}

\newcommand{\fl}[2]{\frac{#1}{#2}}

\newcommand{\bx}{{\mathbf x}}
\newcommand{\by}{{\mathbf y}}
\newcommand{\bz}{{\mathbf z}}
\newcommand{\bn}{{\mathbf n}}
\newcommand{\br}{{\mathbf r}}
\newcommand{\bp}{{\mathbf p}}
\newcommand{\im}{\textrm i}

\newcommand{\bu}{{\textbf u}}
\newcommand{\bv}{{\textbf v}}
\newcommand{\bbf}{{\textbf f}}
\newcommand{\bbg}{{\textbf g}}
\newcommand{\bh}{{\textbf h}}

\newcommand{\Th}{{\mathcal T}_{\bf h}}

\newcommand{\dif}{{\textrm d}}

\newcommand{\gm}{{\gamma}}

\newcommand\bl{{\textbf B}_\textbf{L}}
\newcommand{\bk}{{\textbf k}}
\newcommand{\bR}{{\textbf R}}
\newcommand{\hh}{\widetilde{h}}
\renewcommand{\theequation}{\arabic{section}.\arabic{equation}} %
\newtheorem{exmp}{Example}
\newtheorem{remark}{Remark}[section]
\newtheorem{prop}{Proposition}
\newtheorem{thm}{Theorem}
\newtheorem{lem}{Lemma}

\DeclareRobustCommand{\rchi}{{\mathpalette\irchi\relax}}
\newcommand{\irchi}[2]{\raisebox{\depth}{$#1\chi$}} % inner command, used by \rchi